\documentclass[a4paper,11pt]{article}
\usepackage[english]{babel}
\usepackage[utf8]{inputenc}
\usepackage{amsthm,amsmath,amssymb}
\usepackage{bbold}
\usepackage[dvipsnames]{xcolor}
\usepackage{cancel}
\usepackage{hyperref}

\usepackage[linesnumbered,ruled,vlined]{algorithm2e}

\usepackage[margin=1in]{geometry}
\usepackage{subcaption}
\usepackage{pgf,tikz}
\usepackage{epstopdf}
\usepackage{graphicx}
\usepackage{mathrsfs}
\usepackage{bm}
\usetikzlibrary{shapes,arrows,calc}
\usepackage{lscape}

\title{Exploiting Sparsity for Semi-Algebraic Set Volume Computation\footnote{This work was partly funded by the RTE company and by the ERC Advanced Grant Taming.}}

\newtheorem{thm}{Theorem}
\newtheorem{lem}[thm]{Lemma}
\newtheorem{cor}[thm]{Corollary}

\newtheorem{asm}[thm]{Assumption}

\theoremstyle{definition}
\newtheorem{defn}[thm]{Definition}
\newtheorem{rqe}[thm]{Remark}

\newcommand{\R}{\mathbb{R}}
\newcommand{\N}{\mathbb{N}}

\newcommand{\mcK}{\mathcal{K}}
\newcommand{\mcE}{\mathcal{E}}
\newcommand{\mcF}{\mathcal{F}}
\newcommand{\mcT}{\mathcal{T}}
\newcommand{\mcC}{\mathscr{C}}
\newcommand{\mcM}{\mathscr{M}}
\newcommand{\mcD}{\mathcal{D}}

\newcommand{\mcO}{\mathbf{O}}

\newcommand{\mcP}{\mathcal{P}}

\newcommand{\mbA}{\mathbf{A}}
\newcommand{\mbX}{\mathbf{X}}
\newcommand{\mbY}{\mathbf{Y}}
\newcommand{\mbZ}{\mathbf{Z}}
\newcommand{\mbB}{\mathbf{B}}
\newcommand{\mbP}{\mathbf{P}}
\newcommand{\mbR}{\mathbf{R}}

\newcommand{\mbM}{\mathbf{M}}
\newcommand{\mbU}{\mathbf{U}}
\newcommand{\mbV}{\mathbf{V}}

\newcommand{\mbK}{\mathbf{K}}

\newcommand{\mbx}{\mathbf{x}}
\newcommand{\mbe}{\mathbf{e}}
\newcommand{\mbz}{\mathbf{z}}
\newcommand{\mby}{\mathbf{y}}
\newcommand{\mbv}{\mathbf{v}}
\newcommand{\mbw}{\mathbf{w}}
\newcommand{\mbp}{\mathbf{p}}
\newcommand{\mbm}{\mathbf{m}}
\newcommand{\mbn}{\mathbf{n}}
\newcommand{\mbh}{\mathbf{h}}
\newcommand{\mbg}{\mathbf{g}}

\newcommand{\vol}{\mathrm{vol}}
\newcommand{\spt}{\mathrm{spt}}

\newcommand{\ind}{\mathbb{1}}
\newcommand{\argmax}{\mathrm{argmax}}
\newcommand{\grad}{\mathbf{grad}\:}
\newcommand{\dv}{\mathrm{div}\:}

\allowdisplaybreaks[1]

\tikzstyle{decision} = [diamond, draw, fill=blue!20, 
    text centered, node distance=3cm, inner sep=0pt]
\tikzstyle{block} = [rectangle, draw, text centered, rounded corners, minimum height=1em]
\tikzstyle{init} = [rectangle, draw,
    text width=15em, text centered, rounded corners, minimum height=4em, node distance = 5cm]
\tikzstyle{line} = [draw, -latex']
\tikzstyle{cloud} = [draw, ellipse,fill=red!20, text width = 5 em, text badly centered, node distance=5.5cm,
    minimum height=2em]

\definecolor{qqttcc}{rgb}{0.,0.2,0.8}
\definecolor{ffqqqq}{rgb}{1.,0.,0.}
\definecolor{fleche}{rgb}{0.0,0.5,0.0}
\definecolor{ffffff}{rgb}{1.,1.,1.}
\definecolor{qqqqff}{rgb}{0.,0.,1.}
\definecolor{ttzzqq}{rgb}{0.2,0.6,0.}

\begin{document}


\author{M. \textsc{Tacchi}$^{1,2}$ \and T. \textsc{Weisser}$^3$ \and J. B. \textsc{Lasserre}$^{1,4}$ \and D. \textsc{Henrion}$^{1,5}$}

\footnotetext[1]{LAAS-CNRS, Universit\'e de Toulouse, France.}
\footnotetext[2]{R\'eseau de Transport d'Electricit\'e (RTE), France.}
\footnotetext[3]{Theoretical Division (T-5/CNLS), Los Alamos National Laboratory, Los Alamos, USA.}
\footnotetext[4]{Institut de Math\'ematiques de Toulouse,  Universit\'e de Toulouse, France.}
\footnotetext[5]{Faculty of Electrical Engineering, Czech Technical University in Prague, Czechia.}

\maketitle

\begin{abstract}
We provide a systematic deterministic numerical scheme
to approximate the volume (i.e. the Lebesgue measure) of a basic semi-algebraic set
whose description follows a {correlative} sparsity pattern. As in previous works (without sparsity), the underlying strategy 
is to consider an infinite-dimensional linear program on measures
whose optimal value is the volume of the set. This is a particular instance of a generalized moment problem which in turn can be approximated as closely as desired by solving a hierarchy of semidefinite relaxations of increasing size. The novelty 
with respect to previous work is that by exploiting the sparsity pattern  we can provide a
sparse formulation for which the associated semidefinite relaxations are of much smaller size.
In addition, we can decompose the sparse relaxations into completely decoupled
subproblems of smaller size, and in some cases computations can be done
in parallel. To the best of our knowledge, it is the first  contribution that exploits {correlative} sparsity for volume computation of semi-algebraic sets which are possibly high-dimensional and/or non-convex and/or non-connected.

\end{abstract}

\section{Introduction}

This paper is in the line of research concerned with computing
approximations of the volume (i.e. the Lebesgue measure) of a given compact basic semi-algebraic set 
$\mbK$ of $\R^n$ neither necessarily convex nor connected. 
Computing or even approximating the volume of a convex body
is hard theoretically and in practice as well. Even if $\mbK$ is a convex polytope,
exact computation of its volume or integration over $\mbK$ is a computational challenge.
Computational complexity of these problems is discussed in, e.g. \cite{bollobas,dyer1,elekes}.
In particular, any deterministic algorithm with polynomial-time complexity
that would compute an upper bound and a lower bound  on the volume cannot yield 
an estimate on the bound ratio better than polynomial in
the dimension $n$. For more detail, the interested reader 
is referred to the discussion in  \cite{sirev} and to  \cite{bueler}
for a comparison.

If one accepts randomized algorithms that fail with small probability, then the situation
is more favorable. Indeed, the
probabilistic
approximation algorithm of \cite{dyer2}  computes the volume to fixed 
arbitrary relative precision $\epsilon>0$  in time polynomial in $1/\epsilon$. The algorithm uses approximation schemes based on rapidly mixing Markov chains and isoperimetric inequalities, see also 
hit-and-run algorithms described in e.g. \cite{belisle1,smith1,smith2}.
So far, it seems that the recent work \cite{vempala} has provided the best algorithm of this type.

\subsection*{The moment approach for volume computation}
In \cite{sirev} a general deterministic methodology 
was proposed for approximating  the volume of a compact basic semi-algebraic set $\mbK$, not necessarily convex or connected. 
It was another illustration of the versatility of the 
so-called  \emph{moment-SOS (sums of squares) hierarchy} developed in \cite{lass-book1} for solving the Generalized Moment
Problem (GMP) with polynomial data.

Briefly, the underlying 
idea is to view the volume as the optimal value of a GMP, i.e., an infinite-dimensional Linear Program (LP)
on an appropriate space of finite Borel measures. Then one may approximate the value from above by solving a hierarchy of \emph{semidefinite programming (SDP) relaxations} with associated sequence of optimal values indexed by an integer $d$. 
Monotone convergence of the bounds is guaranteed when $d$ increases. Extensions to more general measures
and possibly non-compact sets were then proposed in \cite{aam}.
The order $d$ in the hierarchy encodes the amount of information  that is used, namely the number of moments of the Lebesgue measure up to degree $d$.  It is a crucial factor for the size of the associated SDP problem. More precisely, the size grows in $d^n$ as $d$ increases, which so far limits its application to sets of small dimension $n$, typically up to $4$ or $5$.

In view of the growth of the size of the SDP problem  with increasing order it is desirable to converge
quickly towards the optimal value of the LP. However, this convergence is expected to be slow in general. 
One reason becomes clear when looking at the dual LP
which attempts to compute a polynomial approximation (from above) of the (discontinuous)  indicator function of the set $\mbK$. Hence one is then faced with a typical Gibbs effect\footnote{The Gibbs effect appears at a jump discontinuity when one approximates a piecewise continuously differentiable function with a continuous 
function, e.g. by its Fourier series.}, well-known in 
the theory of approximation. To overcome this drawback the authors in \cite{sirev} have proposed 
to use an alternative criterion for the LP, which results in a significantly faster
convergence. However in doing so, the monotonicity of the convergence
(a highly desirable feature to obtain a sequence of upper bounds) is lost. In the related work \cite{aam} the author
has proposed an alternative strategy which consists of strengthening the relaxations by incorporating additional linear moment constraints known to be satisfied at the optimal solution of the LP.  These constraints come from a specific application of Stokes' theorem. Remarkably, adding these \emph{Stokes constraints} results in a significantly faster convergence while keeping monotonicity.

\subsection*{Motivation}

The measure approach for the volume computation problem is intimately linked to the use of occupation measures, in dynamical systems theory, for computing the region of attraction (ROA) of a given target set. Indeed, in \cite{roa}, the problem of estimating the ROA is formulated as a GMP very similar to the volume computation problem. The idea is to maximize the volume of a set of initial conditions that yield trajectories ending in the target set after a given time.

This problem of estimating the ROA of a target set is crucial in power systems safety assessment, since the power grids must have good stability properties. Perturbations (such as short-circuits or unscheduled commutations) should be tackled before they get the system outside the ROA of the nominal operating domain. Currently the stability of power grids is estimated through numerous trajectory simulations that prove very expensive. In the wake of the energy transition, it is necessary to find new tools for estimating the stability of power systems. The conservative, geometric characterization of the region of attraction as formulated in \cite{roa} is a very promising approach for this domain of application, see \cite{josz}.

As in volume computation, the main limitation of this method is that only problems of modest dimension can be handled by current solvers. Exploiting sparsity seems to be the best approach to allow scalability both in volume computation and ROA estimation. Since volume estimation is a simpler instance of the GMP than ROA estimation, we decided to address first the former problem.

In addition, volume computation with respect to a measure satisfying some conditions (e.g. compactly supported or Gaussian measure) also has applications in the fields of geometry and probability computation, which is the reason why many algorithms were already proposed for volume computation of convex polytopes and convex bodies in general.

\subsection*{Contribution}

We design deterministic methods that provide
approximations with strong asymptotic guarantees of convergence to the volume of $\mbK$.
The methodology that we propose is similar in spirit to the one initially developed in \cite{sirev}
as described above and its extension to non-compact sets and 
Gaussian measures of \cite{aam}. However it is not a straightforward or direct extension of \cite{sirev} or \cite{aam}, and it has the following important distinguishing features:

(i) It can handle sets $\mbK\subset\R^n$ of potentially
large dimension $n$ provided that some sparsity pattern {(namely: correlative sparsity, see section \ref{prelimpattern} as well as \cite{waki, diego} for details)} is present in the description of $\mbK$.
This is in sharp contrast with \cite{sirev}.

(ii) The computation of upper and lower bounds can
be decomposed into smaller independent problems of the same type, and  
depending on the sparsity pattern, some of the computations can even be done in parallel. 
This fact alone  is remarkable and unexpected.

To the best of our knowledge, this is the first deterministic method for volume computation that takes benefit from a {correlative} sparsity pattern in the description of $\mbK$ in the two directions
of (a) decomposition into problems of smaller size and (b) parallel computation. {Of course this sharp improvement is performed at some price: our framework only works on semi-algebraic sets that present the appropriate correlative sparsity pattern (see Assumption \ref{asm} as well as section \ref{secdip} and appendix \ref{DIP} for detailed discussion on its applicability).}

The key idea is to provide a new and very specific \emph{sparse formulation} of the original problem
in which one defines a set of marginal measures whose (small dimensional) support is in accordance with the {correlative} sparsity pattern present in the description of the set $\mbK$. However, those
marginal measures are not similar to the ones used in the sparse formulation \cite{siam-sparse}
of polynomial optimization problems 
over the same set $\mbK$. Indeed they are not expected to satisfy the consistency
condition of \cite{siam-sparse}\footnote{If two measures share some variables then the consistency condition requires that their respective marginals coincide.}.

Finally, in principle, our floating point volume computation in large dimension $n$ is faced with a crucial numerical issue.
Indeed as in Monte-Carlo methods, up to rescaling, one has to include
the set $\mbK$ into a box $\mbB$ of unit volume. Therefore the volume of $\mbK$ is of the order $\epsilon^n$ for some $\epsilon \in (0,1)$ and therefore far beyond machine precision
as soon as $n$ is large.  To handle this critical issue we develop 
a \emph{sparsity-adapted rescaling} which allows us to compute very small volumes in potentially very high dimension with good precision.

\subsection*{A motivating example}

Consider the following set
{
$$ \mbK := \{ \mbx \in [0,1]^{100} \; : \;   x_i x_{i+1}  \leq 1/2, \: i=1,\ldots,99\}. $$
}
This is a {\em high-dimensional non-convex sparse semi-algebraic set}.  The precise definition of a {sparse semi-algebraic set} will be given later on, but so far notice that in the description of $\mbK$ each constraint involves only $2$ variables out of $100$.
The volume of $\mbK$ is hard to compute, but thanks to the structured description of the set we are able to prove numerically that its volume is smaller than {$2\cdot 10^{-5}$ in less than 2 minutes on a standard computer.}

For this we have implemented a specific version of the moment-SOS hierarchy of 
SDP relaxations to solve the GMP, in which we exploit the {correlative} sparsity pattern of the set $\mbK$. The basic idea is to replace the original GMP that involves an unknown measure on $\R^{100}$ (whose SDP relaxations are hence untractable)
with a GMP involving $99$ measures on $\R^2$ (hence tractable). In addition, this new GMP 
can be solved either in one shot (with the $99$ unknown measures) or by solving sequentially $99$ GMPs involving (i) one measure on $\R^2$ and (ii) some data obtained from the GMP solved at previous step.
Our approach can be sketched as follows.

First, we rescale the problem so that the set $\mbK$ is included in the unit box $\mbB := [0,1]^n$ on which the moments of the Lebesgue measure are easily computed.

Next, we describe the volume problem on $\mbK$ as a chain of volume subproblems on the subspaces Im$(\pi_i)$ where $\pi_i(x_1,\dots,x_{100}) = (x_i,x_{i+1})$, with a link between 
the $i$-th and $(i+1)$-th sub-problems.

Finally, in this example, as $n=100$ and $\mbK\subset \mbB$, the  volume of $\mbK$ is very small and far below standard floating point machine precision.
To handle this numerical issue, we have implemented a sparsity-adapted strategy which consists of rescaling each subproblem defined on the projections of $\mbK$ to obtain 
intermediate values all with the same order of magnitude. Once all computations (involving quantities of the same order of magnitude) have been performed, the correct value of the volume is obtained  by a reverse scaling.

The sparse formulation stems from considering some measure marginals
appropriately defined according to the {correlative} sparsity pattern present
in the description of $\mbK$. It leads to a variety of algorithms to compute the volume of sparse semi-algebraic sets.

The outline of the paper is as follows. In Section \ref{cylinder} we describe briefly the moment-SOS hierarchy for semi-algebraic set volume approximation, as well as our notion of sparse semi-algebraic set. In Section \ref{LCT}, we introduce a first constructive theorem that allows to efficiently compute the volume of specific sparse sets with the hierarchy. Section \ref{stokes} is dedicated to a method for accelerating the convergence of the sparse hierarchy, and in section \ref{secex} we discuss some numerical examples. Eventually, Section \ref{DCT} presents the general theorem for computing the volume of {a more general variety of} sparse sets using parallel computations, accompanied with some other examples.

\section{Preliminaries}
\label{cylinder}

\subsection{Notations and definitions}

Given a closed set $\mbK\subset\R^n$, we denote by
$\mathscr{C}(\mbK)$ the space of continuous functions on $\mbK$,
$\mathscr{M}(\mbK)$ the space of finite signed Borel measures on $\mbK$, with $\mathscr{C}_+(\mbK)$ and
$\mathscr{M}_+(\mbK)$ their respective cones of positive elements.

The Lebesgue measure on $\R^n$ is $\lambda^n(d\mbx):=d x_1 \otimes d x_2 \otimes \cdots \otimes dx_n=d\mbx$, and its
restriction to a set $\mbK \subset \R^n$ is
$\lambda_\mbK:=\ind_\mbK\lambda^n$ , where $\ind_\mbK$ denotes the indicator function equal to $1$ in $\mbK$ and zero outside.
In this paper, we focus on computing the volume or Lebesgue measure of $\mbK$, that we denote by $\vol\:\mbK$ or $\lambda^n(\mbK)$ or $\lambda_\mbK(\R^n)$. 

Given a Euclidean space $\mbX$ and a subspace $\mbX_i \subset \mbX$, the orthogonal projection map from $\mbX$ to $\mbX_i$ is denoted by $\pi_{\mbX_i}$. {Let $\mbe_j$ denote the $j$-th vector of the canonical basis of $\R^n$ such that if $\mbx = (x_1,\dots,x_n)$ then $x_j = \mbx \cdot \mbe_j$, where the dot denotes the scalar product between vectors.} The $m$-dimensional subspace spanned by {vectors $\mbe_{i_1}, \ldots, \mbe_{i_m}$ is denoted $\langle x_{i_1}, \ldots, x_{i_m} \rangle$ or $\langle x_{i_j} \rangle_{1 \leq j \leq m}$}. Given a measure $\mu\in\mathscr{M}(\mbX)$, its marginal with respect to $\mbX_i$ is denoted by $\mu^{{\mbX}_i} \in \mathscr{M}(\mbX_i)$. It is equal to the image or push-forward measure of $\mu$ through the map $\pi_{\mbX_i}$.

Let $\R[\mbx]$ be the ring of polynomials in the variables $\mbx=(x_1,\ldots,x_n)$ and let
$\R[\mbx]_d$ be the vector space of polynomials of degree at most $d$, whose dimension is $s(d):={n+d\choose n}$.
For every $d\in\N$, let  $\N^n_d:=\{\alpha\in\N^n:\vert\alpha\vert \,(=\sum_i\alpha_i)\leq d\}$, 
and
let $\mbv_d(\mbx)=(\mbx^\alpha)_{\alpha\in\N^n_d}$ be the vector of monomials of the canonical basis 
$(\mbx^\alpha)_{\alpha\in\N^n_d}$ of $\R[\mbx]_{d}$. 
A polynomial $p\in\R[\mbx]_d$ is written as
\[\mbx\mapsto p(\mbx)\,=\,\sum_{\alpha\in\N^n}p_\alpha\,\mbx^\alpha\,=\,\mbp \cdot\mbv_d(\mbx)\]
for some vector of coefficients $\mathbf{p}=(p_\alpha)_{\alpha\in\N^n_d} \in\R^{s(d)}$.

We say that $\mu \in \mcM(\R^n)$ is a {representing measure} of the sequence $\mbm=(m_\alpha)_{\alpha\in\N^n} \subset \R$ whenever
\[
m_\alpha= \int \mbx^\alpha\,d\mu(\mbx)
\]
for all $\alpha\in\N^n$. Given a sequence $\mbm=(m_\alpha)_{\alpha\in\N^n}$, let $L_\mbm:\R[\mbx]\to\R$ be the linear functional
\[f\:\left(=\sum_\alpha f_\alpha\, \mbx^\alpha\right)\:\mapsto L_\mbm(f)\,:=\,\sum_\alpha f_\alpha\,m_\alpha.\]
Given a sequence $\mbm=(m_\alpha)_{\alpha\in\N^n}$, and a polynomial $g := \sum_\gamma g_\gamma \, \mbx^\gamma \in\R[\mbx]$, the {localizing} moment matrix of order $d$
associated with $\mbm$ and $g$ is the 
real symmetric matrix $\mbM_d(g\,\mbm)$ of size $s(d)$ with rows and columns indexed in $\N^n_d$ and with entries 
\begin{eqnarray*}
\mbM_d(g\,\mbm)(\alpha,\beta)&:=& L_\mbm(g(\mbx)\,\mbx^{\alpha+\beta})\\
&=&\sum_\gamma g_\gamma\,m_{\alpha+\beta+\gamma},\quad \alpha,\beta\in\N^n_d.
\end{eqnarray*}
When $g\equiv 1$, the localizing moment matrix $\mbM_d(\mbm)$ is called simply the moment matrix.

\subsection{The Moment-SOS hierarchy for volume computation}

Let $\mbB:=[0,1]^n \subset \mbX:=\R^n$ be the $n$-dimensional unit box, and let $\mbK\subset\mbB$ be a closed basic semialgebraic set defined by
\[
\mbK\,:=\,\{\mbx\in\mbX:\: g_i(\mbx)\,\geq\,0,\:i=1,\ldots,m\,\} = \{\mbx\in\mbX:\: \mbg(\mbx)\,\geq\,0\}
\]
where $\mbg=(g_i)_{i=1,\ldots,m}\in \R[\mbx]^m$ and the rightmost vector inequality is meant entrywise.
As in \cite{sirev} consider the infinite-dimensional linear program (LP) on measures
\begin{equation}
\label{general-lp}
\begin{array}{rl}
\displaystyle \max_{\mu,\:\hat{\mu}\in\mathscr{M}_+(\mbB)} & \displaystyle \int d\mu \\
\mathrm{s.t.} & \mu+\hat{\mu} = \lambda_\mbB \\
& \spt\:\mu \subset\mbK\\
& \spt\:\hat{\mu}\subset\mbB.
\end{array}
\end{equation}
Its value is equal to $\vol\:\mbK$ and the measures  $\mu^*=\lambda_{\mbK}$, $\hat{\mu}^* := \lambda_{\mbB\backslash\mbK}$ are the unique optimal solutions of \eqref{general-lp}. The dual of
\eqref{general-lp} is  the infinite-dimensional LP on continuous functions
\begin{equation}
\label{general-lp-dual}
\begin{array}{rl}
\displaystyle \inf_{v\in\mathscr{C}_+(\mbB)} & \displaystyle \int v\,d\lambda_\mbB \\
\mathrm{s.t.}  & v\geq \ind_{\mbK}.
\end{array}
\end{equation}
It turns out that there is no duality gap between \eqref{general-lp} and \eqref{general-lp-dual}, i.e., they both have the same optimal value. Notice that a minimizing sequence of \eqref{general-lp-dual}
approximates the indicator function $\ind_{\mbK}$ from above by polynomials of increasing degrees.

The LP \eqref{general-lp} is a particular and simple instance of the {Generalized Moment Problem} (GMP). As described in \cite{sirev} one may approximate its optimal value as closely as desired
by using the following  key result in \cite{lass-book1}. Given an infinite dimensional LP on measures, one can construct a hierarchy of finite dimensional semidefinite programs\footnote{A semidefinite program is a convex conic optimization problem that can be solved numerically efficiently, e.g. by using interior point methods.}
 (SDP) whose associated sequence of optimal values converges monotonically to the optimal value of the original infinite dimensional LP. The basic idea is to represent a measure with the sequence $\mbm$ of its moments, and to formulate finite dimensional SDPs on truncations of the sequence $\mbm$. When this strategy is applied to LP \eqref{general-lp},
 the step $d$ of the moment-SOS hierarchy consists of solving the SDP relaxation 
 \begin{equation}
 \label{sdp-primal}
 \begin{array}{rrl}
 \mbP_d: & \displaystyle 
 \max_{\mbm,\:\widehat{\mbm}  \in \R^{s(d)}} & m_0 \\
 & \mathrm{s.t.} & m_\alpha+\widehat{m}_\alpha= \int_{\mbB} \mbx^\alpha \: d\mbx,\quad \alpha\in\N^n_{2d}\\
 && \mbM_d(\mbm) \succeq 0,\:\mbM_d(\widehat{\mbm})\,\succeq0\\
&& \mbM_{d-d_i}(g_i\,\mbm)\,\succeq0,\quad i=1,\ldots,m
\end{array}
\end{equation}
where $\mbm=(m_\alpha)_{\alpha\in\N^n_{2d}}$,
$\widehat{\mbm}=(\widehat{m}_\alpha)_{\alpha\in\N^n_{2d}}$, and $d_i=\lceil ({\rm deg}\,g_i)/2\rceil$,
$i=1,\ldots,m$.

The sequence of SDP problems  $(\mbP_d)_{d \in \N}$ indexed by the relaxation order $d$ is a hierarchy in the sense that its sequence of values converges monotonically from above to $\vol\:\mbK$ when $d$ increases. Each SDP relaxation $\mbP_d$ has a dual formulated in terms of sums of squares (SOS) of polynomials, which leads to a dual SOS hierarchy, whence the name \emph{moment-SOS hierarchy}. 
The basic moment-SOS hierarchy can be modeled using the GloptiPoly Matlab toolbox \cite{glopti} and solved using any SDP solver, e.g. SeDuMi or Mosek. For more details on the moment-SOS hierarchy and some of its applications, the interested reader is referred to \cite{lass-book1}.

\subsection{The {correlative} sparsity pattern and its graph representation}

\label{prelimpattern}

This work heavily relies on a specific notion of sparsity defined as follows.

\begin{defn}
A scalar polynomial $p$ is said to be \textit{sparse} when its vector of coefficients $\mbp$ is sparse. In other words, $p$ is a linear combination of a small number of monomials.
\end{defn}

\begin{defn}
A family of polynomial vectors $(\mbg_1,\dots,\mbg_m)$ is said to be \textit{correlatively sparse} whenever its correlative sparsity pattern matrix $\mbR := (R_{ij})_{1 \leq i,j \leq n}$, defined by
$$ R_{ij} :=  \delta_{ij} + \sum\limits_{k=1}^m \left\Vert \frac{\partial}{\partial x_i} \mbg_k \right\Vert\left\Vert \frac{\partial}{\partial x_j} \mbg_k \right\Vert$$
(where $\delta_{ij}=1$ if $i=j$ and $0$ otherwise, and $\|\cdot\|$ is any norm on polynomial vectors), is sparse. In other words, for many pairs of indices $i\neq j$, the variables $x_i$ and $x_j$ do not appear simultaneously in any element of $\{\mbg_1,\dots,\mbg_m\}$.
\end{defn}

\begin{defn}
The \textit{correlation graph} $G = (V,E)$ of $(\mbg_1,\dots,\mbg_m)$ is defined by vertices $V = \{1,\dots,n\}$ and edges $E = \left\{(i,j) \in \{1,\dots,n\}^2 : i\neq j \ \& \ R_{ij} \neq 0 \right\}$.
The \textit{correlative sparsity} $CS$ of $(\mbg_1,\dots,\mbg_m)$ is the treewidth of its correlation graph. 
\end{defn}

This paper proposes a method to reduce the size of the volume computation SDP to $CS + d + 1 \choose d$ instead of $n + d \choose d$, under appropriate assumptions.

\begin{defn}
The \textit{support} of $\mbg_i$ is the set $S(\mbg_i) := \left\{j \in \{1,\dots,n\} : \frac{\partial}{\partial x_j} \mbg_i \neq 0 \right\}$. 
The \textit{support subspace} of $\mbg_i$ is the set $\mbX_i := \langle x_j \rangle_{j \in S(\mbg_i)}$, whose dimension is smaller than $n$. Since by definition $\mbg_i = \mbg_i \circ \pi_{\mbX_i}$, we use both notations with the same meaning. Then $\mbX := \sum_{i=1}^m \mbX_i$ is called the \textit{coordinate subspace decomposition} associated to $(\mbg_1,\dots,\mbg_m)$. 
\end{defn}

Without loss of generality, we can suppose that $\mbX = \R^n$ (otherwise, there would be variables that appear in none of the $\mbg_i$s).
\color{black}

\begin{defn}
A {\em sparse basic semi-algebraic set} has a description
\[
\mbK := \{ \mbx \in \mbX \; : \; \mbg_i(\pi_{\mbX_i}(\mbx)) \geq 0, \: i=1,\ldots,m\}
\]
where $(\mbg_i)_{1 \leq i \leq m}$ is a { correlatively sparse family of polynomial vectors} (inequalities are meant entrywise) and { $\mbX = \sum_{i=1}^m \mbX_i$ is the coordinate subspace decomposition associated to $(\mbg_i)_{1 \leq i \leq m}$ (and, by extension, to $\mbK$). A \textit{sparse semi-algebraic set} is a finite union of sparse basic semi-algebraic sets that share the same coordinate subspace decomposition.}
\end{defn}

A simple example of a sparse basic semi-algebraic set is
\begin{equation}\label{sparsex1}
\mbK := \{\mbx = (x_1,x_2,x_3,x_4) \in \R^4 \; : \; \mbg_1(x_1,x_2) \geq 0,\: \mbg_2(x_2,x_3) \geq 0, \: \mbg_3(x_3,x_4) \geq 0 \}
\end{equation}
for $\mbX=\R^4$, $\mbX_1= \langle x_1,x_2 \rangle$, $\mbX_2 = \langle x_2,x_3 \rangle$, $\mbX_3 = \langle x_3,x_4 \rangle$ and the projection maps are
$\pi_{\mbX_1}(\mbx)=(x_1,x_2)$, $\pi_{\mbX_2}(\mbx)=(x_2,x_3)$, $\pi_{\mbX_3}(\mbx)=(x_3,x_4)$.

Our methodology is based on the classical theory of clique trees. Up to a chordal extension (which is equivalent to slightly weakening the {correlative} sparsity pattern), we suppose that the {correlation} graph is \emph{chordal} (i.e. every cycle of length greater than 3 has a chord, i.e. an edge linking two nonconsecutive vertices). Then we construct the {maximal} \emph{cliques} of the graph (a clique $C$ is a subset of vertices such that every vertex of $C$ is connected to all the other vertices of $C$, { a clique is maximal when its cardinal is maximal}). {Most of the time, up to concatenation of some of the $\mbg_i$ the maximal cliques of a chordal correlation graph are exactly the supports of the $\mbg_i$: $C_i = S(\mbg_i)$, so in the following we will consider only such case\footnote{The only exception would be cliques forming a triangle and is tackled in detail in section \ref{secdip}.}}. Figure \ref{fig:sparsex1} illustrates this constuction for the sparse set (\ref{sparsex1}), the vertices are denoted by $x_i$ and our {maximal} cliques are denoted by $C_j$.

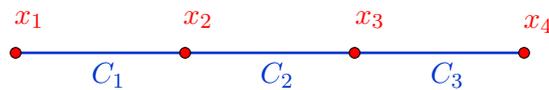
\begin{figure}[!h]
\begin{center}
\begin{tikzpicture}[line cap=round,line join=round,>=triangle 45,x=1.0cm,y=1.0cm]
\clip(-2.5,-0.7) rectangle (5.,0.7);
\draw [line width=1.pt,color=qqttcc] (-2.230336297512104,0.)-- (0.,0.);
\draw [line width=1.pt,color=qqttcc] (0.,0.)-- (2.230336297512104,0.);
\draw [line width=1.pt,color=qqttcc] (2.230336297512104,0.)-- (4.460672595024208,0.);
\draw [fill=ffqqqq] (0.,0.) circle (2.0pt);
\draw[color=ffqqqq] (0.17,0.44) node {$x_2$};
\draw [fill=ffqqqq] (-2.230336297512104,0.) circle (2.0pt);
\draw[color=ffqqqq] (-2.05,0.44) node {$x_1$};
\draw [fill=ffqqqq] (2.230336297512104,0.) circle (2.0pt);
\draw[color=ffqqqq] (2.43,0.44) node {$x_3$};
\draw[color=qqttcc] (-1.01,-0.3) node {$C_1$};
\draw [fill=ffqqqq] (4.460672595024208,0.) circle (2.0pt);
\draw[color=ffqqqq] (4.65,0.38) node {$x_4$};
\draw[color=qqttcc] (1.21,-0.3) node {$C_2$};
\draw[color=qqttcc] (3.45,-0.3) node {$C_3$};
\end{tikzpicture}
\label{fig:4dlin}
\end{center}
\caption{Graph associated to the sparse set (\ref{sparsex1}).}
\label{fig:sparsex1}
\end{figure}

Then, we construct a \emph{clique tree} which is instrumental to the computer implementation. It is proved in \cite{graphs} that if the graph is chordal, then its {maximal} cliques can be organized within a tree satisfying the clique intersection property: for two {maximal} cliques $C$ and $C'$ the intersection $C\cap C'$ is contained in every {maximal} clique on the path from $C$ to $C'$.
Figure \ref{fig:4dcl} represents the clique tree associated to the sparse set (\ref{sparsex1}).

\begin{figure}[!h]
\begin{center}
\begin{tikzpicture}[line cap=round,line join=round,>=triangle 45,x=1.0cm,y=1.0cm]
\clip(-1.5,-0.7) rectangle (3.7,1.);
\draw [line width=1.pt] (-1.115168148756052,0.)-- (1.115168148756052,0.);
\draw [line width=1.pt] (1.115168148756052,0.)-- (3.3455044462681562,0.);
\draw [fill=qqttcc] (-1.115168148756052,0.) circle (2.0pt);
\draw[color=qqttcc] (-1.09,0.58) node {$C_1$};
\draw [fill=qqttcc] (1.115168148756052,0.) circle (2.0pt);
\draw[color=qqttcc] (1.13,0.62) node {$C_2$};
\draw [fill=qqttcc] (3.3455044462681562,0.) circle (2.0pt);
\draw[color=qqttcc] (3.27,0.6) node {$C_3$};
\draw[color=fleche,->] (3.3455044462681562,-0.2) to [bend left = 45] (1.115168148756052,-0.2);
\draw[color=fleche,->] (1.115168148756052,-0.2) to [bend left = 45] (-1.115168148756052,-0.2);
\end{tikzpicture}
\caption{Linear clique tree associated to the sparse set (\ref{sparsex1}).}
\label{fig:4dcl}
\end{center}
\end{figure}
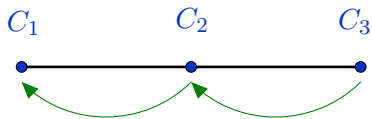

For a slightly more complicated illustration, consider the sparse set
\begin{equation}\label{sparsex2}
\mbK := \{\mbx \in \R^6 \; : \; \mbg_1(x_1,x_2) \geq 0, \: \mbg_2(x_2,x_3,x_4) \geq 0, \: \mbg_3(x_3,x_5) \geq 0, \: \mbg_4(x_4,x_6) \geq 0\}
\end{equation}
whose {correlation} graph is represented on Figure \ref{fig:sparsex2} and whose clique tree is represented on Figure \ref{fig:6dcl}.
The clique tree of Figure \ref{fig:4dcl} is called \emph{linear} because 
all {maximal} cliques form a single chain (i.e. they are in a sequence)  with no branching. In contrast, the clique tree of Figure \ref{fig:6dcl} is called \emph{branched}.

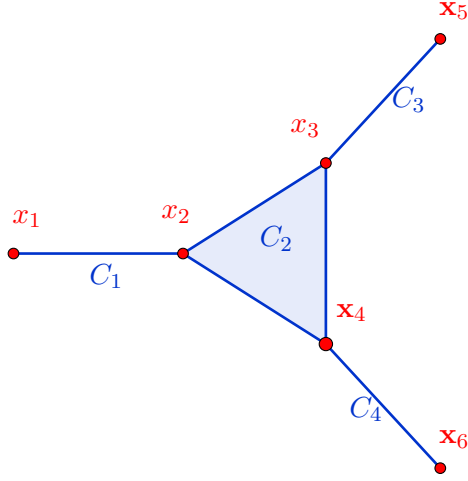
\begin{figure}[!h]
\begin{center}
\begin{tikzpicture}[line cap=round,line join=round,>=triangle 45,x=1.0cm,y=1.0cm]
\clip(-2.5,-3.) rectangle (4.,3.5);
\fill[line width=1.pt,color=qqttcc,fill=qqttcc,fill opacity=0.10000000149011612] (0.,0.) -- (1.88,1.2) -- (1.88,-1.2) -- cycle;
\draw [line width=1.pt,color=qqttcc] (0.,0.)-- (1.88,1.2);
\draw [line width=1.pt,color=qqttcc] (1.88,1.2)-- (1.88,-1.2);
\draw [line width=1.pt,color=qqttcc] (1.88,-1.2)-- (0.,0.);
\draw [line width=1.pt,color=qqttcc] (1.88,1.2)-- (3.3853576993004233,2.845690796339621);
\draw [line width=1.pt,color=qqttcc] (1.88,-1.2)-- (3.3853576993004233,-2.845690796339621);
\draw [line width=1.pt,color=qqttcc] (-2.230336297512104,0.)-- (0.,0.);
\draw [fill=ffqqqq] (0.,0.) circle (2.0pt);
\draw[color=ffqqqq] (-0.09,0.52) node {$x_2$};
\draw [fill=ffqqqq] (1.88,1.2) circle (2.0pt);
\draw[color=ffqqqq] (1.61,1.66) node {$x_3$};
\draw [fill=ffqqqq] (1.88,-1.2) circle (2.5pt);
\draw[color=ffqqqq] (2.22,-0.78) node {$\mbx_4$};
\draw[color=qqttcc] (1.23,0.2) node {$C_2$};
\draw [fill=ffqqqq] (3.3853576993004233,2.845690796339621) circle (2.0pt);
\draw[color=ffqqqq] (3.57,3.22) node {$\mbx_5$};
\draw [fill=ffqqqq] (3.3853576993004233,-2.845690796339621) circle (2.0pt);
\draw[color=ffqqqq] (3.57,-2.46) node {$\mbx_6$};
\draw[color=qqttcc] (2.97,2.04) node {$C_3$};
\draw[color=qqttcc] (2.41,-2.06) node {$C_4$};
\draw [fill=ffqqqq] (-2.230336297512104,0.) circle (2.0pt);
\draw[color=ffqqqq] (-2.05,0.46) node {$x_1$};
\draw[color=qqttcc] (-1.01,-0.3) node {$C_1$};
\end{tikzpicture}
\end{center}
\caption{Graph associated to the sparse set (\ref{sparsex2}).}
\label{fig:sparsex2}
\end{figure}

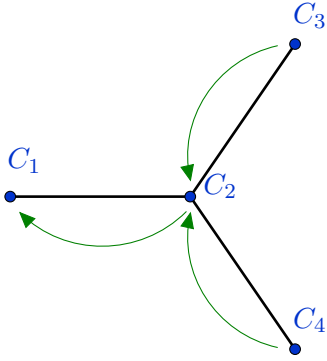
\begin{figure}[!h]
\begin{center}
\begin{tikzpicture}[line cap=round,line join=round,>=triangle 45,x=1.0cm,y=1.0cm]
\clip(-1.5,-3) rectangle (3.5,2.7);
\draw [line width=1.pt] (-1.115168148756052,0.)-- (1.2533333333333332,0.);
\draw [line width=1.pt] (1.2533333333333332,0.)-- (2.632678849650212,2.0228453981698107);
\draw [line width=1.pt] (1.2533333333333332,0.)-- (2.632678849650212,-2.0228453981698107);
\draw [fill=qqttcc] (-1.115168148756052,0.) circle (2.0pt);
\draw[color=qqttcc] (-0.93,0.48) node {$C_1$};
\draw [fill=qqttcc] (2.632678849650212,2.0228453981698107) circle (2.0pt);
\draw[color=qqttcc] (2.83,2.4) node {$C_3$};
\draw [fill=qqttcc] (2.632678849650212,-2.0228453981698107) circle (2.0pt);
\draw[color=qqttcc] (2.83,-1.64) node {$C_4$};
\draw [fill=qqttcc] (1.2533333333333332,0.) circle (2.0pt);
\draw[color=qqttcc] (1.65,0.14) node {$C_2$};
\draw[color=fleche,->] (2.4,2.0) to [bend right = 45] (1.26,0.2);
\draw[color=fleche,->] (2.4,-2.0) to [bend left = 45] (1.26,-0.2);
\draw[color=fleche,->] (1.2,-0.2) to [bend left = 45] (-1,-0.2);
\end{tikzpicture}
\end{center}
\caption{Branched clique tree associated to the sparse set (\ref{sparsex2}).}
\label{fig:6dcl}
\end{figure}

Our method consists of conveniently rooting the clique tree and splitting the volume computation problem into lower-dimensional subproblems that are in correspondence with the {maximal} cliques of the graph. The subproblem associated  with a {maximal} clique $C$ takes as only input the solutions of the subproblems associated with the children of $C$ in the clique tree. This way, one can compute in parallel the solutions of all the subproblems of a given generation, and then use their results to solve the subproblems of the parent generation. This is the meaning of the arrows in Figures \ref{fig:4dcl} and \ref{fig:6dcl}. The volume of $\mbK$ is the optimal value of the (last) sub-problem associated with the root $C_1$ of the tree.

\section{Linear sparse volume computation}

In this section we describe the method in the prototype case of linear clique trees. The more general case of branched clique trees is treated later on.

\label{LCT}

\subsection{An illustrative example: the bicylinder}\label{bicylinder}

\begin{figure}[!h]
\begin{center}
\includegraphics[scale=0.4]{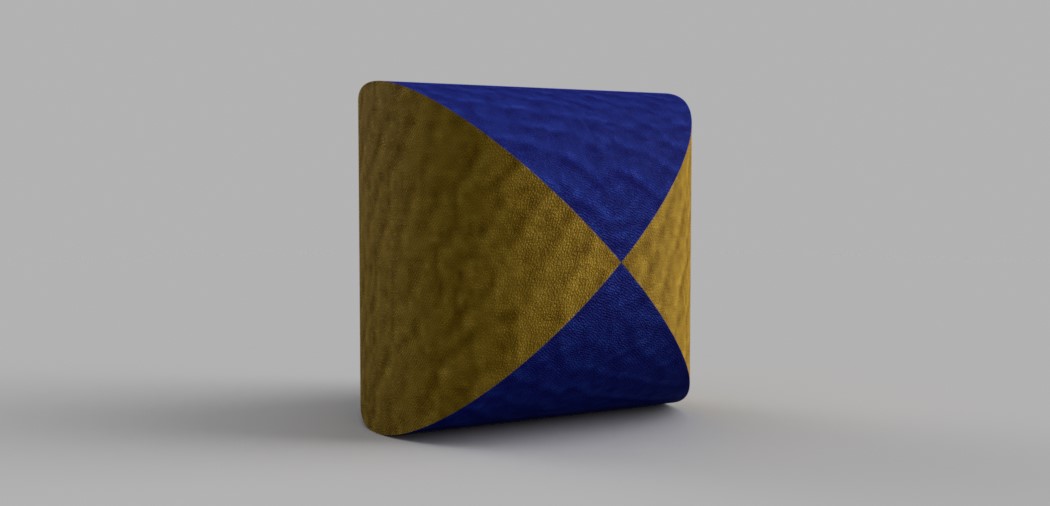}
\caption{A representation of the bicylinder produced by the AutoDesk Fusion 360 software.\label{fig:bicylinder}}
\end{center}
\end{figure}

Before describing the methodology in the general case, we briefly explain the general underlying idea on a simple illustrative example.
Consider the sparse semi-algebraic set
\begin{equation}
\label{setK-sparse}
\mbK := \left\{\mbx \in \R^3 \;  : \; \mbg_1(x_1,x_2):=1-x^2_1-x^2_2 \geq 0, \:  \mbg_2(x_2,x_3):=1-x^2_2-x^2_3 \geq 0\right\}
\end{equation}
modelling the intersection of two truncated cylinders $\mbK_1 := \{\mbx \in \R^3 \;  : \; x_1^2 + x_2^2 \leq 1\}$ and $\mbK_2 := \{\mbx \in \R^3 \;  : \: x_2^2 + x_3^2 \leq 1\}$, see Figure \ref{bicylinder}. The subspaces are $\mbX_1 = \langle x_1,x_2 \rangle$ and $\mbX_2 = \langle x_2,x_3 \rangle$ and the projection maps are $\pi_{\mbX_1}(\mbx)=(x_1,x_2)$ and $\pi_{\mbX_2}(\mbx)=(x_2,x_3)$. Let $\mbU_i := \pi_{\mbX_i}(\mbK_i)$ for $i = 1,2$.

Following \cite{sirev}, computing $\vol\:\mbK$ is equivalent 
to solving the infinite-dimensional LP \eqref{general-lp}.
Next observe that in the description (\ref{setK-sparse}) of $\mbK$ there is no direct interaction between
 variables $x_1$ and $x_3$, but this is neither exploited
 in the LP formulation (\ref{general-lp}) nor in the 
 SDP relaxations \eqref{sdp-primal} to solve (\ref{general-lp}).
 To exploit this {correlative} sparsity pattern we propose the following alternative formulation
\begin{eqnarray}
\vol\:\mbK &=& \max\limits_{\substack{\mu_i \in \mcM_+(\mbX_i)\\ i=1,2}} \int_{\R^2} d\mu_1 \label{sparpb} \\
&\text{s.t.}& \mu_2 \leq \lambda \otimes \lambda \nonumber \\
&& \mu_1 \leq \lambda \otimes \mu^{\langle x_2\rangle}_2 \nonumber \\
&& \spt\:\mu_1 \subset \mbU_1,\:\spt\:\mu_2 \subset \mbU_2 \nonumber
\end{eqnarray}
where $\mu_2^{\langle x_2\rangle}$ denotes the marginal of $\mu_2$ in the variable $x_2$.

In the sparse case, the basic idea behind our reformulation of the volume problem is as follows. We are interested in $\vol\:\mbK$. 
However, as the marginal of a measure has the same mass as the measure itself, instead of looking for the full measure $\mu$ in problem \eqref{general-lp}, we only look for its marginal on $\mbX_1$.

This marginal $\mu^{\mbX_1}$ is modeled by $\mu_1$ in \eqref{sparpb}.
In order to compute it, we need some additional information on $\mu$ captured by the measure $\mu_2$ in \eqref{sparpb}.  The unique optimal solution $\mu$ of \eqref{general-lp} is
\begin{eqnarray}
d\mu(\mbx) &=& d\lambda_\mbK(\mbx) \nonumber \\
&=& \ind_{\mbU_1}(x_1,x_2) \; \ind_{\mbU_2}(x_2,x_3) \; d\mbx \nonumber
\end{eqnarray}
and therefore its marginal $\mu_1:=\mu^{\mbX_1}$ on $(x_1,x_2)$ is
\begin{eqnarray}
d\mu_1(x_1,x_2) &=& \int_0^1 d\mu(x_1,x_2,x_3) \nonumber \\
&=& \ind_{\mbU_1}(x_1,x_2) \; dx_1 \underbrace{\left( \int_0^1 \ind_{\mbU_2}(x_2,x_3) \; dx_3 \right) dx_2}_{d\mu^{\langle x_2\rangle}_2(x_2)} \label{mu}
\end{eqnarray}
where
\begin{eqnarray}
d\mu_2(x_2,x_3) &=& d\lambda_{\mbU_2}(x_2,x_3). \label{nu}
\end{eqnarray}

What is the gain in solving (\ref{sparpb}) when compared to solving (\ref{general-lp}) ?
Observe that in (\ref{sparpb}) we have two unknown measures $\mu_1$ and $\mu_2$
on $\R^2$, instead of a single measure $\mu$ on
$\R^3$ in (\ref{general-lp}). In the resulting SDP relaxations
associated with (\ref{sparpb}) this translates into SDP constraints
of potentially much smaller size. For instance, and to fix ideas, for the same relaxation degree $d$:
\begin{itemize}
\item The SDP relaxation $\mbP_d$ associated with (\ref{general-lp}) contains a moment matrix 
(associated with $\mu$ in (\ref{general-lp})) of size ${3 +d\choose d}$;
\item
The SDP relaxation $\mbP_d$ associated with (\ref{sparpb}) contains two
moment matrices, one associated with $\mu_1$
of size ${2+d\choose d}$, and
one associated with $\mu_2$
of size ${2+d\choose d}$, where $\mu_1$ and $\mu_2$ are as in (\ref{sparpb}).
\end{itemize}
As the size of those matrices is the crucial parameter for all SDP solvers,
one can immediately appreciate the computational gain that can be expected from the
formulation (\ref{sparpb})  versus the formulation (\ref{general-lp}) when the dimension is high or the relaxation order increases. 
Next it is not difficult to extrapolate that the gain can be even more impressive
in the case where the {correlative} sparsity pattern is of the form
\begin{equation}
\label{def-K-multiple}
\mbK=\{(\mbx_0,\ldots\mbx_m)\in \mbX : \:g_i(\mbx_{i-1},\mbx_i)\geq0,\quad i=1,\ldots,m\,\},
\end{equation}
with $\mbx_i\in\R^{n_i}$ and $n_i \ll n$ for $i=0,\ldots,m$. In fact,
it is straightforward to define examples of sets $\mbK$ of the form (\ref{def-K-multiple}) where 
the first SDP relaxation associated with the original dense LP formulation (\ref{general-lp}) cannot be even implemented
on state-of-the-art computers, whereas the SDP relaxations associated with a generalization of the sparse LP formulation (\ref{sparpb}) can be easily implemented, at least for reasonable values of $d$.

\subsection{Linear computation theorem}

Let
\[
\mbK_i := \{ \mbx \in \mbX \; : \; \mbg_i(\pi_{\mbX_i}(\mbx)) \geq 0 \}
\]
with $\mbg_i \in \R[\mbx_i]^{p_i}$, so that our sparse semi-algebraic set can be written
\[
\mbK = \bigcap\limits_{i=1}^m\mbK_i.
\]
Moreover, let
\[
\mbU_i := \{\mbx_i \in \mbX_i \; : \; \mbg_i(\mbx_i) \geq 0\} = \pi_{\mbX_i}(\mbK_i)
\]
and let
\[
\mbY_i := \mbX_i \cap \mbX_{i+1}^\perp {= \langle x_j \rangle_{j \in C_i\cap C_{i+1}^c}}
\] 
be a subspace of dimension $n_i { := |C_i \cap C_{i+1}^c| }$ for $i=1,\ldots,m-1$ with $\mbY_m = \mbX_m$. The superscript $\perp$ denotes the orthogonal complement.

\begin{asm}
\label{ass-1}
For all $i \in \{2,\ldots,m\}$ it holds
$\mbX_i \cap \sum\limits_{j=1}^{i-1}\mbX_j \neq \{0\}$.
\end{asm}
\begin{asm}
\label{ass-1b}
For all $i \in \{2,\ldots,m\}$ it holds
$\mbX_i \cap \sum\limits_{j=1}^{i-1}\mbX_j \subset \mbX_{i-1}$.
\end{asm}
If Assumption \ref{ass-1} is violated then $\mbK$ can be decomposed as a Cartesian product, and one should just apply our methodology to each one of its factors. Assumption \ref{ass-1b} ensures that the associated clique tree is linear.

\begin{thm} \label{linco}
If Assumptions \ref{ass-1} and \ref{ass-1b} hold, then $\vol\:\mbK$ is the value of the LP problem
\begin{eqnarray}
&\max\limits_{\substack{\mu_i \in \mcM_+(\mbX_i) \\ i=1,\ldots,m}} & \displaystyle\int d\mu_1 \label{lincomp} \\
&\text{s.t.} & \mu_i \leq \mu_{i+1}^{\mbX_i \cap \mbX_{i+1}} \otimes \lambda^{n_i} \quad i=1,\ldots,m-1 \label{lindom} \\
&& \mu_{m} \leq \lambda^{n_m}\\
&& \spt\:\mu_{i} \subset \mbU_i \quad\quad\quad\quad\;\;\, i=1,\ldots,m. \label{linspt}
\end{eqnarray}
\end{thm}
\begin{proof} 
Let us first prove that the value of the LP is larger than $\vol\:\mbK$. For $i=1,\ldots,m$, let \mbox{$\mbZ_i := \mbX_i^\perp \cap \sum\limits_{j=i+1}^m \mbX_j$} so that $\sum\limits_{j=i}^m \mbX_j = \mbX_i \oplus \mbZ_i$. For $\mbx_i  \in \mbX_i$ define
\begin{equation} \label{eqproof}
d\mu_i(\mbx_i) := \ind_{\mbU_i}(\mbx_i) \left(\int_{\mbZ_i} \prod\limits_{j = i+1}^m \ind_{\mbU_j}\circ\pi_{\mbX_j}(\mbx_i + \mbz_i) \; d \mbz_i \right) d \mbx_i.
\end{equation}
By construction $\mu_i \in \mcM_+(\mbX_i)$ and constraints \eqref{linspt} are enforced. In addition, one can check that, if $\mbx_{i,i+1} \in \mbX_i \cap \mbX_{i+1}$, then
\begin{eqnarray*}
d\mu_{i+1}^{\mbX_i \cap \mbX_{i+1}}(\mbx_{i,i+1}) &\stackrel{\text{def}}{=}& \int_{{\mby_{i,i+1}} \in \mbX_i^\perp \cap \mbX_{i+1}} d\mu_{i+1}(\mbx_{i,i+1} + {\mby_{i,i+1}}) \\
&\stackrel{\eqref{eqproof}}{=}& \left( \int_{\mbX_i^\perp \cap \mbX_{i+1}} \ind_{\mbU_{i+1}}(\mbx_{i,i+1}+{\mby_{i,i+1}}) \right. \\
&& \left. \quad \left( \int_{\mbZ_{i+1}} \prod\limits_{j=i+2}^m \ind_{\mbU_j}\circ\pi_{\mbX_j}(\mbx_{i,i+1}+{\mby_{i,i+1}}+\mbz_{i+1}) \; d\mbz_{i+1}\right) d{\mby_{i,i+1}} \right) d\mbx_{i,i+1} \\
&=& \left(\int_{\mbZ_i} \prod\limits_{j = i+1}^m \ind_{\mbU_j}\circ\pi_{\mbX_j}(\mbx_{i,i+1} + \mbz_i) \; d \mbz_i \right) d \mbx_{i,i+1}
\end{eqnarray*}
since $(\mbX_i^\perp \cap \mbX_{i+1}) \oplus \mbZ_{i+1} = \mbZ_i$.

Thus, constraints \eqref{lindom} are satisfied. Moreover, they are saturated on $\mbU_i$. Eventually, one has
$$ \mbX = \mbX_1 \oplus \mbZ_1 $$
and thus
\begin{eqnarray*} \int_{\mbX_1} d\mu_1 &=& \int_{\mbX_1} \ind_{\mbU_1}(\mbx_1) \left(\int_{\mbZ_1} \prod\limits_{j = 2}^m \ind_{\mbU_j}\circ\pi_{\mbX_j}(\mbx_1 + \mbz_1) \; d \mbz_1 \right) d \mbx_1 \\
&=& \int_\mbX \left( \prod\limits_{j=1}^m \ind_{\mbU_j}\circ\pi_{\mbX_j}(\mbx) \right) d \mbx \\
&=& \int_\mbX \ind_\mbK(\mbx) \; d \mbx \\
&=& \vol\:\mbK,
\end{eqnarray*}
that is, we have just proved that the value of the LP is larger than or equal to $\vol\:\mbK$.

To prove the converse inequality, observe that our previous choice $\mu_1,\ldots,\mu_m$ saturates the constraints \eqref{lindom} while enforcing the constraints \eqref{linspt}. Any other feasible solution $\tilde{\mu}_1,\ldots,\tilde{\mu}_m$ directly satisfies the inequality $\tilde{\mu}_i \leq \mu_i$.
In particular, $\tilde{\mu}_1 \leq \mu_1$ and thus
\[
\int  d\tilde{\mu}_1(\mbx_1) \leq \vol\:\mbK.
\]
\end{proof}

\begin{rqe} \label{dual}
The dual of the LP problem of Theorem \ref{linco} is the LP problem
\begin{equation} \label{dualeq}
\begin{array}{rll}
\inf\limits_{\substack{v_i \in \mcC_+(\mbX_i) \\ i=1,\ldots,m}} & \displaystyle \int_{\mbX_m} v_m(\mbx_m)\; d\mbx_m \\
\text{s.t.}& v_1(\mbx_1) \geq 1 & \forall\:\mbx_1 \in \mbU_1\\
& v_{i+1}(\mbx_{i+1}) \geq \displaystyle \int_{\mbY_i} v_{i}(\mby_i,\pi_{\mbX_i}(\mbx_{i+1})) \; d \mby_i & \forall \: \mbx_{i+1} \in \mbU_{i+1},\: i=1,\ldots,m-1.\\
\end{array}
\end{equation}
According to \cite{lass-book1}, there is no duality gap, i.e. the value of the dual is $\vol\:\mbK$. For example, in the case of the bicylinder treated in Section \ref{bicylinder}, the dual 
reads:
$$\begin{array}{rll}
\inf\limits_{\substack{v_1,\:v_2 \in \mcC_+([0,1]^2)}} & \displaystyle \int_0^1 \int_0^1 v_2(x_2,x_3) \; dx_2 dx_3 & \\
\text{s.t.} & v_1(x_1,x_2) \geq 1 & \forall\:(x_1,x_2)\in\mbU_1\\
& v_2(x_2,x_3) \geq \displaystyle \int_0^1 v_1(x_1,x_2) \; dx_1 \quad & \forall\: (x_2,x_3) \in \mbU_2.
\end{array}$$
Thus, if $(v_1^k,v_2^k)_{k\in\N}$ is a minimizing sequence for this dual LP, then the sets {
$$\mbA^k := \left\{(x_1,x_2,x_3) \in [0,1]^3 \; : \; v_1^k(x_1,x_2) \geq 1 \ , \ v_2^k(x_2,x_3) \geq \int_0^1 v_1^k(x,x_2) dx \right\}$$}
are outer approximations of the set $\mbK$ and the sequences $(\vol\:\mbA^k)_k$ { and $(\int v_2^k \; d\lambda^2)$} decrease to $\vol\:\mbK$. Similar statements can be made for the general dual problem.
\end{rqe}

\begin{cor} \label{lincv}
\textbf{\emph{[Convergence of the linear computation scheme]}}\\
The Moment-SOS hierarchy associated to problem \eqref{lincomp} converges to $\vol \: \mbK$.
\end{cor}
\begin{proof}
We know that $\vol \: \mbK$ is the value of the infinite dimensional primal LP on measures \eqref{lincomp}. The absence of duality gap (which is a consequence of the strong duality property in \cite{sirev}) implies that there exists a sequence of continuous maps $(v_1^{(k)},\dots,v_m^{(k)})_{k\in\N}$ feasible for its dual LP on functions \eqref{dualeq} and such that $\int_{\mbX_m} v_m^{(k)}(\mbx_m) \; d\mbx_m \underset{k\to\infty}{\longrightarrow} \vol \ \mbK$. The Stone-Weierstrass Theorem allows to replace these continuous maps with vectors of polynomials while keeping the convergence of the optimization criterion. Eventually, Putinar's Positivstellensatz allows to replace the positivity constraints with SOS constraints while keeping the convergence of the optimization criterion. Thus, there is no relaxation gap between the dual \eqref{dualeq} and its SOS reinforcement. Since the Moment-SOS hierarchy only consists of restricting the feasible set to bounded degree to enforce finite dimension, we can conclude that it converges to $\vol \ \mbK$ when the upper bound on the degree tends to infinity.
\end{proof}
\color{black}
\begin{rqe}
The LP \eqref{lincomp}-\eqref{linspt} is formulated as a single problem on $m$ unknown measures. However, it is possible to split it in small chained subproblems to be solved in sequence. Each subproblem is associated with a {maximal} clique (in the linear clique tree)   
and it takes as input the results of the subproblem associated with its parent clique. This way, the sparse volume computation is split into $m$ linked low-dimension problems and solved sequentially. This may prove useful when $m$ is large because when solving
the SDP relaxations associated with the single LP \eqref{lincomp}-\eqref{linspt},
the SDP solver may encounter difficulties in handling a high number of measures 
simultaneously. It should be easier to sequentially solve a high number of low-dimensional problems with only one unknown measure. {Both formulations being strictly equivalent, this would not change the convergence properties of the sparse scheme.}
\end{rqe}

\subsection{Lower bounds for the volume}

As explained in the introduction, the hierarchy of SDP relaxations associated with our infinite-dimensional LP provides us with a sequence of upper bounds on $\vol\:\mbK$. One may also be interested in computing lower bounds on $\vol\:\mbK$.
In principle it suffices to apply the same methodology and 
approximate  from above the volume of $\mbB \setminus \mbK$ since $\mbK$ is included in the unit box $\mbB$. However, it is unclear whether $\mbB \setminus \mbK$ has also a sparse description. We show that this is actually the case and so one may exploit {correlative} sparsity to compute lower bounds 
although it is more technical. The following result is a consequence of Theorem \ref{linco}:

 \begin{cor}
If $\mbK$ is sparse, then $\widehat{\mbK} := \mbB \setminus \mbK$ is sparse as well, and
$\vol\:\widehat{\mbK}$ is the value of the LP problem
\[
\begin{array}{rlll}
& \max\limits_{\substack{\mu_{i,j} \in \mcM_+(\mbX_j) \\ 1 \leq i \leq j \leq m}} & \displaystyle\sum\limits_{j=1}^m \int_{\mbX_j} d\mu_{1,j} \\
&\text{s.t.}& \mu_{j,j} \leq \lambda^{m_j} \\
&& \mu_{i,j} \leq \mu_{i+1,j}^{\mbX_i\cap\mbX_{i+1}} \otimes \lambda^{n_{i}} &  i=1,\ldots,j-1 \\
&& \spt\:\mu_{i,j}  \subset \mbU_i & i=1,\ldots j-1 \\
&& \spt\:\mu_{j,j} \subset \mathrm{cl}\:\widehat{\mbU}_j
\end{array}
\]
where $m_j := \dim \mbX_j$, $n_i := \dim \mbX_{i+1}^\perp \cap \mbX_i$, $\widehat{\mbU}_j:=[0,1]^{m_j} \setminus \mbU_j$ is open and $\mathrm{cl}\:\widehat{\mbU}_j$ denotes its closure\footnote{This is necessary since the support of a measure is a closed set by definition.}. 
\end{cor}
\begin{proof}
The following description
$$\widehat{\mbK} = \bigsqcup\limits_{j=1}^m \left[ \bigcap\limits_{i=1}^{j-1} \mbU_i \cap \widehat{\mbU}_j \right] ,$$
where $\bigsqcup$ stands for disjoint union, is sparse. Indeed  the description of the basic semi-algebraic set
\[
\mbV_j := \bigcap\limits_{i=1}^{j-1} \mbU_i \cap \widehat{\mbU}_j
\]
is sparse. In addition, by $\sigma$-additivity of the Lebesgue measure, it holds
\[
\vol\:\widehat{\mbK} = \displaystyle\sum\limits_{j=1}^m \vol\:\mbV_j.
\]
Finally, by  using  Theorem \ref{linco}  we conclude that $\vol\:\mbV_j$ is the value of LP consisting of maximizing $\int_{\mbX_j} d\mu_{1,j}$ subject to the same constraints as in the above LP problem. Summing up yields the correct value.
\end{proof}

\section{Accelerating convergence}
\label{stokes}

As already mentioned, the convergence of the standard SDP relaxations 
\eqref{sdp-primal} for solving the GMP \eqref{general-lp} is expected to be slow in general.
To cope with this issue we introduce additional linear constraints that are redundant for the infinite dimensional GMP, and that are helpful to accelerate the convergence of the SDP relaxations. These constraints come from a specific application of Stokes' theorem.

\subsection{Dense Stokes constraints}

\label{fullstokes}

We first focus on the dense formulation \eqref{general-lp}. We know that the optimal measure of our infinite dimensional LP is $\mu = \lambda_\mbK$. Thus, one can put additional constraints in the hierarchy in order to give more information on the target sequence of moments, without increasing the dimension of the SDP relaxation. To keep the optimal value unchanged, such constraints should be redundant in the infinite dimensional LP. Ideally, we would like to characterize the whole set of polynomials $p$ such that
\[
\int_\mbK p(\mbx)\;d\mu(\mbx) = 0.
\]
Indeed, given any such polynomial $p$, the moments $\mbm$ of $\mu$ necessarily satisfy the linear constraint $L_\mbm(p) = 0$. However, for a general semi-algebraic set $\mbK$, we do not have an explicit description of this set of polynomials. Nevertheless, we can generate many of them, and hence improve convergence of the SDP relaxations significantly, as it was done originally in \cite{stocon} in another context. Let us explain how we generate these linear moment constraints.

We recall that Stokes' theorem states that if $\mcO$ is an open set of $\R^n$ with a boundary $\partial\mcO$ smooth almost everywhere, and $\omega$ is a $(n-1)$-differential form on $\R^n$, then one has
$$ \int_{\partial\mcO} \omega = \int_\mcO d \omega. $$
A corollary to this theorem is obtained by choosing $\omega(\mbx) = \mbh(\mbx) \cdot  \mbn(\mbx) \; \sigma(d\mbx)$, where the dot denotes the inner product between $\mbh$, a smooth (ideally polynomial) vector field, and $\mbn$ the outward pointing unit vector orthogonal to the boundary $\partial\mcO$, and $\sigma$ denotes the $(n-1)$-dimensional Hausdorff measure on $\partial\mcO$. In this case, one obtains the Gauss formula \cite{stokes}
$$ \int_{\partial\mcO} \mbh(\mbx)\cdot \mbn(\mbx)\;  \sigma(d\mbx) = \int_\mcO \dv\mbh(\mbx) \; d \mbx. $$
Choosing $\mbh(\mbx) = h(\mbx) \: \mbe_i$, where $h$ is a smooth function (ideally a polynomial) and $\mbe_i$ is the $i$-th vector of the canonical basis of $\R^n$, one obtains the following vector equality
$$ \int_{\partial\mcO} h(\mbx) \; \mbn(\mbx) \; d \sigma(\mbx) = \int_\mcO \grad h(\mbx) \; d \mbx. $$
Then, if we choose $\mcO = \mbK\setminus\partial\mbK$ and a polynomial $h$ vanishing on $\partial\mbK$, the vector constraint
\[
\int_\mbK \grad h(\mbx) \; d \mbx = 0
\]
is automatically satisfied and it can be added to the optimization problem \eqref{general-lp} without changing its optimal value. Such constraints are redundant in the infinite-dimensional LP formulation \eqref{general-lp} but not in the SDP relaxations \eqref{sdp-primal}. It has been numerically shown in \cite{stocon} that adding these constraints dramatically increases the convergence rate of the hierarchy of SDP relaxations.

These so-called \emph{Stokes constraints} can be added to the formulation of problem \eqref{general-lp} to yield
\begin{equation}\label{stokeslp}
\begin{array}{ll}
\displaystyle \max\limits_{\mu,\:\hat{\mu} \in \mcM_+(\mbB)} & \displaystyle \int d\mu  \\
\text{s.t.}& \mu + \hat{\mu} = \lambda_\mbB  \\
& (\grad h)\mu = \grad(h\mu)  \\
& \spt\:\mu \subset \mbK  \\
& \spt\:\hat{\mu} \subset \mbB
\end{array}
\end{equation}
where $h$ is any polynomial vanishing on the boundary of $\mbK$, without changing its value $\vol\:\mbK$.
The vector constraint $(\grad h)\mu = \grad(h\mu)$ should be understood in the sense of distributions, i.e. for all test functions $v \in {\mathscr C}^1(\mbB)$ it holds
\[
\int (\grad h) v \;d\mu = - \int (\grad v) h \; d\mu
\]
or equivalently
\[
\int \left((\grad h) v + (\grad v) h\right) d\mu = \int \grad(hv)\; d\mu = 0
\]
which becomes a linear moment constraint
\[
L_\mbm(\grad(hv)) = 0
\]
if $v$ is polynomial. In practice, when implementing the SDP relaxation of degree $d$, we choose $h(\mbx):=\prod_{j=1}^m g_j(\mbx)$ and $v(\mbx)=\mbx^{\alpha}$, $\alpha \in \N^n$, $|\alpha| \leq d+1-\sum_{j=1}^m\deg g_j$.

\begin{rqe} \label{rqstokes} The dual to the LP problem (\ref{stokeslp}) is the LP problem
\[
\begin{array}{rl}
\displaystyle \inf\limits_{\substack{v\in\mathscr{C}_+(\mbB) \\ \mbw\in\mathscr{C}(\mbB)^n}} & \displaystyle \int v\,d\lambda_\mbB \\
\mathrm{s.t.}  & v + \text{div}(h\mbw)\geq \ind_{\mbK}.
\end{array}
\]
It follows that the function $v$ is not required anymore to approximate from above the discontinuous indicator function $\ind_{\mbK}$, so that the Gibbs effect is reduced. {We believe that the infimum in this dual of Stokes is in fact a minimum, i.e. there exists optimal decision variables $v,\mbw$. This would make the Gibbs effect totally disappear. Proving this reduces to a problem of existence and uniqueness of the solution to a degenerate linear PDE, and it is out of the scope of this paper.}
\end{rqe}

\subsection{Sparse Stokes constraints{: the bicylinder}}

We have designed efficient Stokes constraints for the dense formulation of problem \eqref{general-lp}, at the price of introducing in problem \eqref{stokes} a polynomial $h$ vanishing on the boundary of $\mbK$. However, in the sparse case \eqref{sparpb}, the polynomial $h$ would destroy the sparsity structure, as it is the product of all polynomials defining $\mbK$. So we must adapt our strategy to introduce sparse Stokes constraints.

In this section, to keep the notations simple, we illustrate the ideas on our introductive bicylinder example of Section \ref{bicylinder}. Considering the optimal measures $\mu_1$ and $\mu_2$ defined in \eqref{mu},\eqref{nu}, we can apply Stokes constraints derived from the Gauss formula, in the directions in which they are Lebesgue: for $\mu_1$ in the $x_1$ direction and for $\mu_2$ in the remaining directions. To see this, define
{
\begin{eqnarray*}
\mbh_1(x_1,x_2) &=& g_1(x_1,x_2) \; \mbe_1,\\
\mbh_2(x_2,x_3) &=& g_2(x_2,x_3) \; \mbe_2,\\
\mbh_3(x_2,x_3) &=& g_2(x_2,x_3) \; \mbe_3
\end{eqnarray*}}
where $g_i(x_i,x_{i+1})=1-x_i^2-x_{i+1}^2$, such that $\mbh_1\cdot\mbn_{\mbU_1}$ vanishes on the boundary of $\mbU_1$ and $\mbh_2\cdot\mbn_{\mbU_2}$ and $\mbh_3\cdot\mbn_{\mbU_2}$ vanish on the boundary of $\mbU_2$, where $\mbn_{\mbU_i}$ is the outward point vector orthogonal to the boundary of  ${\mbU_i}$. For $i,j,k \in \N$, the Gauss formula yields
{
\begin{eqnarray*}
\int_{\mbU_1} \dfrac{\partial}{\partial x_1} (g_1(x_1,x_2)x_1^{i}x_2^{j}) \; d\mu_1= 0,\\
\int_{\mbU_2} \dfrac{\partial}{\partial x_2} (g_2(x_2,x_3)x_2^{j}x_3^{k}) \; d\mu_2 = 0,\\
\int_{\mbU_2} \dfrac{\partial}{\partial x_3} (g_2(x_2,x_3)x_2^{j}x_3^{k}) \; d\mu_2 = 0.\\
\end{eqnarray*}}
Hence, adding these constraints does not change the optimal value of the LP problem \eqref{sparpb}.

\subsection{General sparse Stokes constraints}

Consider the sequential decomposition of Theorem \ref{linco}:

\begin{eqnarray*}
&\max\limits_{\mu_i \in \mcM_+(\mbX_i)} & \displaystyle\int d\mu_i \label{lincompseq} \\
&\text{s.t.} & \mu_i \leq \mu_{i+1}^{\mbX_i \cap \mbX_{i+1}} \otimes \lambda^{n_i} \label{lindomseq} \\
&& \spt\:\mu_{i} \subset \mbU_i \label{linsptseq}
\end{eqnarray*}
for $1 \leq i \leq m-1$, and
\begin{eqnarray*}
&\max\limits_{\mu_m \in \mcM_+(\mbX_m)} & \displaystyle\int d\mu_m \\
&\text{s.t.} & \mu_{m} \leq \lambda^{n_m}\\
&& \spt\:\mu_{m} \subset \mbU_m.
\end{eqnarray*}

Our algorithm consists of sequentially solving these problems, starting with determining $\mu_m$, then $\mu_{m-1}$, and so on until $\mu_1$, whose mass will be $\vol(\mbK)$. We implement Stokes constraints on each one of these problems. For the problem in $\mu_m$, we implement regular Stokes constraints as in section \ref{fullstokes}:
\begin{eqnarray*}
&\max\limits_{\mu_m \in \mcM_+(\mbX_m)} & \displaystyle\int d\mu_m \\
&\text{s.t.} & \mu_{m} \leq \lambda^{n_m}\\
&& (\grad h_m)\mu_m = (\grad h_m\mu_m) \\
&& \spt\:\mu_{m} \subset \mbU_m
\end{eqnarray*}
where $h_m$ is a polynomial vanishing on $\partial \mbU_m$.

Then, let $i \in \{1 , \dots , m-1\}$ and suppose that $\mu_{i+1}$ is known, such that determining $\mu_i$ is reduced to a linear programming problem. From the arguments of Section \ref{LCT}, we know that the optimal measure $\mu_i$ is supported on $\mbU_i$ and that on this set it is the product measure between $\mu_{i+1}^{\mbX_i\cap\mbX_{i+1}}$ and the uniform measure on $\mbY_i = \langle x_j \rangle_{j \in C_i  \cap C_{i+1}^c} $. Since Stokes' theorem is only valid for uniform measures, it will only apply to $\mu_i^{\mbY_i}$.

Let $J \subset \{1,\dots,n\}$. For $f \in \mcC^1(\R^n)$ we define
$$\nabla_Jf := \left(\frac{\partial f}{\partial x_j}\right)_{j\in J}$$
such that $\nabla_{\{1,\dots,n\}}f = \grad f$ and $\nabla_{\{j\}}f = \frac{\partial f}{\partial x_j}$ for example. This notation allows us to define Stokes constraints exactly in the directions we are interested in and to formulate the general sparse Stokes constraints:
\begin{eqnarray*}
&\max\limits_{\mu_i \in \mcM_+(\mbX_i)} & \displaystyle\int d\mu_i \label{lincompseqstokes} \\
&\text{s.t.} & \mu_i \leq \mu_{i+1}^{\mbX_i \cap \mbX_{i+1}} \otimes \lambda^{n_i} \label{lindomseqstokes} \\
&& (\nabla_{C_i\cap C_{i+1}^c} h_i) \mu_i = (\nabla_{C_i\cap C_{i+1}^c} h_i \mu_i) \label{linstokeseq} \\
&& \spt\:\mu_{i} \subset \mbU_i \label{linsptseqstokes}
\end{eqnarray*}
where $h_i$ is a polynomial vanishing on $\partial \mbU_i$.

\begin{rqe} In some cases, in both dense and sparse contexts, these Stokes constraints can be slightly improved by choosing a different polynomial $h_j$ for each basis vector $\mbe_j$ when applying the Gauss formula, such that $h_j$ can be taken with the lowest possible degree, allowing for a better implementation of the hierarchy. For example, if one is looking for the volume of $\mbK := [0,1]^2$, the polynomial vanishing on $\partial\mbK$ with the lowest degree is $h(x_1,x_2) := x_1(1-x_1)x_2(1-x_2)$, but one can formulate Stokes constraints by applying the Gauss formula to $x_1(1-x_1) \mbe_1$ and $x_2(1-x_2) \mbe_2$, instead of $h(x_1,x_2) \mbe_1$ and $h(x_1,x_2) \mbe_2$. By doing so, one would replace the constraint $(\grad h)\mu = \grad(h\mu)$ with $\left(\frac{\partial h_j}{\partial x_j}\right)\mu = \frac{\partial}{\partial x_j}\left(h_j\mu\right)$ for every possible $j$. This is what we actually implemented in our numerical examples, but we presented the Stokes constraints in the restrictive case of $h_j = h$ for all $j$ for the sake of readability.
\end{rqe}

\section{Numerical examples}

\label{secex}

\color{black}

\subsection{Bicylinder revisited}

We refer to \eqref{general-lp} as the dense problem and to \eqref{sparpb} as the sparse problem. For both problems, we consider instances with and without additional Stokes constraints. 
Note that for the bicylinder example of Section \ref{bicylinder} the optimal value for both the dense and the sparse problem is
\[
\vol\:\mbK = \frac{16}{3} \approx 5.3333
\]
since adding Stokes constraints does not change the optimal value. 

We solve the SDP relaxations with Mosek on a standard laptop, for various relaxation orders and we report the bounds and the computation times in Table \ref{tab:ex1}.
\begin{table}[h!]
\begin{center}
\begin{tabular}{|r|l|l|l|l|}
\hline
	& \multicolumn{2}{c|}{full}& \multicolumn{2}{c|}{sparse}\\
$d$	& without Stokes		& with Stokes 	& without Stokes	& with Stokes\\
\hline
$2$& 7.8232 (1.0s)	& 5,828 (1.1s)	& 7,7424 (1.1s)	& 5,4984 (1.1s)\\
$3$& 7.2368 (0.9s)	& 5,4200 (1.3s)	& 7,1920 (0.9s)	& 5,3488 (1.1s)\\
$4$& 7.0496 (1.4s)	& 5,3520 (2.2s)	& 7,0040 (1.2s)	& 5,3376 (1.2s)\\
$5$& 6,8136 (3.1s)	& 5,3400 (4.4s)	& 6,7944 (1.8s)	& 5,3352 (1.8s)\\
$6$& 6,7376 (7.2s)	& 5,3376 (8.2s)	& 6,6960 (2.1s)	& 5,3344 (2.3s)\\
$7$& 6,6336 (12.8s)	& 5,3360 (18.3s) & 6,6168 (2.6s) & 5,3344 (3.2s)\\
\hline
\end{tabular}
\caption{Bounds on the volume (and computation times in seconds) vs relaxation order for the bicylinder.}
\label{tab:ex1}
\end{center}
\end{table}
We observe a slow convergence for the dense and the sparse versions without Stokes constraints, and a much faster convergence with Stokes constraints.  We also observe significantly smaller computation times when using the sparse formulation.

\subsection{A nonconvex set}

Let $\mbX := \R^5$, $\mbX_1 = \langle x_1,x_2\rangle$, $\mbX_2 = \langle x_1,x_3\rangle$, $\mbX_3 = \langle x_1,x_4\rangle$, $\mbX_4 = \langle x_1,x_5\rangle$ and
\begin{itemize}
\item $\mbg_i(x_1,x_{i+1}) := (2x_1^2 - x_{1+i}^2 - 1 \: , \: x_1\:(1-x_1) \: , \: x_{i+1}\:(1-x_{i+1})), \quad i = 1, \ldots, 4$
\item $\mbU_i := \mbg_i^{-1}\left((\R_+)^3\right) = \{(x_1,x_{i+1})\in [0,1]^2 : 2x_1^2 - x_{i+1}^2 \geq 1 \}$, $i = 1, \ldots, 4$.
\end{itemize}
Let us approximate the volume of the sparse set
$$ \mbK := \left\{(x_1,x_2,x_3,x_4,x_5) \in [0,1]^5 : 2x_1^2-x_{i+1}^2 \geq 1, \: i=1,\ldots,4\right\} = \bigcap\limits_{i=1}^4 \pi_{\mbX_i}^{-1}\left(\mbU_i\right).$$
Here the coordinates $x_2$, $x_3$, $x_4$ and $x_5$ do not interact: they are only linked with the coordinate $x_1$. The proper way to apply our linear computation Theorem \ref{linco} is to define a linear clique tree as shown in Figure \ref{nobranch}.

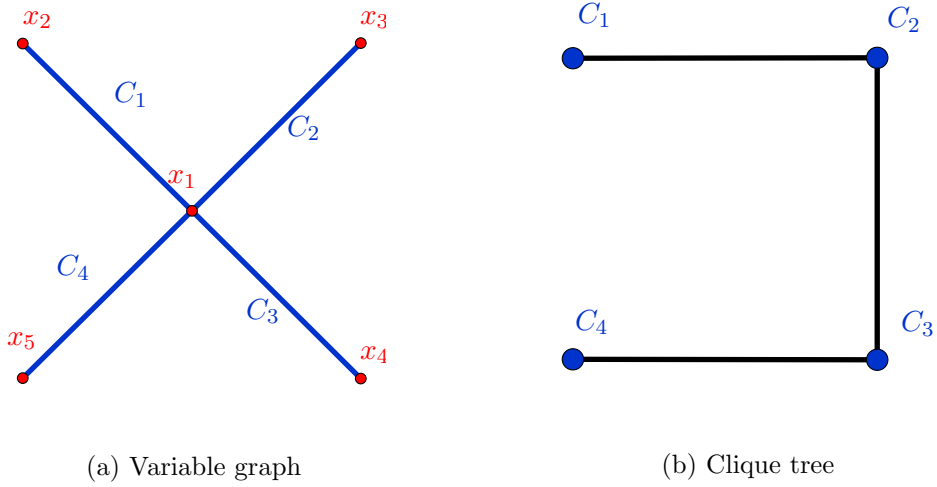
\begin{figure}[!h]
\begin{center}
\begin{subfigure}{0.45\textwidth}
\begin{center}
\begin{tikzpicture}[line cap=round,line join=round,>=triangle 45,x=1.0cm,y=1.0cm]
\clip(-2.5,-3.) rectangle (2.55,3.);
\draw [line width=2.pt,color=qqttcc] (0.,0.)-- (2.22096761177982,2.223980704890109);
\draw [line width=2.pt,color=qqttcc] (0.,0.)-- (2.22096761177982,-2.223980704890109);
\draw [line width=2.pt,color=qqttcc] (0.,0.)-- (-2.226989710267781,2.217950436475055);
\draw [line width=2.pt,color=qqttcc] (0.,0.)-- (-2.226989710267781,-2.217950436475055);
\draw [fill=ffqqqq] (0.,0.) circle (2.0pt);
\draw[color=ffqqqq] (-0.128576146772983,0.42019711072555577) node {$x_1$};
\draw [fill=ffqqqq] (2.22096761177982,2.223980704890109) circle (2.0pt);
\draw[color=ffqqqq] (2.405971057335158,2.529236390056418) node {$x_3$};
\draw [fill=ffqqqq] (2.22096761177982,-2.223980704890109) circle (2.0pt);
\draw[color=ffqqqq] (2.405971057335158,-1.9108463032717127) node {$x_4$};
\draw[color=qqttcc] (1.4717036572806974,1.0954596870025424) node {$C_2$};
\draw[color=qqttcc] (0.935193665170215,-1.3280854497723955) node {$C_3$};
\draw [fill=ffqqqq] (-2.226989710267781,2.217950436475055) circle (2.0pt);
\draw[color=ffqqqq] (-2.0341116359929727,2.529236390056418) node {$x_2$};
\draw [fill=ffqqqq] (-2.226989710267781,-2.217950436475055) circle (2.0pt);
\draw[color=ffqqqq] (-2.237615426103845,-1.725842857716374) node {$x_5$};
\draw[color=qqttcc] (-0.8038387230499695,1.5394679563353555) node {$C_1$};
\draw[color=qqttcc] (-1.5623528498268586,-0.7545747685508455) node {$C_4$};
\end{tikzpicture}
\caption{Variable graph}
\end{center}
\end{subfigure}
\begin{subfigure}{0.45\textwidth}
\begin{center}
\vspace*{-1.7em}
\begin{tikzpicture}[line cap=round,line join=round,>=triangle 45,x=1.0cm,y=1.0cm,scale=2]
\clip(-1.5,-1.5) rectangle (1.8,1.8);
\draw [line width=2.pt] (-1.00249256081015,0.9979729245624902)-- (0.9994817385567067,1.0009882851118512);
\draw [line width=2.pt] (0.9994817385567067,1.0009882851118512)-- (0.9994817385567067,-1.0009882851118512);
\draw [line width=2.pt] (0.9994817385567067,-1.0009882851118512)-- (-1.00249256081015,-0.9979729245624902);
\draw [fill=qqttcc] (-1.00249256081015,0.9979729245624902) circle (2.0pt);
\draw[color=qqttcc] (-0.8593397567165706,1.2839669226687537) node {$C_1$};
\draw [fill=qqttcc] (0.9994817385567067,1.0009882851118512) circle (2.0pt);
\draw[color=qqttcc] (1.175698144392156,1.2544645107800166) node {$C_2$};
\draw [fill=qqttcc] (0.9994817385567067,-1.0009882851118512) circle (2.0pt);
\draw[color=qqttcc] (1.2681998671698254,-0.762073045773176) node {$C_3$};
\draw [fill=qqttcc] (-1.00249256081015,-0.9979729245624902) circle (2.0pt);
\draw[color=qqttcc] (-0.8778401012721045,-0.7510709784399728) node {$C_4$};
\end{tikzpicture}
\caption{Clique tree}
\end{center}
\end{subfigure}
\caption{Graph with linear clique tree for the nonconvex set.}\label{nobranch}
\end{center}
\end{figure}
This yields the following formulation
\begin{eqnarray}
\vol \: \mbK &=& \max\limits_{\substack{\mu_i \in \mcM_+(\mbX_i) \\ i=1,\ldots,4}} \int_{\mbX_1} d \mu_1 \label{nobpb} \\
&\text{s.t.}& d\mu_1(x_1,x_2) \leq d\mu^{\langle x_1\rangle}_2(x_1) \; d x_2 \nonumber \\
&& d\mu_2(x_1,x_3) \leq d\mu^{\langle x_1\rangle}_3(x_1) \; d x_3 \nonumber \\
&& d\mu_3(x_1,x_4) \leq d\mu^{\langle x_1\rangle}_4(x_1) \; d x_4 \nonumber \\
&& d\mu_4(x_1,x_5) \leq d x_1 \; d x_5 \nonumber \\
&& \spt \; \mu_i \subset \mbU_i \quad\quad i=1,\ldots,4\nonumber 
\end{eqnarray}
with Stokes constraints
\small
\begin{eqnarray*}
\frac{\partial}{\partial x_2} \left[(2x_1^2-x_2^2-1)\: x_2\: (1-x_2) \vphantom{\sum} \right] \; d\mu_1(x_1,x_2) &=& \frac{\partial}{\partial x_2} \left[ (2x_1^2-x_2^2-1)\: x_2\: (1-x_2) \; d\mu_1(x_1,x_2) \right] \\
\frac{\partial}{\partial x_3} \left[(2x_1^2-x_3^2-1)\: x_3\: (1-x_3) \vphantom{\sum} \right] \; d\mu_2(x_1,x_3) &=& \frac{\partial}{\partial x_3} \left[ (2x_1^2-x_3^2-1)\: x_3\: (1-x_3) \; d\mu_2(x_1,x_3) \right] \\
\frac{\partial}{\partial x_4} \left[(2x_1^2-x_4^2-1)\: x_4\: (1-x_4) \vphantom{\sum} \right] \; d\mu_3(x_1,x_4) &=& \frac{\partial}{\partial x_4} \left[ (2x_1^2-x_4^2-1)\: x_4\: (1-x_4) \; d\mu_3(x_1,x_4) \right] \\
\frac{\partial}{\partial x_5} \left[(2x_1^2-x_5^2-1)\: x_5\: (1-x_5) \vphantom{\sum} \right] \; d\mu_4(x_1,x_5) &=& \frac{\partial}{\partial x_5} \left[ (2x_1^2-x_5^2-1)\: x_5\: (1-x_5) \; d\mu_4(x_1,x_5) \right] \\
\frac{\partial}{\partial x_1} \left[(2x_1^2-x_5^2-1)\: x_1\: (1-x_1) \vphantom{\sum} \right] \; d\mu_4(x_1,x_5) &=& \frac{\partial}{\partial x_1} \left[ (2x_1^2-x_5^2-1)\: x_1\: (1-x_1) \; d\mu_4(x_1,x_5) \right].
\end{eqnarray*}
\normalsize
We can compute analytically
$$ \vol \: \mbK = \frac{1}{15}\left(7-4\sqrt{2}\right) \simeq 0.0895.$$

On Figure \ref{fig:nonconvex} we show results from solving several relaxations via the dense and the sparse approach, with and without Stokes constraints. While solving with Mosek the degree $12$ dense relaxation  took about $1000$ seconds, solving the degree $12$ sparse relaxation  took less than $10$ seconds. With the sparse relaxations it was possible to go much higher in the hierarchy. Figure \ref{fig:nonconvex_val} shows convincingly how Stokes constraints accelerate the convergence of the hierarchy. We can also observe that  the nonconvexity of $\mbK$ poses no special difficulty for the volume computation.

\begin{figure}[!h]
\begin{subfigure}{0.5\textwidth}
\includegraphics[width=\textwidth]{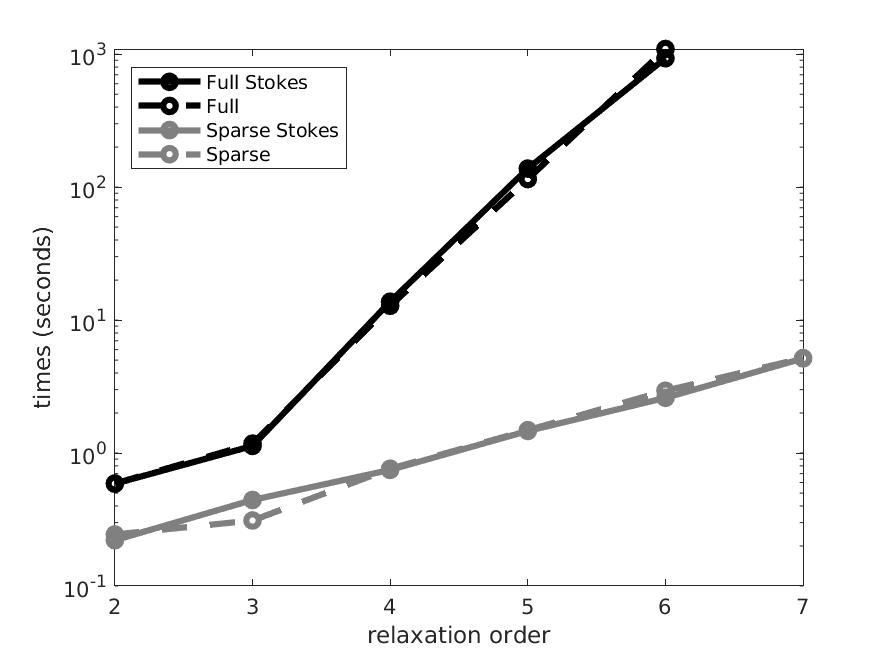}
\caption{Computation time vs relaxation order.}
\label{fig:nonconvex_tim}
\end{subfigure}
\hfill
\begin{subfigure}{0.5\textwidth}
\includegraphics[width=\textwidth]{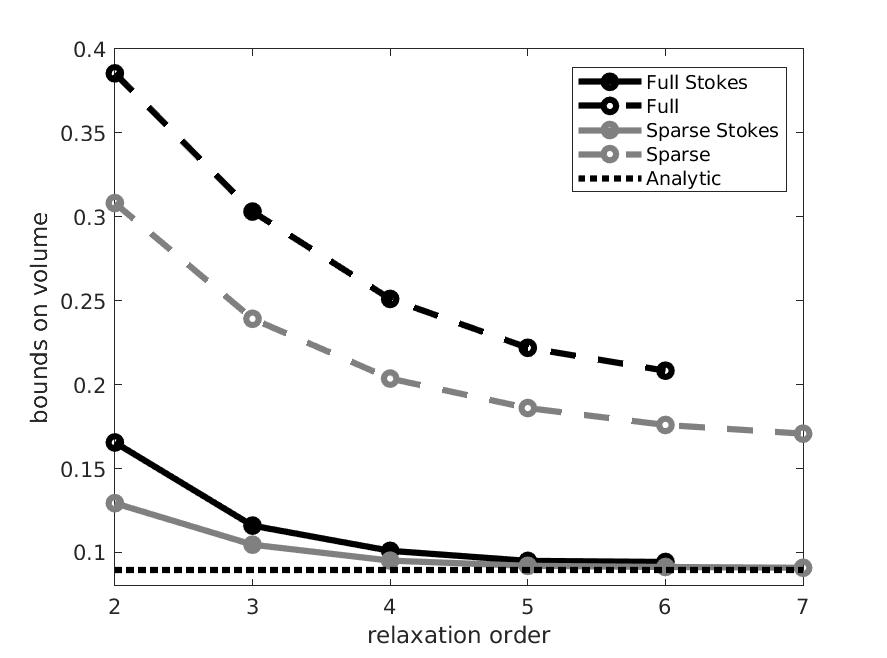}
\caption{Bounds on the volume vs relaxation order.}
\label{fig:nonconvex_val}
\end{subfigure}
\caption{Performance for the nonconvex set.}
\label{fig:nonconvex}
\end{figure}

\subsection{A high dimensional polytope}

Consider 
$$\mbK_n := \{\mbx \in [0,1]^n \; :\: x_i + x_{i+1}  \leq 1 , \, i=1,\ldots,n-1 \}.$$
According to \cite{polytope}, for any $\theta \in ]-\frac{\pi}{2}, \frac{\pi}{2}[$, one has the elegant formula :
$$ \tan \theta + \sec \theta = 1 + \sum_{n=1}^\infty \vol\:\mbK_n \: \theta^n $$
which allows to compute analytically the volume for $n$ arbitrarily large. For example when $n=20$ we obtain
\[
\vol\:\mbK_{20} = \frac{14814847529501}{97316080327065600} \approx 1.522 \cdot 10^{-4}.
\]

From the SDP viewpoint, $\vol \: \mbK_n$ is computed by solving relaxations of the LP problem given in Theorem \ref{linco} where $m = n-1$, $\mbX_i = \langle x_i,x_{i+1}\rangle$ and $\mbg_i(x_i,x_{i+1}) = (x_i,x_{i+1},1 - x_i - x_{i+1})$, $i=1,\ldots,n-1$.\\

We implemented the volume computation algorithm for $n = 20$, with Stokes contraints. This cannot be achieved without resorting to sparse computation as the dimension is too high for regular SDP solvers. With the sparse formulation however we could solve relaxations up to degree $28$ in less than $100$ seconds, see Figure \ref{fig:highdim}. Note however, that the analytic volume is of the order of $10^{-4}$. In consequence we observe a non monotonicity of the relaxation values which contradicts the theory. This issue is surprising as the Mosek SDP solver terminates without reporting issues. This indicates that computing small volumes in large dimension can be numerically sensitive.

\begin{figure}[!h]
\begin{center}
\includegraphics[width=0.5\textwidth]{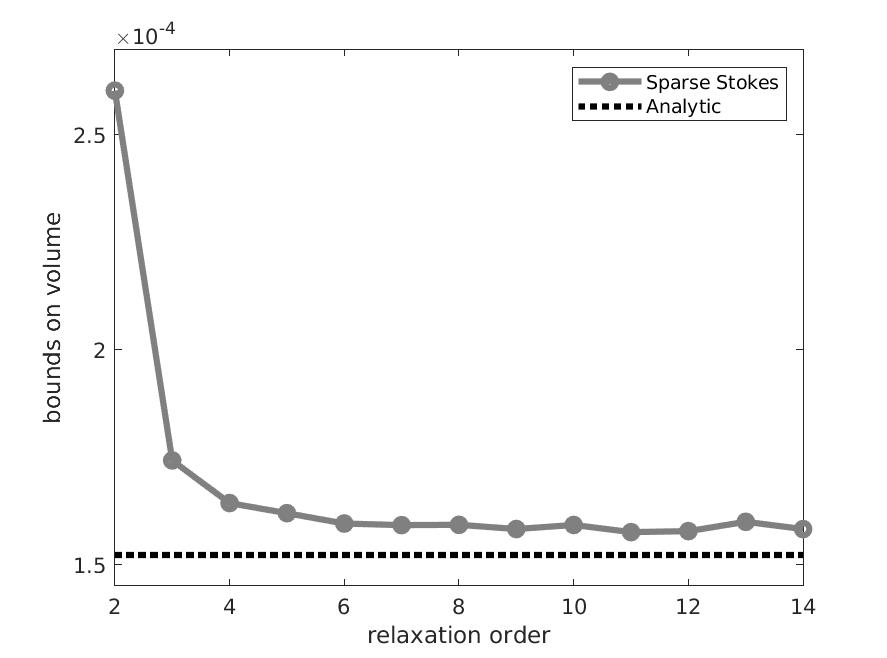}
\caption{Bounds on the volume vs relaxation order for the high dimensional polytope.}
\label{fig:highdim}
\end{center}
\end{figure}

In order to fix the monotonicity issue, we added a sparse rescaling to our problem. The idea is the following: at each step of the algorithm, the mass of the measure $\mu_i$ is less than the mass of the reference measure
\[
\rho_i := \mu^{\mbX_i \cap \mbX_{i+1}}_{i+1} \otimes \lambda^{n_i}.
\]
Defining
\[
\epsilon_i := \frac{\vert \mu_i \vert}{\vert \rho_i \vert} \in ]0,1[,
\]
we obtain that
\[
\vol \: \mbK = \prod\limits_{i=1}^m \epsilon_i
\]
as a telescoping product, since $\vert \rho_m \vert = \vol\:\mbB = 1$. As a result, if $m$ is large and the $\epsilon_i$ are small, one can expect the volume to be very small, which explains why the SDP solver encounters difficulties. Thus, a solution is to multiply each domination constraint by a well-chosen rescaling factor $\epsilon$ such that the mass of $\mu_i$ does not decrease too much. The resulting LP is as follows
\begin{eqnarray}
\vol\:\mbK &=& \epsilon^{m-1} \: \max\limits_{\substack{\mu_i \in \mcM_+(\mbX_i) \\ i=1,\ldots,m}} \int_{\mbX_1} d\mu_1 \label{hdlincomp} \\
&\text{s.t.}& \epsilon \: \mu_i \leq \mu_{i+1}^{\mbX_i \cap \mbX_{i+1}} \otimes \lambda^{n_i} \qquad i=1,\ldots,m-1 \nonumber \\
&& \mu_m \leq \lambda^{n_m} \nonumber \\
&& \spt \; \mu_i \subset \mbU_i \qquad\qquad\qquad\;\;\; i = 1, \ldots, m. \nonumber
\end{eqnarray}
Figure \ref{rescal} gives a comparison between the results obtained with and without sparse rescaling, using the SDP Solvers SeDuMi and Mosek, for the choice $\epsilon = \frac{1}{2}$.

\begin{figure}
\begin{subfigure}{0.5\textwidth}
\includegraphics[width=\textwidth]{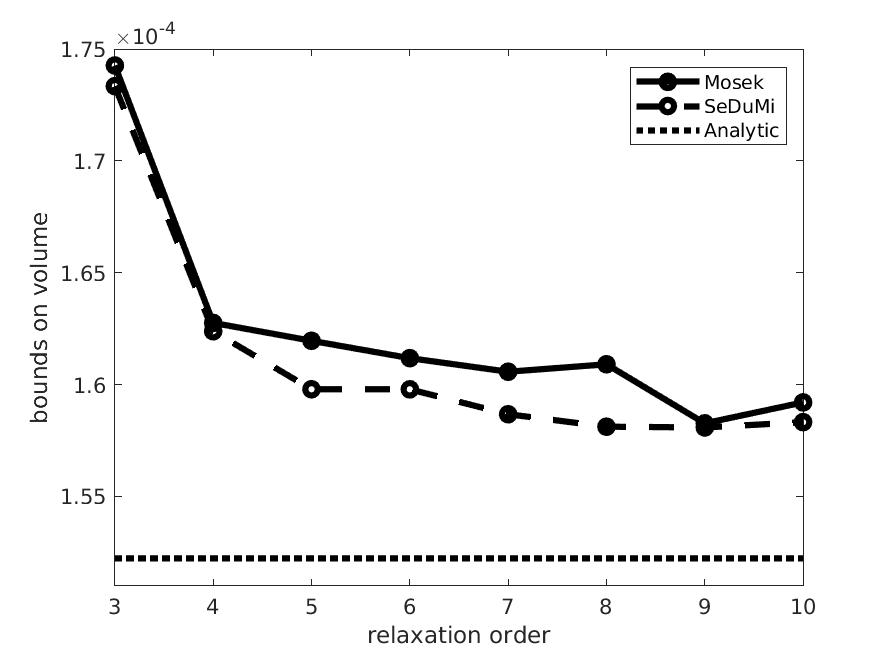}
\caption{Without rescaling.}
\label{fig:non_rescal}
\end{subfigure}
\hfill
\begin{subfigure}{0.5\textwidth}
\includegraphics[width=\textwidth]{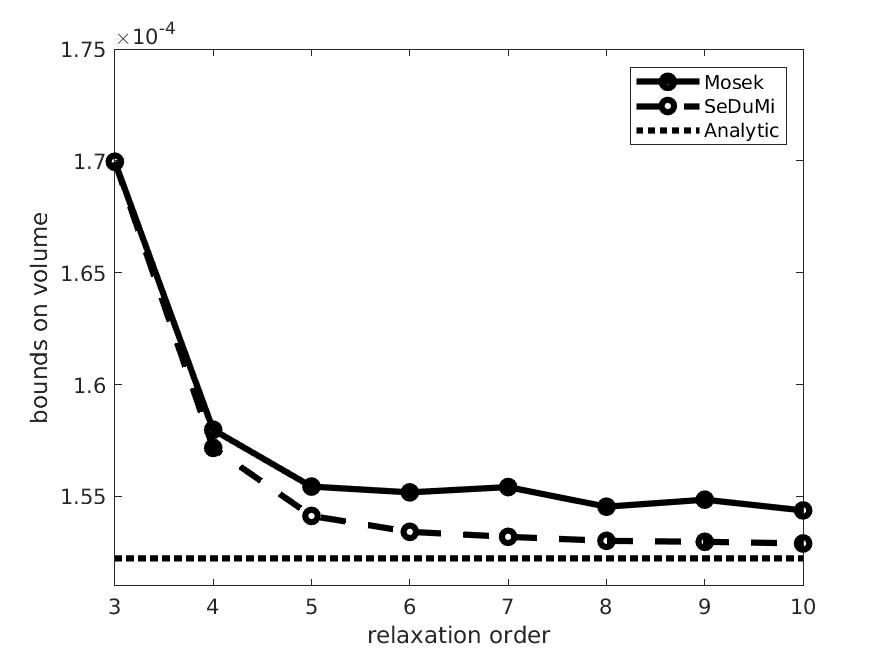}
\caption{With rescaling.}
\label{fig:rescal}
\end{subfigure}
\caption{Bounds on the volume vs relaxation order for the high dimensional polytope.}
\label{rescal}
\end{figure}

First, one can see that without rescaling (Figure \ref{fig:non_rescal}), both SeDuMi and Mosek have accuracy issues that make them lose monotonicity, while the rescaling (Figure \ref{fig:rescal}) allows to recover monotonicity at least when using SeDuMi (which is slower but more accurate than Mosek to our general experience). Second, it is clear that the relative approximation error is much smaller with scaling. This, combined to the fact that the error is relative (after rescaling, the error is much smaller), demonstrates the power of our rescaling method.

\subsection{A nonconvex high dimensional set}

\label{sec:highnonconvex}

Finally, we consider {the set already mentioned in the introduction,} which is both nonconvex and high dimensional. Let 
\[
\mbK_n := \left \{\mbx \in [0,1]^n \; : \; x_{i+1} \ x_{i} \leq 1/2,\:i=1,2,\ldots,n-1 \right \}
\]
whose analytic volume is a function of the dimension $n$. For $n = 3$ the analytic volume is $0.75$, for $n = 4$ it is $0.6566$, approximately. In higher dimensions we do not have an analytic expression for the volume. However, in order to get a feeling for its value for bigger $n$, we ran a Monte Carlo simulation\footnote{We provide a quick introduction in Appendix \ref{app:MC}.} with one million samples for $n = 10, 20, 50$, and $100$.

Before we go on, let us emphasize that the method proposed in this paper is not in concurrence with the Monte Carlo approach. While the Monte Carlo gives a \textit{probabilistic} estimate of the volume, our method provides a \textit{guaranteed} upper bound. Nonetheless, it would be concerning if the computed upper bound was much smaller than the confidence interval, and we consider our approximation valid, when it returns something in the order of the Monte Carlo approximation. The results for different dimensions $n$ and solved with the Mosek SDP solver are summarized in Table \ref{tab:highdimnonconvex}.
As in the previous section we experience accuracy issues for the relaxations of order $14$ and $n = 20, 100$, as well as for order $16$ and $n=50$. Otherwise, the approximations provide better upper bounds for increased relaxation orders as expected. 
For $n=3,4$ the approximation is reasonably close to the analytic value. For $n=10,20,50$ our scheme provides an upper bound in the same order of magnitude as the 99\%-confidence interval of the Monte Carlo simulation. We interpret this as a validation for both the Monte Carlo approach and our own approach. 
For $n = 100$ we could not derive a meaningful confidence interval. Indeed, as our approximation shows, the volume for $n=100$ is less than $9\cdot 10^{-6}$. In order to get an accuracy of $\varepsilon = 10^{-6}$ one would have to draw approximately $N = \frac{1}{\varepsilon^2} = 10^{12}$ samples. With our non-sophisticated implementation, the Monte Carlo simulation for one million points took about 5 seconds. Extending this linearly to a simulation with $10^{12}$ samples would therefore take a little less than 2 months ($5\cdot 10^6 \ s \simeq 1389 \ h \simeq 58 \ d$). With more sophisticated methods, this time could certainly be reduced dramatically. However, it sets the 44 minutes it took to solve relaxation order 16 for $n=100$ into perspective. 
    \begin{table}
\begin{center}
\begin{tabular}{|c|r|r|r|r|r|r|}
\hline
	& \multicolumn{2}{c|}{n=3}& \multicolumn{2}{c|}{n = 4}& \multicolumn{2}{c|}{n=10} \\
    $d$	& value & time (s)	& value & time (s)	& value & time (s)	\\
    \hline
    4   &	7.86E-01	&	 0.95	&	7.09E-01	&	 0.61	&	3.93E-01	&	 1.24	\\
    5   &	7.73E-01	&	 2.87	&	6.90E-01	&	 0.78	&	3.57E-01	&	 1.84	\\
    6   & 	7.69E-01	&	 2.74	&	6.84E-01	&	 0.86	&	3.45E-01	&	 4.21	\\
    7   &	7.66E-01	&	 4.58	&	6.79E-01	&	 1.83	&	3.38E-01	&	 4.55	\\
    8   &	7.63E-01	&	 5.00	&	6.77E-01	&	 2.29	&	3.34E-01	&	 5.97	\\
    9   &	7.63E-01	&	 6.11	&	6.74E-01	&	 3.33	&	3.30E-01	&	11.56	\\
    10  &	7.62E-01	&	 9.83	&	6.73E-01	&	 6.86	&	3.26E-01	&	18.21	\\
    11  &	7.61E-01	&	18.16	&	6.72E-01	&	 8.57	&	3.26E-01	&	22.24	\\
    12  &	7.60E-01	&	19.45	&	6.71E-01	&	10.43	&	3.23E-01	&	33.78	\\
    13  &	7.60E-01	&	22.49	&	6.70E-01	&	17.89	&	3.22E-01	&	74.00	\\
    14  &	7.60E-01	&	27.02	&	6.69E-01	&	26.84	&	3.21E-01	&	79.68	\\
    15  &	7.59E-01	&	32.90	&	6.69E-01	&	39.25	&	3.20E-01	&	119.7	\\
    16  &	7.58E-01	&	78.20	&	6.68E-01	&	61.32	&	3.19E-01	&	176.6	\\
    \hline
    ana/mc& 7.50E-01 & - & 6.57E-01 & - & \multicolumn{2}{c|}{[2.99e-01, 3.03e-01]} \\
    \hline
\end{tabular}
\begin{tabular}{|c|r|r|r|r|r|r|r|r|r|r|r|r|}
\hline
	& \multicolumn{2}{c|}{n = 20} & \multicolumn{2}{c|}{n=50}& \multicolumn{2}{c|}{n = 100}\\
    $d$	& value & time (s)	& value & time (s)	& value & time (s)	\\
    \hline
    4   &	1.47E-01	&	  5.11	&	7.68E-03	&	 10.64	&	9.49E-05	&	  15.19	\\
    5   &	1.20E-01	&	  3.58	&	4.78E-03	&	 15.10	&	4.80E-05	&	  26.87	\\
    6   & 	1.11E-01	&	  8.13	&	3.86E-03	&	 21.75	&	2.84E-05	&	  49.89	\\
    7   &	1.07E-01	&	 11.06	&	3.42E-03	&	 48.31	&	2.27E-05	&	  77.07	\\
    8   &	1.03E-01	&	 17.93	&	3.20E-03	&	 72.78	&	1.91E-05	&	 135.08	\\
    9   &	1.00E-01	&	 33.31	&	2.99E-03	&	120.49	&	1.63E-05	&	 202.12	\\
    10  &	9.81E-02	&	 41.02	&	2.99E-03	&	103.61	&	1.44E-05	&	 299.44	\\
    11  &	9.70E-02	&	 85.83	&	2.89E-03	&	165.56	&	1.22E-05	&	 441.67	\\
    12  &	9.59E-02	&	117.08	&	2.77E-03	&	220.84	&	1.19E-05	&	 623.24	\\
    13  &	9.51E-02	&	138.38	&	2.67E-03	&	314.00	&	1.08E-05	&	 850.92	\\
    14  &	9.57E-02	&	156.32	&	2.60E-03	&	457.92	&	1.10E-05	&	1175.02	\\
    15  &	9.39E-02	&	249.82	&	2.54E-03	&	685.64	&	9.86E-06	&	1589.49	\\
    16  &	9.36E-02	&	357.87	&	2.56E-03	&	859.60	&	9.46E-06	&	2623.02	\\
    \hline
    ana/mc& \multicolumn{2}{c|}{[8.09e-02, 8.24e-02]} & \multicolumn{2}{c|}{[1.48e-03,  1.68e-03]}& \multicolumn{2}{c|}{-}		\\
    \hline
\end{tabular}

    \caption{\textbf{A nonconvex high dimensional set:} ana/mc refers to the analytic value and the 99\%-confidence interval, respectively.}
    \label{tab:highdimnonconvex}
\end{center}
    \end{table}
\color{black}

\section{General sparse volume computation}

\label{DCT}

\subsection{General {correlative} sparsity pattern}

Let us describe a general method to compute the volume of $$\mbK := \bigcap\limits_{i=1}^m\mbK_i$$
{where $\mbK_i = \{\mbx \in \mbX : \mbg_i(\mbx) \geq 0\}$ and $(\mbg_1,\dots,\mbg_m)$ is a correlatively sparse family of polynomial vectors with associated coordinate subspace decomposition $\mbX = \sum_{i=1}^m \mbX_i$.}
For this we construct the correlation graph $G = (V,E)$ as follows:
\begin{itemize}
\item $V = \{1\dots,n\}$ represents the canonical basis $\{\mbe_1,\ldots,\mbe_n\}$ of $\mbX$;
\item $E = \{(i,j) \in \{1,\ldots,n\}^2\ \; : \; i \neq j \ \& \ \mbe_i,\mbe_j \in \mbX_k \;\text{for some}\; k \in \{1,\ldots,m\}\}$.
\end{itemize}
{As stated in Section \ref{cylinder}, we suppose that the correlation graph of $(\mbg_i)_{1\leq i\leq m}$ has exactly $m$ maximal cliques (see Section \ref{secdip} for discussions when it is not the case) that are in correspondence with the $\mbX_i$.}

Let $\mcK$ be the set of {maximal} cliques of $G$. We will use the following property of graphs:
\begin{defn}
\textbf{(CIP)}
The graph $G=(V,E)$ is said to satisfy the \emph{clique intersection property (CIP) iff} there is a clique tree $T=(\mcK,\mcE)$, such that for all $C,C' \in \mcK$, $C\cap C' \subset C''$ for any $C''$ on the path connecting $C$ and $C'$ in the tree $T$.
\end{defn}
Such a property ensures that $G$ is chordal\footnote{The CIP always holds, up to a chordal extension. In particular, cyclic graphs can be handled with empty interactions between well-chosen variables.}, see \cite{graphs}.
We then replace Assumption \ref{ass-1b} with the following strong {correlative} sparsity {assumption:

\begin{asm} \label{asm}
\emph{\textbf{(DIP)}} { We suppose that} there is a clique tree $T=(\mcK,\mcE)$, rooted in some $C_1$, {that simultaneously satisfies the CIP and the following \emph{disjoint intersection property (DIP):}} $\forall C,C',C'' \in \mcK$, 
if $(C,C') \in \mcE_T$ and $(C,C'') \in \mcE$ then $C' = C''$ or $C' \cap C'' = \varnothing$.
\end{asm}
In words, each clique has 
an empty intersection with all its siblings. See Appendix \ref{DIP} for details { on how to check this assumption and construct such a tree when it exists. See Section \ref{secdip} for possible solutions when Assumption \ref{asm} does not hold}.
Figure \ref{global} illustrates the meaning of this { assumption} for $n = 12$ and $m = 8$. One can check that Assumptions \ref{ass-1} and \ref{asm} hold.

\begin{rqe}
{With these assumptions}, the only possible clique trees for applying our method to the nonconvex example illustrated in Figure \ref{nobranch} are linear clique trees. Indeed, any branched clique tree would imply sibling cliques containing $x_1$.
\end{rqe}

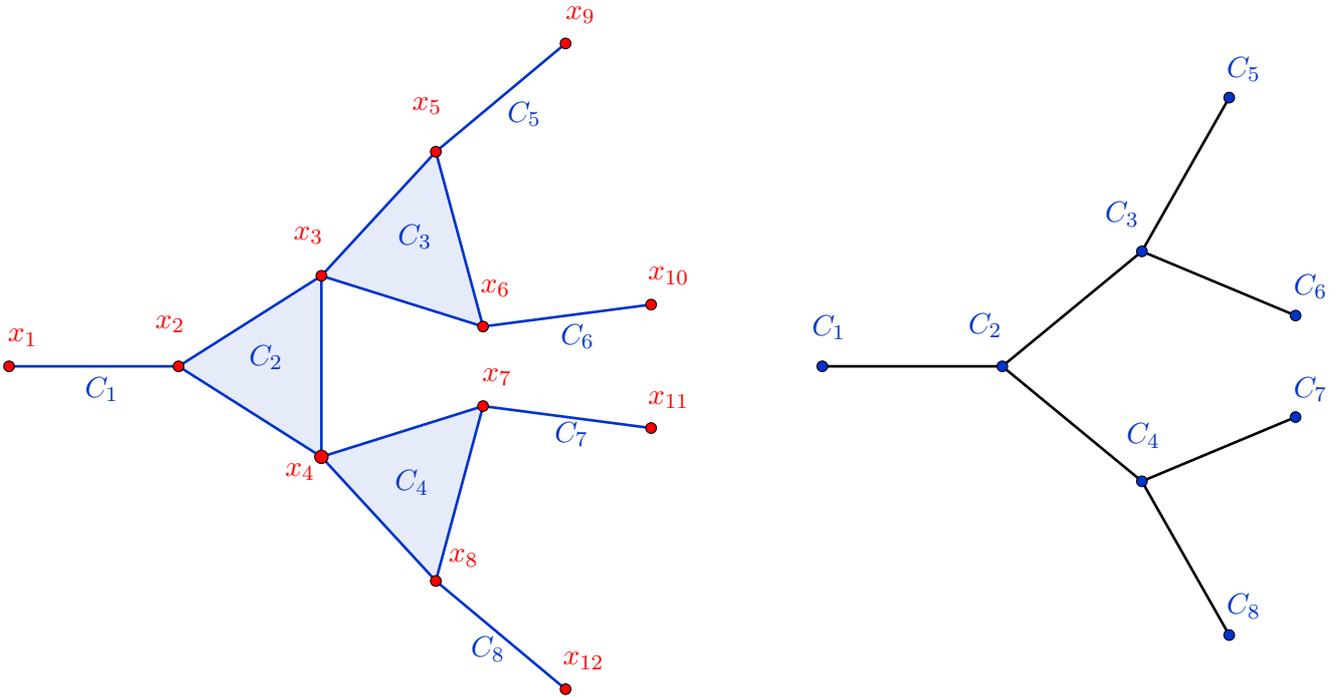
\begin{figure}[!h]
\hspace{-2cm}
\begin{tikzpicture}[line cap=round,line join=round,>=triangle 45,x=1.0cm,y=1.0cm]
\clip(-3.,-5.) rectangle (8.,5.);
\fill[line width=1.pt,color=qqttcc,fill=qqttcc,fill opacity=0.10000000149011612] (0.,0.) -- (1.88,1.2) -- (1.88,-1.2) -- cycle;
\fill[line width=1.pt,color=qqttcc,fill=qqttcc,fill opacity=0.10000000149011612] (1.88,1.2) -- (3.3853576993004233,2.845690796339621) -- (4.006523407546473,0.5274688132458569) -- cycle;
\fill[line width=1.pt,color=qqttcc,fill=qqttcc,fill opacity=0.10000000149011612] (4.006523407546473,-0.5274688132458569) -- (1.88,-1.2) -- (3.3853576993004233,-2.845690796339621) -- cycle;
\draw [line width=1.pt,color=qqttcc] (0.,0.)-- (1.88,1.2);
\draw [line width=1.pt,color=qqttcc] (1.88,1.2)-- (1.88,-1.2);
\draw [line width=1.pt,color=qqttcc] (1.88,-1.2)-- (0.,0.);
\draw [line width=1.pt,color=qqttcc] (1.88,1.2)-- (3.3853576993004233,2.845690796339621);
\draw [line width=1.pt,color=qqttcc] (3.3853576993004233,2.845690796339621)-- (4.006523407546473,0.5274688132458569);
\draw [line width=1.pt,color=qqttcc] (4.006523407546473,0.5274688132458569)-- (1.88,1.2);
\draw [line width=1.pt,color=qqttcc] (4.006523407546473,-0.5274688132458569)-- (1.88,-1.2);
\draw [line width=1.pt,color=qqttcc] (1.88,-1.2)-- (3.3853576993004233,-2.845690796339621);
\draw [line width=1.pt,color=qqttcc] (3.3853576993004233,-2.845690796339621)-- (4.006523407546473,-0.5274688132458569);
\draw [line width=1.pt,color=qqttcc] (3.3853576993004233,2.845690796339621)-- (5.0926426080191245,4.28081387136189);
\draw [line width=1.pt,color=qqttcc] (4.006523407546473,0.5274688132458569)-- (6.2177788688503375,0.8185861175302799);
\draw [line width=1.pt,color=qqttcc] (4.006523407546473,-0.5274688132458569)-- (6.2177788688503375,-0.8185861175302799);
\draw [line width=1.pt,color=qqttcc] (3.3853576993004233,-2.845690796339621)-- (5.0926426080191245,-4.28081387136189);
\draw [line width=1.pt,color=qqttcc] (-2.230336297512104,0.)-- (0.,0.);
\draw [fill=ffqqqq] (0.,0.) circle (2.0pt);
\draw[color=ffqqqq] (-0.11,0.56) node {$x_2$};
\draw [fill=ffqqqq] (1.88,1.2) circle (2.0pt);
\draw[color=ffqqqq] (1.71,1.74) node {$x_3$};
\draw [fill=ffqqqq] (1.88,-1.2) circle (2.5pt);
\draw[color=ffqqqq] (1.6,-1.4) node {$x_4$};
\draw[color=qqttcc] (1.15,0.12) node {$C_2$};
\draw [fill=ffqqqq] (3.3853576993004233,2.845690796339621) circle (2.0pt);
\draw[color=ffqqqq] (3.27,3.46) node {$x_5$};
\draw [fill=ffqqqq] (4.006523407546473,0.5274688132458569) circle (2.0pt);
\draw[color=ffqqqq] (4.17,1.04) node {$x_6$};
\draw [fill=ffqqqq] (4.006523407546473,-0.5274688132458569) circle (2.0pt);
\draw[color=ffqqqq] (4.19,-0.14) node {$x_7$};
\draw [fill=ffqqqq] (3.3853576993004233,-2.845690796339621) circle (2.0pt);
\draw[color=ffqqqq] (3.75,-2.54) node {$x_8$};
\draw[color=qqttcc] (3.11,1.72) node {$C_3$};
\draw[color=qqttcc] (3.07,-1.54) node {$C_4$};
\draw [fill=ffqqqq] (5.0926426080191245,4.28081387136189) circle (2.0pt);
\draw[color=ffqqqq] (5.29,4.66) node {$x_9$};
\draw [fill=ffqqqq] (6.2177788688503375,0.8185861175302799) circle (2.0pt);
\draw[color=ffqqqq] (6.45,1.2) node {$x_{10}$};
\draw [fill=ffqqqq] (6.2177788688503375,-0.8185861175302799) circle (2.0pt);
\draw[color=ffqqqq] (6.45,-0.44) node {$x_{11}$};
\draw [fill=ffqqqq] (5.0926426080191245,-4.28081387136189) circle (2.0pt);
\draw[color=ffqqqq] (5.33,-3.9) node {$x_{12}$};
\draw[color=qqttcc] (4.55,3.35) node {$C_5$};
\draw[color=qqttcc] (5.25,0.4) node {$C_6$};
\draw[color=qqttcc] (5.17,-0.9) node {$C_7$};
\draw[color=qqttcc] (4.07,-3.75) node {$C_8$};
\draw [fill=ffqqqq] (-2.230336297512104,0.) circle (2.0pt);
\draw[color=ffqqqq] (-2.05,0.38) node {$x_1$};
\draw[color=qqttcc] (-1.01,-0.3) node {$C_1$};
\end{tikzpicture}
\begin{tikzpicture}[line cap=round,line join=round,>=triangle 45,x=1.0cm,y=1.0cm]
\clip(-1.5,-5.) rectangle (8.,5.);
\draw [line width=1.pt] (-1.115168148756052,0.)-- (1.2533333333333332,0.);
\draw [line width=1.pt] (1.2533333333333332,0.)-- (3.090627035615632,1.5243865365284928);
\draw [line width=1.pt] (3.090627035615632,1.5243865365284928)-- (4.239000153659774,3.563252333850756);
\draw [line width=1.pt] (3.090627035615632,1.5243865365284928)-- (5.112151138198405,0.6730274653880683);
\draw [line width=1.pt] (1.2533333333333332,0.)-- (3.090627035615632,-1.5243865365284925);
\draw [line width=1.pt] (3.090627035615632,-1.5243865365284925)-- (5.112151138198405,-0.6730274653880683);
\draw [line width=1.pt] (3.090627035615632,-1.5243865365284925)-- (4.239000153659774,-3.563252333850756);
\draw [fill=qqttcc] (-1.115168148756052,0.) circle (2.0pt);
\draw[color=qqttcc] (-1.03,0.52) node {$C_1$};
\draw [fill=qqttcc] (4.239000153659774,3.563252333850756) circle (2.0pt);
\draw[color=qqttcc] (4.43,3.94) node {$C_5$};
\draw [fill=qqttcc] (5.112151138198405,0.6730274653880683) circle (2.0pt);
\draw[color=qqttcc] (5.31,1.06) node {$C_6$};
\draw [fill=qqttcc] (5.112151138198405,-0.6730274653880683) circle (2.0pt);
\draw[color=qqttcc] (5.31,-0.3) node {$C_7$};
\draw [fill=qqttcc] (4.239000153659774,-3.563252333850756) circle (2.0pt);
\draw[color=qqttcc] (4.43,-3.18) node {$C_8$};
\draw [fill=qqttcc] (1.2533333333333332,0.) circle (2.0pt);
\draw[color=qqttcc] (1.03,0.54) node {$C_2$};
\draw [fill=qqttcc] (3.090627035615632,1.5243865365284928) circle (2.0pt);
\draw[color=qqttcc] (2.83,2.02) node {$C_3$};
\draw [fill=qqttcc] (3.090627035615632,-1.5243865365284925) circle (2.0pt);
\draw[color=qqttcc] (3.11,-0.92) node {$C_4$};
\end{tikzpicture}
\vspace{-1cm}
\caption{\label{global} Chordal graph (left) with its clique tree (right).}
\end{figure}

\subsection{Distributed computation theorem}

One can formulate a simple generalization of the sequential implementation of Theorem \ref{linco} to our general {correlative} sparsity pattern.
\begin{thm} \label{main}
Let Assumptions \ref{ass-1} and \ref{asm} hold.
Let $T = (\mcK , \mcE)$ be a clique tree as in {Assumption} \ref{asm}. Then
$\vol \: \mbK = \displaystyle\int_{\mbX_1} d \mu_1^\ast$ where for $i=1,\ldots,m$, $\mu_i^\ast$ is an optimal solution to
\begin{eqnarray}
& \max\limits_{\mu_i \in \mcM_+(\mbX_i)} & \int_{\mbX_i}d\mu_i \label{distr_comp} \\
&\text{s.t.}& \mu_i \leq \left( \bigotimes\limits_{(C_i,C_j)\in\mcE} \mu_j^{\ast\mbX_i\cap\mbX_j} \right) \otimes \lambda^{n_i} \label{maj} \\
&& \spt \; \mu_i \subset \mbU_i \label{spt}
\end{eqnarray}
and $n_i = \dim\:\mbX_i \cap \left( \sum\limits_{(C_i,C_j)\in \mcE} \mbX_j \right)^\perp = \left| C_i \cap \left( \bigcup\limits_{(C_i,C_j) \in \mcE} C_j \right)^c \right|$.
\end{thm}
%
%

\begin{proof}
We define, for $i = 1,\dots,m$,
$$ \mbY_i := \mbX_i \cap \left( \sum\limits_{(C_i,C_j)\in \mcE} \mbX_j \right)^\perp = \langle x_k \rangle_{k \in C_i \ ; \ (C_i,C_j) \in \mcE \Rightarrow  k \notin C_j}$$
and we observe that, according to Assumption \ref{asm}, for any $i =1,\ldots,m$
\[
\mbX_i = \left(\bigoplus\limits_{(C_i,C_j)\in\mcE} \mbX_i \cap \mbX_j\right) \oplus \mbY_i.
\]
Thus, constraint \eqref{maj} is well-posed.

For $i = 1,\ldots,m$, let $\mcD(i) := \{j \neq i \; : \; \exists \; \text{an oriented path from } C_i \text{ to } C_j \text{ in } T\}$, the set of descendants of $C_i$, as well as $\mbZ_i:= \sum\limits_{j \in \mcD(i)} \mbX_j = \langle x_k \rangle_{k \in C_j \ ; \ j \in \mcD(i)}$ and $q_i := \dim \mbZ_i = \left|\bigcup\limits_{j\in\mcD(i)} C_j \right|$.

We are going to show by induction that for $i = 1,\dots,m$,
$$ \mu_i^\ast = \ind_{\mbU_i} \left(\prod\limits_{j \in \mcD(i)} \ind_{\mbU_j}\circ\pi_{\mbX_j} \; \lambda^{q_i}\right)^{\mbX_i\cap\mbZ_i} \otimes \lambda^{n_i}. $$
Our base cases are the leaves of $T$, i.e. the $i$ such that $\mcD(i) = \varnothing$. Then, problem \eqref{distr_comp} is reduced to the classical problem of computing the volume of $\mbU_i$, whose optimal solution is exactly $\mu_i^\ast = \lambda_{\mbU_i} = \ind_{\mbU_i} \; \lambda^{n_i}$ (because $D(i) = \varnothing \Rightarrow n_i = \dim \mbX_i$), which is the expected result.

Then we can proceed to the induction: let $i$ be a node of $T$ that is not a leaf: $\mcD(i) \neq \varnothing$; and suppose that for $j \in \mcD(i)$ such that $(C_i,C_j) \in \mcE$,
$$ \mu_j^\ast = \ind_{\mbU_j} \left(\prod\limits_{k \in \mcD(j)} \ind_{\mbU_k}\circ\pi_{\mbX_k} \; \lambda^{q_j}\right)^{\mbX_j\cap\mbZ_j} \otimes \lambda^{n_j}. $$
Then, constraint \eqref{maj} can be rewritten as
$$ \mu_i \leq \left( \bigotimes\limits_{(C_i,C_j)\in\mcE} \left( \ind_{\mbU_j} \left(\prod\limits_{k \in \mcD(j)} \ind_{\mbU_k}\circ\pi_{\mbX_k} \; \lambda^{q_j}\right)^{\mbX_j\cap\mbZ_j} \otimes \lambda^{n_j} \right)^{\mbX_i\cap\mbX_j} \right) \otimes \lambda^{n_i} $$
which in turn is simplified into
$$ \mu_i \leq \left( \bigotimes\limits_{(C_i,C_j)\in\mcE} \left( \ind_{\mbU_j} \left(\prod\limits_{k \in \mcD(j)} \ind_{\mbU_k}\circ\pi_{\mbX_k} \right) \; \lambda^{n_j+q_j} \right)^{\mbX_i\cap\mbZ_j} \right) \otimes \lambda^{n_i} $$
since the CIP ensures that $\mbX_i \cap \mbX_j \cap \mbZ_j = \mbX_i \cap \mbZ_j$: indeed $C_j$ is on the path between $C_i$ and any $C_k$ with $k \in \mcD(j)$, so that $C_i \cap C_k \subset C_j$ and thus $\mbX_i \cap \mbX_k \subset \mbX_j$, yielding $\mbX_i \cap \mbZ_j \subset \mbX_j$. At this point one can notice that $n_j + q_j = \dim (\mbX_j + \mbZ_j)$.

Then, using the DIP, we know that if $(C_i,C_j),(C_i,C_k) \in \mcE$ with $j \neq k$ then $C_j \cap C_k = \varnothing$ and with the CIP $C_j \cap C_l = \varnothing$ for any $l \in \mcD(k)$. This yields that $(\mbX_j + \mbZ_j) \cap (\mbX_k + \mbZ_k) = \{0\}$ and thus $\mbZ_i = \bigoplus\limits_{(C_i,C_j) \in \mcE} (\mbX_j + \mbZ_j)$, allowing to rewrite constraint \eqref{maj} as
$$ \mu_i \leq \left( \bigotimes\limits_{(C_i,C_j)\in\mcE} \ind_{\mbU_j} \left(\prod\limits_{k \in \mcD(j)} \ind_{\mbU_k}\circ\pi_{\mbX_k} \right) \; \lambda^{n_j+q_j} \right)^{\mbX_i\cap\mbZ_i} \otimes \lambda^{n_i}, $$
which simplifies into
$$ \mu_i \leq \left( \prod\limits_{j \in \mcD(i)} \ind_{\mbU_j}\circ\pi_{\mbX_j} \; \lambda^{q_i} \right)^{\mbX_i\cap\mbZ_i} \otimes \lambda^{n_i}. $$
Eventually, we are again faced to a classical instance of the dense volume problem for $\mbU_i$, with $\left( \prod\limits_{j \in \mcD(i)} \ind_{\mbU_j}\circ\pi_{\mbX_j} \; \lambda^{q_i} \right)^{\mbX_i\cap\mbZ_i} \otimes \lambda^{n_i}$ instead of only the uniform Lebesgue measure, and we know that the optimal solution is obtained by multiplying this non-negative dominating measure with the indicator of $\mbU_i$, yielding
$$ \mu_i^\ast = \ind_{\mbU_i} \left( \prod\limits_{j \in \mcD(i)} \ind_{\mbU_j}\circ\pi_{\mbX_j} \; \lambda^{q_i} \right)^{\mbX_i\cap\mbZ_i} \otimes \lambda^{n_i} $$
which is the announced result.

We conclude by using the fact that $\mcD(1) = \{2,\dots,m\}$ and $\R^n = \mbX_1 + \mbZ_1 = \mbX_1 \oplus (\mbX_1^\perp \cap \mbZ_1)$ to compute the value:
\begin{eqnarray*}
\int_{\mbX_1} d\mu_1^\ast &=& \int_{\mbX_1} \ind_{\mbU_1}  \left( \prod\limits_{j = 2}^m \ind_{\mbU_j}\circ\pi_{\mbX_j} \; \lambda^{q_1} \right)^{\mbX_1\cap\mbZ_1} d\lambda^{n_1} \\
&=& \int_{\mbX_1} \ind_{\mbU_1}(\mbx_1) \left( \int_{\mbX_1^\perp \cap \mbZ_1} \prod\limits_{j = 2}^m \ind_{\mbU_j}\circ\pi_{\mbX_j}(\mbx_1+\mbz_1) \; d\mbz_1 \right) d\mbx_1 \\
&=& \int_{\R^n} \left(\prod\limits_{i=1}^m \ind_{\mbU_i}\circ\pi_{\mbX_i}(\mbx) \right) d\mbx \\
&=& \int_{\R^n} \left(\prod\limits_{i=1}^m \ind_{\mbK_i}(\mbx) \right) d\mbx \\
&=& \int_{\R^n} \ind_\mbK(\mbx) \; d\mbx \\
&=& \vol \ \mbK.
\end{eqnarray*}
\end{proof}
\color{black}

Therefore one obtains a sequence of infinite dimensional LPs on measures that can be algorithmically addressed using the usual SDP relaxations. The computations start from the leaves of the clique tree and proceed down to the root. It is worth noting that all the {maximal} cliques of the same generation in the tree are totally independent, which allows to treat them simultaneously, \textit{i.e.} to partially parallelize the computations. {Let $d \in \N$. We consider the solutions $\mbm_i^{(d)}$ to the moment relaxations, for $i = 1,\dots,m$:

\begin{eqnarray}
\mbP_{d,i} = & \max\limits_{\mbm_i,\hat{\mbm}_i \in \R^{s(d)}} & m_{i,0} \label{blh} \\
&\text{s.t.}& m_{i,(\alpha^{(j)})_j,\beta} + \hat{m}_{i,(\alpha^{(j)})_j,\beta} = \left( \prod\limits_{(C_i,C_j)\in\mcE} m_{j,\alpha^{(j)},0}^{(d)} \right) \ell_{i,\beta} \label{majh} \\
&& \mbM_d(\mbm_i) \succeq 0, \mbM_d(\hat{\mbm}_i) \succeq 0 \nonumber \\
&& \mbM_{d-d_i}(\mbg_i \mbm_i) \succeq 0 \nonumber
\end{eqnarray}
where $(\alpha^{(j)})_j,\beta$ are appropriate multi-indices and $\ell_{i,\beta}$ is the $\beta$ moment of $\lambda^{n_i}$ on the appropriate projection of $\mbB$. We are going to study the convergence of the sequence $\mbP_{d,1}$ to $\vol \ \mbK$.


\begin{thm}[Convergence of the branched Moment-SOS hierarchy]
Let $C_i \in \mcK$ such that for $C_j \in \mcK$ satisfying $(C_i,C_j) \in \mcE$, we have a converging sequence of moment vectors $(\mbm_j^{(d)})_d$: $\mbm_j^{(d)} \in \R^{s(d)}$ and for any appropriate multi-index $\alpha$, $m_{j,\alpha}^{(d)}{\underset{d \to \infty}{\longrightarrow}} \int \mbx^\alpha \; d\mu_j^{\ast}$.

In this setting, if each relaxation of the LP problem \eqref{blh} associated to the clique $C_i$ has at least one feasible solution, then the Moment-SOS hierarchy associated to clique $C_i$ converges, in the sense that for any appropriate multi-index $\alpha$, $m_{i,\alpha}^{(d)} {\underset{d \to \infty}{\longrightarrow}} \int \mbx^\alpha \; d\mu_i^\ast$.

Thus, by induction, if at all nodes of $T$ the moment relaxations remain feasible at all degrees of relaxation, then the branched Moment-SOS hierarchy converges, namely $\mbP_{d,1} \underset{d \to \infty}{\longrightarrow} \vol \ \mbK$.
\end{thm}
\begin{proof}
The feasibility assumption ensures that the $\mbm_i^{(d)}$ are properly defined at all degrees $d$. Then, pointwise convergence of the $(\mbm_j^{(d)})_d$ yields that $m^{(d)}_{j,0}$ is bounded for the weak-$\ast$ topology on $\R[\mbx]'$. Constraint \eqref{majh} yields that $(\mbm_i^{(d)})_d$ is bounded for the weak-$\ast$ topology on $\R[\mbx]'$, which means, according to Banach-Alaoglu's theorem, that it has an accumulation point. Finally, by uniqueness of the solution to the infinite dimensional LP problem \eqref{distr_comp}, the convergence of $(\mbm_j^{(d)})_d$ to the moment sequence of $\mu_j^\ast$ ensures that this accumulation point is none other than the moment sequence of $\mu_i^\ast$. This proves existence and uniqueness of the accumulation point of $(\mbm_i^{(d)})_d$. Then, we get for any appropriate multi-index $\alpha$
$$m^{(d)}_{i,\alpha} \underset{d \to \infty}{\longrightarrow} \int \mbx^\alpha \; d\mu_i^\ast. $$
To conclude for the global convergence of the sparse scheme, we just need to check the base case of this induction. Here again the base case is the leaves of the tree at which we are face to standard instances of the volume problem, whose associated Moment-SOS hierarchy is already proved to converge. Thus, our convergence assumption is satisfied, which means that as long as all the relaxations are feasible, their solutions converge weakly-$\ast$ to the infinite dimensional optimal measures, and in particular
$$ \mbP_{d,1} \underset{d \to \infty}{\longrightarrow} \vol \ \mbK. $$
\end{proof}
}


\begin{rqe}
Stokes constraints can be implemented similarly to the linear case.
\end{rqe}

\subsection{Example: 6D polytope}\label{ex1}

Let $\mbX := \R^6$ and $\mbX_1 = \langle x_1,x_2 \rangle$, 
$\mbX_2 = \langle x_2,x_3,x_4 \rangle$, $\mbX_3 = \langle x_3,x_5 \rangle$, $\mbX_4 = \langle x_4,x_6 \rangle$. For $i=1,3,4$ let
$\mbg_i(x,y) := (x,y,1-x-y)$ and $\mbU_i := \mbg_i^{-1}\left( \R_+^2 \right) = \{(u,v) \in [0,1]^2 \; : \; u+v \leq 1 \}$. Let $\mbg_2(x,y,z) := (x,y,z,1-x-y-z)$ and $\mbU_2 := \mbg_2^{-1}\left(\R_+^3\right) = \{(x,y,z) \in [0,1]^3 \; \vert \; x+y+z \leq 1 \}$. Let us approximate the volume of the 6D polytope
$$ \mbK := \{\mbx \in \R_+^6 \; :\; x_1+x_2 \leq 1,\: x_2+x_3+x_4 \leq 1,\: x_3+x_5 \leq 1,\: x_4+x_6 \leq 1\} = \bigcap\limits_{i=1}^4 \pi_{\mbX_i}^{-1}\left(\mbU_i\right). $$
No linear clique tree is associated to this problem through Proposition \ref{asm}. The only possible clique trees for applying our method are the two branched clique trees of Figure \ref{nonlinear}.
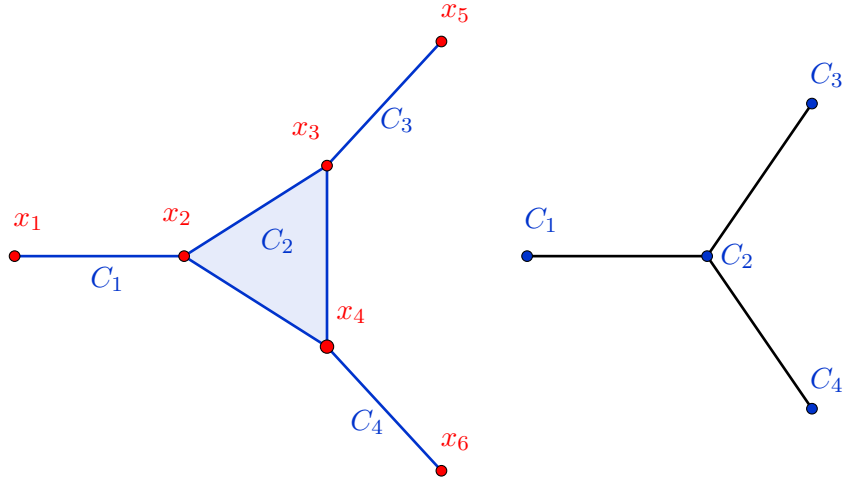
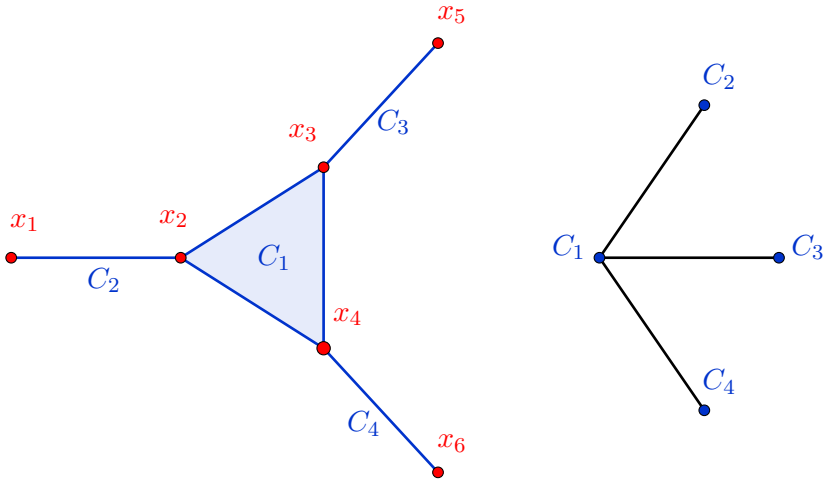
\begin{figure}[!h]
\begin{center}
\begin{subfigure}{\textwidth}
\begin{center}
\begin{tikzpicture}[line cap=round,line join=round,>=triangle 45,x=1.0cm,y=1.0cm]
\clip(-2.5,-3.) rectangle (4.,3.5);
\fill[line width=1.pt,color=qqttcc,fill=qqttcc,fill opacity=0.10000000149011612] (0.,0.) -- (1.88,1.2) -- (1.88,-1.2) -- cycle;
\draw [line width=1.pt,color=qqttcc] (0.,0.)-- (1.88,1.2);
\draw [line width=1.pt,color=qqttcc] (1.88,1.2)-- (1.88,-1.2);
\draw [line width=1.pt,color=qqttcc] (1.88,-1.2)-- (0.,0.);
\draw [line width=1.pt,color=qqttcc] (1.88,1.2)-- (3.3853576993004233,2.845690796339621);
\draw [line width=1.pt,color=qqttcc] (1.88,-1.2)-- (3.3853576993004233,-2.845690796339621);
\draw [line width=1.pt,color=qqttcc] (-2.230336297512104,0.)-- (0.,0.);
\draw [fill=ffqqqq] (0.,0.) circle (2.0pt);
\draw[color=ffqqqq] (-0.09,0.52) node {$x_2$};
\draw [fill=ffqqqq] (1.88,1.2) circle (2.0pt);
\draw[color=ffqqqq] (1.61,1.66) node {$x_3$};
\draw [fill=ffqqqq] (1.88,-1.2) circle (2.5pt);
\draw[color=ffqqqq] (2.2,-0.78) node {$x_4$};
\draw[color=qqttcc] (1.23,0.2) node {$C_2$};
\draw [fill=ffqqqq] (3.3853576993004233,2.845690796339621) circle (2.0pt);
\draw[color=ffqqqq] (3.57,3.22) node {$x_5$};
\draw [fill=ffqqqq] (3.3853576993004233,-2.845690796339621) circle (2.0pt);
\draw[color=ffqqqq] (3.57,-2.46) node {$x_6$};
\draw[color=qqttcc] (2.8,1.8) node {$C_3$};
\draw[color=qqttcc] (2.41,-2.2) node {$C_4$};
\draw [fill=ffqqqq] (-2.230336297512104,0.) circle (2.0pt);
\draw[color=ffqqqq] (-2.05,0.46) node {$x_1$};
\draw[color=qqttcc] (-1.01,-0.3) node {$C_1$};
\end{tikzpicture}
\begin{tikzpicture}[line cap=round,line join=round,>=triangle 45,x=1.0cm,y=1.0cm]
\clip(-1.5,-3) rectangle (3.5,2.7);
\draw [line width=1.pt] (-1.115168148756052,0.)-- (1.2533333333333332,0.);
\draw [line width=1.pt] (1.2533333333333332,0.)-- (2.632678849650212,2.0228453981698107);
\draw [line width=1.pt] (1.2533333333333332,0.)-- (2.632678849650212,-2.0228453981698107);
\draw [fill=qqttcc] (-1.115168148756052,0.) circle (2.0pt);
\draw[color=qqttcc] (-0.93,0.48) node {$C_1$};
\draw [fill=qqttcc] (2.632678849650212,2.0228453981698107) circle (2.0pt);
\draw[color=qqttcc] (2.83,2.4) node {$C_3$};
\draw [fill=qqttcc] (2.632678849650212,-2.0228453981698107) circle (2.0pt);
\draw[color=qqttcc] (2.83,-1.64) node {$C_4$};
\draw [fill=qqttcc] (1.2533333333333332,0.) circle (2.0pt);
\draw[color=qqttcc] (1.65,0.) node {$C_2$};
\end{tikzpicture}
\caption{3 step clique tree}
\end{center}
\end{subfigure}

\begin{subfigure}{\textwidth}
\begin{center}
\hspace{-1cm}
\begin{tikzpicture}[line cap=round,line join=round,>=triangle 45,x=1.0cm,y=1.0cm]
\clip(-2.5,-3.) rectangle (4.,3.5);
\fill[line width=1.pt,color=qqttcc,fill=qqttcc,fill opacity=0.10000000149011612] (0.,0.) -- (1.88,1.2) -- (1.88,-1.2) -- cycle;
\draw [line width=1.pt,color=qqttcc] (0.,0.)-- (1.88,1.2);
\draw [line width=1.pt,color=qqttcc] (1.88,1.2)-- (1.88,-1.2);
\draw [line width=1.pt,color=qqttcc] (1.88,-1.2)-- (0.,0.);
\draw [line width=1.pt,color=qqttcc] (1.88,1.2)-- (3.3853576993004233,2.845690796339621);
\draw [line width=1.pt,color=qqttcc] (1.88,-1.2)-- (3.3853576993004233,-2.845690796339621);
\draw [line width=1.pt,color=qqttcc] (-2.230336297512104,0.)-- (0.,0.);
\draw [fill=ffqqqq] (0.,0.) circle (2.0pt);
\draw[color=ffqqqq] (-0.09,0.52) node {$x_2$};
\draw [fill=ffqqqq] (1.88,1.2) circle (2.0pt);
\draw[color=ffqqqq] (1.61,1.66) node {$x_3$};
\draw [fill=ffqqqq] (1.88,-1.2) circle (2.5pt);
\draw[color=ffqqqq] (2.2,-0.78) node {$x_4$};
\draw[color=qqttcc] (1.23,0.) node {$C_1$};
\draw [fill=ffqqqq] (3.3853576993004233,2.845690796339621) circle (2.0pt);
\draw[color=ffqqqq] (3.57,3.22) node {$x_5$};
\draw [fill=ffqqqq] (3.3853576993004233,-2.845690796339621) circle (2.0pt);
\draw[color=ffqqqq] (3.57,-2.46) node {$x_6$};
\draw[color=qqttcc] (2.8,1.8) node {$C_3$};
\draw[color=qqttcc] (2.41,-2.2) node {$C_4$};
\draw [fill=ffqqqq] (-2.230336297512104,0.) circle (2.0pt);
\draw[color=ffqqqq] (-2.05,0.46) node {$x_1$};
\draw[color=qqttcc] (-1.01,-0.3) node {$C_2$};
\end{tikzpicture}
\hspace{0.5cm}
\begin{tikzpicture}[line cap=round,line join=round,>=triangle 45,x=1.0cm,y=1.0cm]
\clip(.5,-3) rectangle (4.3,3);
\draw [line width=1.pt] (1.2533333333333332,0.)-- (3.615168148756052,0.);
\draw [line width=1.pt] (1.2533333333333332,0.)-- (2.632678849650212,2.0228453981698107);
\draw [line width=1.pt] (1.2533333333333332,0.)-- (2.632678849650212,-2.0228453981698107);
\draw [fill=qqttcc] (3.615168148756052,0.) circle (2.0pt);
\draw[color=qqttcc] (4,0.14) node {$C_3$};
\draw [fill=qqttcc] (2.632678849650212,2.0228453981698107) circle (2.0pt);
\draw[color=qqttcc] (2.83,2.4) node {$C_2$};
\draw [fill=qqttcc] (2.632678849650212,-2.0228453981698107) circle (2.0pt);
\draw[color=qqttcc] (2.83,-1.64) node {$C_4$};
\draw [fill=qqttcc] (1.2533333333333332,0.) circle (2.0pt);
\draw[color=qqttcc] (0.85,0.14) node {$C_1$};
\end{tikzpicture}
\caption{2 step clique tree}
\end{center}
\end{subfigure}
\caption{Two possible branched clique trees for the 6D polytope.}\label{nonlinear}
\end{center}
\end{figure}

Let us compare the performance of the algorithms derived from the two possible clique tree configurations and with the dense problem. For that, we first write the problem associated with the 3 step clique tree configuration of the top of Figure \ref{nonlinear}: 
\begin{equation} \label{linear}
\vol \: \mbK = \int_{\mbX_1} d \mu^{\ast}_1
\end{equation} where \begin{eqnarray*}
\mu^{\ast}_1 &=& \underset{\mu_1 \in \mcM_+(\mbX_1)}{\argmax} \int_{\mbX_1} d \mu_1\\
&\text{s.t.}& d\mu_1(x_1,x_2) \leq d x_1 \; d\mu_2^{\ast \langle x_2\rangle}(x_2) \\
&& \spt \: \mu_1 \subset \mbU_1 \\
\mu^{\ast}_2 &=& \underset{\mu_2 \in \mcM_+(\mbX_2)}{\argmax} \int_{\mbX_2} d \mu_2\\
&\text{s.t.}& d\mu_2(x_2,x_3,x_4) \leq d x_2 \; d\mu^{\ast\langle x_3\rangle}_3(x_3) \; d\mu^{\ast\langle x_4\rangle}_4(x_4) \\
&& \spt \: \mu_2 \subset \mbU_2 \\
\mu_i^\ast &=& \underset{\substack{\mu_i \in \mcM_+(\mbX_i)\\i=
3,4}}{\argmax} \int_{\mbX_i} d \mu_i\\
&\text{s.t.}& \mu_i \leq \lambda^2 \\
&& \spt \; \mu_i \subset \mbU_i, \qquad\qquad i=3,4.
\end{eqnarray*}
This problem can be complemented with the following Stokes constraints:
\begin{eqnarray*}
\frac{\partial}{\partial x_1} \left[x_1\:(1-x_1-x_2) \vphantom{\sum} \right] \; d\mu_1(x_1,x_2) &=& \frac{\partial}{\partial x_1} \left[ x_1\:(1-x_1-x_2) \; d\mu_1(x_1,x_2)\right] \\
\frac{\partial}{\partial x_2} \left[x_2\:(1-x_2-x_3-x_4) \vphantom{\sum} \right] \; d\mu_2(x_2,x_3,x_4) &=& \frac{\partial}{\partial x_2} \left[ x_2\:(1-x_2-x_3-x_4) \;d\mu_2(x_2,x_3,x_4) \right]  \\
\frac{\partial}{\partial x_i} \left[x_i \: (1-x_i-x_{i+2}) \vphantom{\sum} \right] \; d\mu_i(x_i,x_{i+2}) &=& \frac{\partial}{\partial x_i} \left[ x_i \: (1-x_i-x_{i+2}) \; d\mu_i(x_i,x_{i+2}) \right]\\
\frac{\partial}{\partial x_{i+2}} \left[x_{i+2} \: (1-x_i-x_{i+2}) \vphantom{\sum} \right] \; d\mu_i(x_i,x_{i+2}) &=& \frac{\partial}{\partial x_{i+2}} \left[ x_{i+2} \: (1-x_i-x_{i+2}) \;d\mu_i(x_i,x_{i+2}) \right] \quad  \quad i=3,4.
\end{eqnarray*}
The 2 step clique tree of the bottom of Figure \ref{nonlinear} yields the following formulation
\begin{equation} \label{branched}
\vol \: \mbK = \int_{\mbX_2} d \mu^{\ast}_2
\end{equation}
where
\begin{eqnarray*}
\mu^{\ast}_2 &=& \underset{\mu_2 \in \mcM_+(\mbX_2)}{\argmax} \int_{\mbX_2} d \mu_2\\
&\text{s.t.}& d\mu_2(x_2,x_3,x_4) \leq d\mu^{\ast\langle x_2\rangle}_1(x_2) \; d\mu^{\ast\langle x_3\rangle}_3(x_3) \; d\mu^{\ast\langle x_4\rangle}_4(x_4) \\
&& \spt \: \mu_2 \subset \mbU_2 \\
\mu_i^\ast &=& \underset{\substack{\mu_i \in \mcM_+(\mbX_i)\\i=1,3,4}}{\argmax} \int_{\mbX_i} d \mu_i\\
&\text{s.t.}& \mu_i \leq \lambda_{\mbX_i} \\
&& \spt \; \mu_i \subset \mbU_i,  \qquad \qquad \qquad i = 1,3,4
\end{eqnarray*}
with Stokes constraints
\begin{eqnarray*}
\frac{\partial}{\partial x_i} \left[x_i \: (1-x_i-x_j) \vphantom{\sum} \right] \; d\mu_i(x_i,x_j) &=& \frac{\partial}{\partial x_i} \left[ x_i \: (1-x_i-x_j) \;d\mu_i(x_i,x_j)\right]\\
\frac{\partial}{\partial x_j} \left[x_j \: (1-x_i-x_j) \vphantom{\sum} \right] \; d\mu_i(x_i,x_j) &=& \frac{\partial}{\partial x_j} \left[ x_j \: (1-x_i-x_j) \; d\mu_i(x_i,x_j)\right]
\end{eqnarray*}
for $(i,j) = (1,2),(3,5),(4,6)$.
However, no Stokes constraints can be applied for the computation of $\mu_2$ (there is no Lebesgue measure in the domination constraint, so the optimal measure is not uniform). For this reason one can expect a slower convergence than in the linear configuration.

We implement the hierarchies associated to the 2 and 3 step sparse formulations, as well as the dense problem hierarchy, and compare their performance in Figure \ref{fig:nonlinear}. We can compute analytically
\[
\vol \: \mbK = \frac{1}{18} \simeq 0.0556.
\]
\begin{figure}[!h]
\begin{subfigure}{0.5\textwidth}
\includegraphics[width=\textwidth]{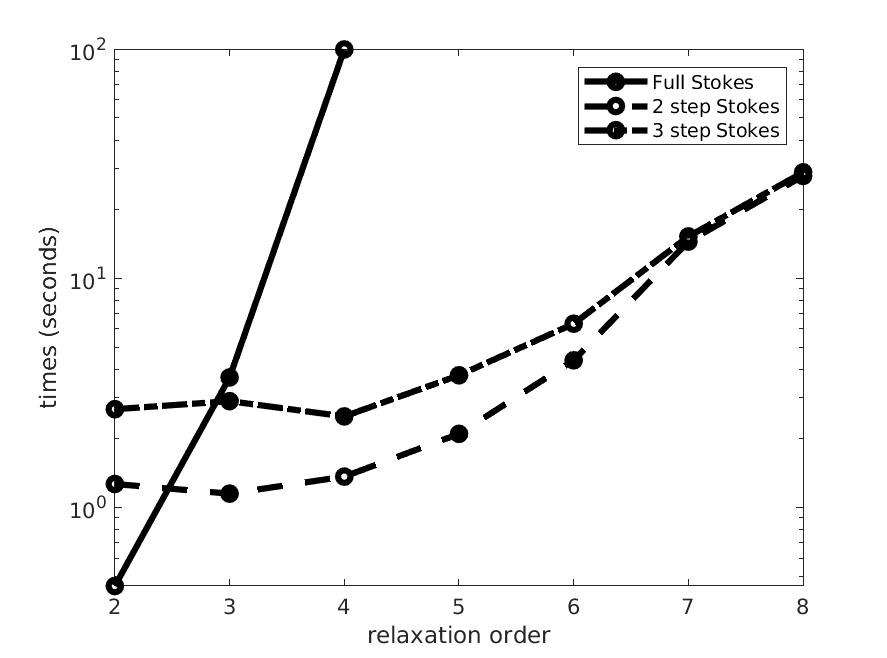}
\caption{Computation time vs relaxation order.}
\label{fig:nonlinear_tim}
\end{subfigure}
\hfill
\begin{subfigure}{0.5\textwidth}
\includegraphics[width=\textwidth]{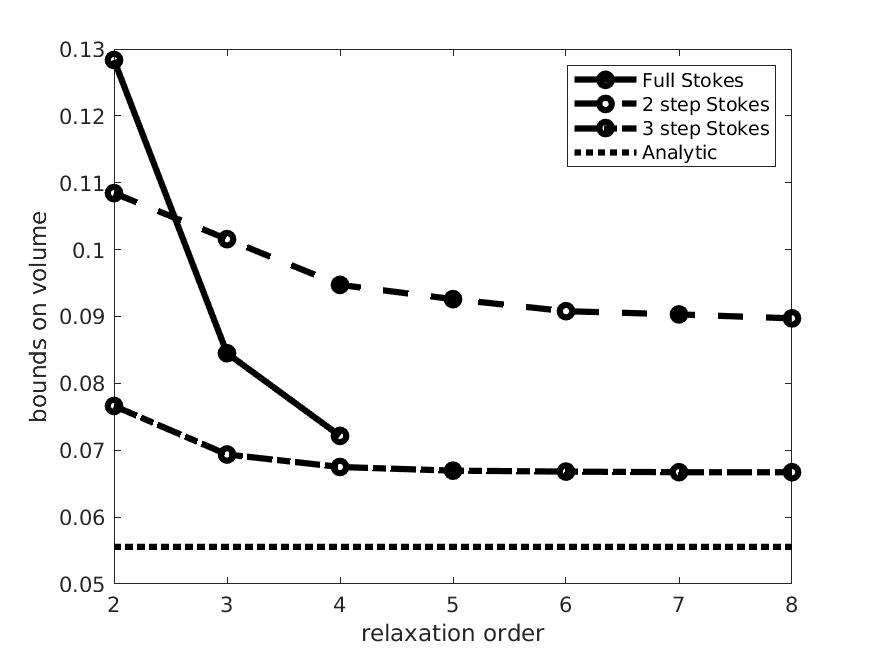}
\caption{Bounds on the volume vs relaxation order.}
\label{fig:nonlinear_val}
\end{subfigure}
\caption{Performance for the 6D polytope.}
\label{fig:nonlinear}
\end{figure}

Both sparse formulations outperform the dense one in terms of computational time needed to solve the corresponding SDPs (Figure \ref{fig:nonlinear_tim}). On the accuracy side however (Figure \ref{fig:nonlinear_val}), we observe that the 2 step formulation seems to be less efficient than the 3 step formulation. In particular when considering the accuracy/time effort relation at order $3$ the dense formulation provides a better value in almost the same time.

We believe that this can be explained by the way Stokes constraints are added to the program. {Indeed, at a given clique, Stokes constraints can only be implemented in the variables that are not shared with the input measure.} In the {fully (2 step)} branched configuration, the last step of the optimization program cannot be accelerated by Stokes constraints at all{, while in the 3 step configuration, step 1 includes Stokes constraints in $x_3,x_4,x_5,x_6$, step 2 includes Stokes constraints in $x_2$ and step 3 includes Stokes constraints on $x_1$, which explains the gap between the optimal values of these two configurations}.

{Moreover, it seems that even the least branched (3 step) configuration still presents a gap between its optimal value and the analytic solution. This might also happen with a non-sparse instance of the Moment-SOS hierarchy (which converges theoretically) and it is likely due to the choice of the monomial basis to represent polynomials. Indeed, most of the Moment-SOS parsers generate SDP problems with the basis of monomials, while sometimes other bases (e.g. Chebyshev or Legendre polynomials) are more appropriate. However, in this precise case, it might also be linked again with the sparse Stokes constraints implementation. Indeed, in step 2 of the scheme, the unknown measure measures $x_2,x_3$ and $x_4$ but Stokes constraints are implemented only in $x_2$, leaving a possible Gibbs effect in $x_3,x_4$. Unlike most of our numerical examples, this one still includes an optimization step in which most of the variables are not controlled through Stokes constraints.} The gap between the optimal value and the analytic value for the {3 step} branched formulation in Figure \ref{fig:nonlinear_val} could be explained by a Gibbs effect in the {second} optimization step.

As a consequence, in the following, one should avoid the branched hierarchies that cannot be accelerated at each step { at least partially} by Stokes constraints. Such a hierarchy appears when the root of the chosen clique tree shares all its vertices with its children cliques. It can be proved that such a configuration can always be avoided while implementing sparse volume computation, by choosing a leaf as the new root of the tree.

\subsection{Example: 4D polytope}\label{ex2}

Let $\mbX := \R^4$, $\mbX_1:=\langle x_1,x_2\rangle$, $\mbX_2:=\langle x_2,x_3\rangle$, $\mbX_3:=\langle x_3,x_4\rangle$, $\mbg_i(u,v) := (u,v,1-u-v)$, $i=1,2,3$ and $\mbU_i := \mbg_i^{-1}\left(\R_+^3 \right) = \{(u,v) \in [0,1]^2 \; : \; u+v \leq 1\}$. Let us approximate the 
volume of the 4D polytope
$$ \mbK := \left\{(x_1,x_2,x_3,x_4) \in \R_+^4 \: : \:
x_1+x_2 \leq 1,\: x_2+x_3 \leq 1,\: x_3+x_4 \leq 1 \right\} = \bigcap\limits_{i=1}^3 \pi_{\mbX_i}^{-1}\left(\mbU_i\right). $$
In such a case, there are two possible configurations for the associated clique tree of Proposition \ref{asm}, see Figure \ref{threecyl}. Accordingly, we can compute $\vol \: \mbK$ in two different ways. The first way
\begin{figure}[!h]
\begin{center}
\begin{subfigure}{\textwidth}
\begin{center}
\hspace{-0.8cm}
\begin{tikzpicture}[line cap=round,line join=round,>=triangle 45,x=1.0cm,y=1.0cm]
\clip(-2.5,-0.5) rectangle (5.,1);
\draw [line width=1.pt,color=qqttcc] (-2.230336297512104,0.)-- (0.,0.);
\draw [line width=1.pt,color=qqttcc] (0.,0.)-- (2.230336297512104,0.);
\draw [line width=1.pt,color=qqttcc] (2.230336297512104,0.)-- (4.460672595024208,0.);
\draw [fill=ffqqqq] (0.,0.) circle (2.0pt);
\draw[color=ffqqqq] (0.17,0.44) node {$x_2$};
\draw [fill=ffqqqq] (-2.230336297512104,0.) circle (2.0pt);
\draw[color=ffqqqq] (-2.05,0.44) node {$x_1$};
\draw [fill=ffqqqq] (2.230336297512104,0.) circle (2.0pt);
\draw[color=ffqqqq] (2.43,0.44) node {$x_3$};
\draw[color=qqttcc] (-1.01,-0.25) node {$C_1$};
\draw [fill=ffqqqq] (4.460672595024208,0.) circle (2.0pt);
\draw[color=ffqqqq] (4.65,0.38) node {$x_4$};
\draw[color=qqttcc] (1.21,-0.25) node {$C_2$};
\draw[color=qqttcc] (3.45,-0.25) node {$C_3$};
\end{tikzpicture}
\begin{tikzpicture}[line cap=round,line join=round,>=triangle 45,x=1.0cm,y=1.0cm]
\clip(-1.5,-0.5) rectangle (3.7,1.);
\draw [line width=1.pt] (-1.115168148756052,0.)-- (1.115168148756052,0.);
\draw [line width=1.pt] (1.115168148756052,0.)-- (3.3455044462681562,0.);
\draw [fill=qqttcc] (-1.115168148756052,0.) circle (2.0pt);
\draw[color=qqttcc] (-1.09,0.58) node {$C_1$};
\draw [fill=qqttcc] (1.115168148756052,0.) circle (2.0pt);
\draw[color=qqttcc] (1.13,0.62) node {$C_2$};
\draw [fill=qqttcc] (3.3455044462681562,0.) circle (2.0pt);
\draw[color=qqttcc] (3.27,0.6) node {$C_3$};
\end{tikzpicture}
\caption{linear clique tree}
\end{center}
\end{subfigure}

\begin{subfigure}{\textwidth}
\begin{center}
\hspace{-1.5cm}
\begin{tikzpicture}[line cap=round,line join=round,>=triangle 45,x=1.0cm,y=1.0cm]
\clip(-2.5,-1.) rectangle (5.,1);
\draw [line width=1.pt,color=qqttcc] (-2.230336297512104,0.)-- (0.,0.);
\draw [line width=1.pt,color=qqttcc] (0.,0.)-- (2.230336297512104,0.);
\draw [line width=1.pt,color=qqttcc] (2.230336297512104,0.)-- (4.460672595024208,0.);
\draw [fill=ffqqqq] (0.,0.) circle (2.0pt);
\draw[color=ffqqqq] (0.17,0.44) node {$x_2$};
\draw [fill=ffqqqq] (-2.230336297512104,0.) circle (2.0pt);
\draw[color=ffqqqq] (-2.05,0.44) node {$x_1$};
\draw [fill=ffqqqq] (2.230336297512104,0.) circle (2.0pt);
\draw[color=ffqqqq] (2.43,0.44) node {$x_3$};
\draw[color=qqttcc] (-1.01,-0.25) node {$C_2$};
\draw [fill=ffqqqq] (4.460672595024208,0.) circle (2.0pt);
\draw[color=ffqqqq] (4.65,0.38) node {$x_4$};
\draw[color=qqttcc] (1.21,-0.25) node {$C_1$};
\draw[color=qqttcc] (3.45,-0.25) node {$C_3$};
\end{tikzpicture}
\hspace{0.8cm}
\begin{tikzpicture}[line cap=round,line join=round,>=triangle 45,x=1.0cm,y=1.0cm]
\clip(-0.5,-2.) rectangle (3.,2.);
\draw [line width=1.pt] (0.,0.)-- (2.34,1.36);
\draw [line width=1.pt] (0.,0.)-- (2.34,-1.36);
\draw [fill=qqttcc] (0.,0.) circle (2.0pt);
\draw[color=qqttcc] (-0.09,0.5) node {$C_1$};
\draw [fill=qqttcc] (2.34,1.36) circle (2.0pt);
\draw[color=qqttcc] (2.53,1.74) node {$C_2$};
\draw [fill=qqttcc] (2.34,-1.36) circle (2.0pt);
\draw[color=qqttcc] (2.53,-0.98) node {$C_3$};
\end{tikzpicture}
\caption{branched clique tree}
\end{center}
\end{subfigure}
\caption{Two possible clique trees for the 4D polytope.} \label{threecyl}
\end{center}
\end{figure}
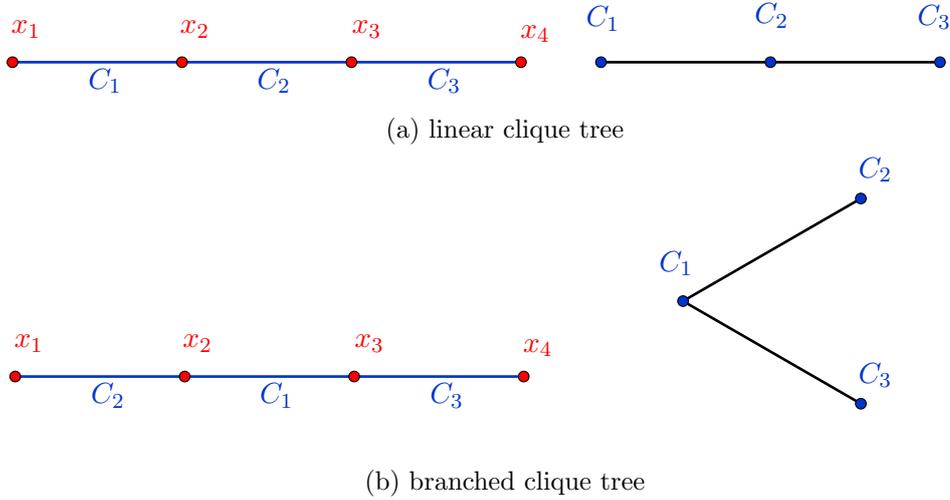
\begin{eqnarray}
\vol \: \mbK &=& \max\limits_{\substack{\mu_1 \in \mcM_+(\mbX_1) \\ \mu_2 \in \mcM_+(\mbX_2) \\ \mu_3 \in \mcM_+(\mbX_3)}} \int_{\mbX_1} d \mu_1 \\
&\text{s.t.}& d\mu_1(x_1,x_2) \leq d x_1 \; d\mu_2^{\langle x_2\rangle}(x_2) \nonumber \\
&& d\mu_2(x_2,x_3) \leq d x_2 \; d\mu_3^{\langle x_3\rangle}(x_2) \nonumber \\
&& d\mu_3(x_3,x_4) \leq d x_3 \; d x_4 \nonumber \\
&& \spt \; \mu_i \subset \mbU_i, \qquad  \qquad\qquad i=1,2,3. \nonumber 
\end{eqnarray}
is the linear formulation given by Corollary \ref{linco}, which is under the form of a non-parallelizable single linear problem.
The following additional Stokes constraints can be added:
\begin{eqnarray*}
\frac{\partial}{\partial x_1} \left[x_1 \: (1-x_1-x_2) \vphantom{\sum} \right] \; d\mu_1(x_1,x_2) &=& \frac{\partial}{\partial x_1} \left[ x_1 \: (1-x_1-x_2) \; d\mu_1(x_1,x_2) \right]\\
\frac{\partial}{\partial x_2} \left[x_2 \: (1-x_2-x_3) \vphantom{\sum} \right] \; d\mu_2(x_2,x_3) &=& \frac{\partial}{\partial x_2} \left[ x_2 \: (1-x_2-x_3) \; d\mu_2(x_2,x_3) \right]\\
\frac{\partial}{\partial x_3} \left[x_3 \: (1-x_3-x_4) \vphantom{\sum} \right] \; d\mu_3(x_3,x_4) &=& \frac{\partial}{\partial x_3} \left[ x_3 \: (1-x_3-x_4) \; d\mu_3(x_3,x_4) \right]\\
\frac{\partial}{\partial x_4} \left[x_4 \: (1-x_3-x_4) \vphantom{\sum} \right] \; d\mu_3(x_3,x_4) &=& \frac{\partial}{\partial x_4} \left[ x_4 \: (1-x_3-x_4) \; d\mu_3(x_3,x_4) \right].\\
\end{eqnarray*}
On the other hand, if one associates the {maximal} clique $C_1$ to the subspace $\mbX_2$ and the {maximal} clique $C_2$ to the subspace $\mbX_1$, one also has
\begin{equation}
\vol \: \mbK = \int_{\mbX_2} d \mu^{\ast}_2
\end{equation}
where
\begin{eqnarray*}
\mu^{\ast}_2 &=& \underset{\mu_2 \in \mcM_+(\mbX_2)}{\argmax} \int_{\mbX_2} d \mu_2\\
&\text{s.t.}& d\mu_2(x_2,x_3) \leq \mu^{\ast\langle x_2\rangle}_1(dx_2) \; d\mu^{\ast\langle x_3\rangle}_3(x_3) \\
&& \spt \: \mu_2 \subset \mbU_2 \\
\mu^{\ast}_1 &=& \underset{\mu_1 \in \mcM_+(\mbX_1)}{\argmax} \int_{\mbX_1} d \mu_1\\
&\text{s.t.}& d\mu_1(x_1,x_2) \leq d x_1 \; d x_2 \\
&& \spt \: \mu_1 \subset \mbU_1 \\
\mu^{\ast}_3 &=& \underset{\mu_3 \in \mcM_+(\mbX_3)}{\argmax} \int_{\mbX_3} d \mu_3\\
&\text{s.t.}& d\mu_3(x_3,x_4) \leq d x_3 \; d x_4 \\
&& \spt \: \mu_3 \subset \mbU_3
\end{eqnarray*}
which is the branched formulation associated to Theorem \ref{main}. Here one can see that $\mu^{\ast}_1$ and $\mu^{\ast}_3$ can be computed independently in parallel, and then re-injected in the problem to which $\mu^{\ast}_2$ is the solution.
One can add the following Stokes constraints:
\begin{eqnarray*}
\frac{\partial}{\partial x_1} \left[x_1 \: (1-x_1-x_2) \vphantom{\sum} \right] \; d\mu_1(x_1,x_2) &=& \frac{\partial}{\partial x_1} \left[ x_1 \: (1-x_1-x_2) \; d\mu_1(x_1,x_2) \right]\\
\frac{\partial}{\partial x_2} \left[x_2 \: (1-x_1-x_2) \vphantom{\sum} \right] \; d\mu_1(x_1,x_2) &=& \frac{\partial}{\partial x_2} \left[ x_2 \: (1-x_1-x_2) \; d\mu_1(x_1,x_2) \right]\\
\frac{\partial}{\partial x_3} \left[x_3 \: (1-x_3-x_4) \vphantom{\sum} \right] \; d\mu_3(x_3,x_4) &=& \frac{\partial}{\partial x_3} \left[ x_3 \: (1-x_3-x_4) \; d\mu_3(x_3,x_4) \right]\\
\frac{\partial}{\partial x_4} \left[x_4 \: (1-x_3-x_4) \vphantom{\sum} \right] \; d\mu_3(x_3,x_4) &=& \frac{\partial}{\partial x_4} \left[ x_4 \: (1-x_3-x_4) \; d\mu_3(x_3,x_4) \right].
\end{eqnarray*}
We can compute analytically
\[
\vol \: \mbK=\frac{5}{24}\simeq 0.2083.
\]

In Figure \ref{fig:4dprism} we compare the two sparse formulations with and without Stokes constraints. Surprisingly the linear formulation is faster than the branched one for small relaxation degrees{, most probably because at this level of precision the branching costs more in terms of constructing and parsing the LMIs than it saves in computational time}. When going deeper in the hierarchy we see the advantage of the branched formulation where more computations are done in parallel. As observed in the previous example however, the branched formulation seems to have problems to converge to the optimal value on an early relaxation. While the values of both formulations without Stokes constraints almost coincide, the values of the linear formulation with Stokes are strictly better than the ones of the accelerated branched formulation. This further supports our conjecture that formulations where Stokes constraints can be added at every step of the optimization program are to be preferred{: the fact that both configurations behave equally without Stokes constraints and that the branched configuration keeps a relaxation gap when implementing Stokes constraints suggests that these Stokes constraints behave better in linear configurations than in branched configurations. For this reason, in the 6D case where all possible configurations are branched, we could not completely eliminate the relaxation gap, while in this case where there is a linear configuration, the relaxation gap vanishes.}

\begin{figure}[!h]
\begin{subfigure}{0.5\textwidth}
\includegraphics[width=\textwidth]{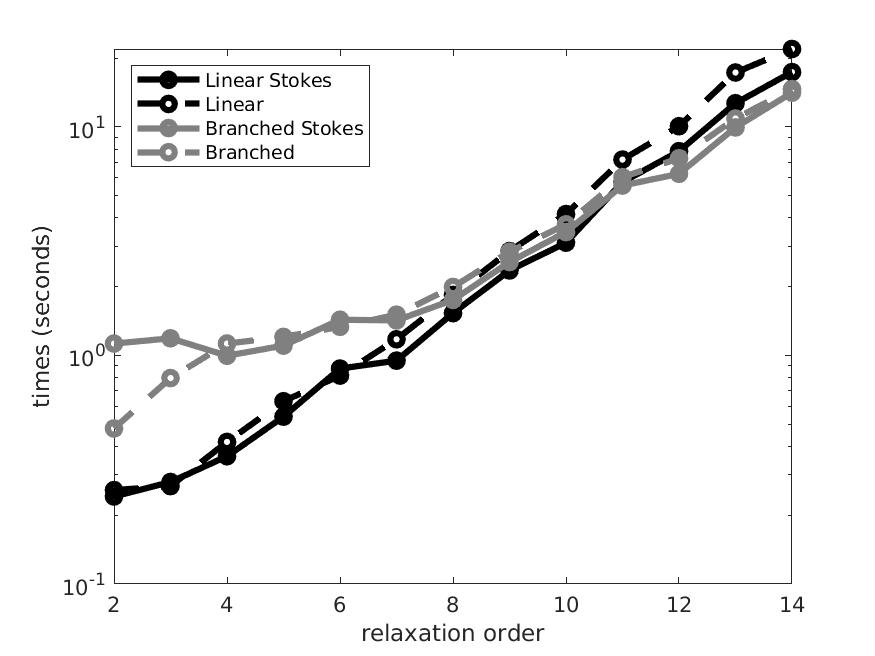}
\caption{Computation time vs relaxation order.}
\label{fig:4dprism_tim}
\end{subfigure}
\hfill
\begin{subfigure}{0.5\textwidth}
\includegraphics[width=\textwidth]{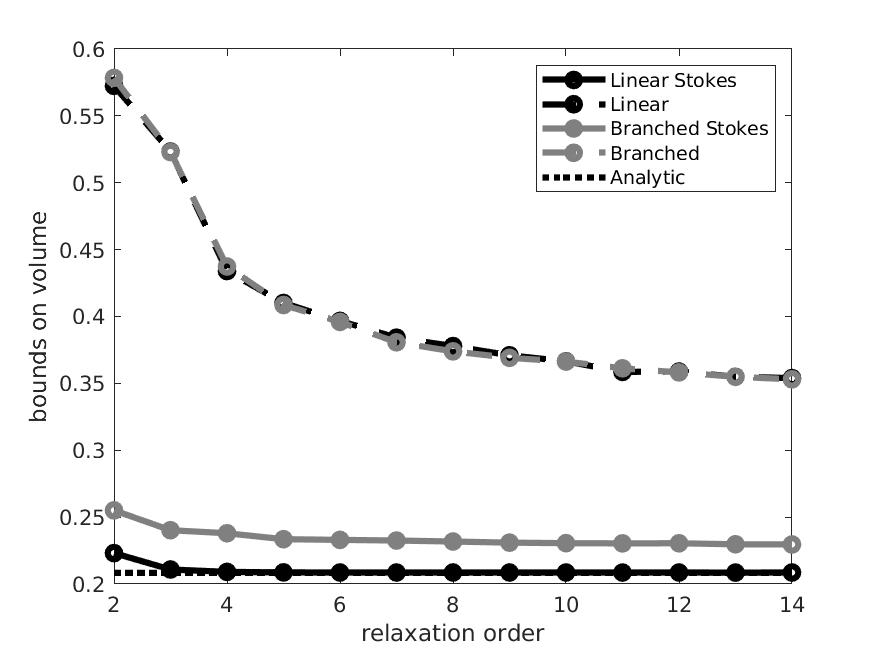}
\caption{Bounds on the volume vs relaxation order.}
\label{fig:4dprism_val}
\end{subfigure}
\caption{Performance for the 4D polytope.}
\label{fig:4dprism}
\end{figure}

\subsection{Discussion on the DIP assumption} \label{secdip}

It may happen that Assumption \ref{asm} does not hold, in which case all the above results would not apply.
For example one could think of the following set:
$$ \mbK := \left\{ \mbx \in \R^6 : x_1^2 + x_2^2 + x_3^2 \leq 1,\: x_i^2 + x_{i+2}^2 + x_{i+3}^2 \leq 1 , \:i=2,3  \right\}$$
whose correlative graph is represented on Figure \ref{counterex}. Here the DIP and CIP cannot be simultaneously enforced: the CIP would only be satisfied by a branched clique tree, but since all the cliques share common vertices, in such a branched tree there would automatically be sibling cliques with nonempty intersection. Also, one can notice that in this case (and, as far as we know, only in similar configurations where Assumption \ref{asm} is violated), the clique $C_2$ does not correspond to a polynomial appearing in the description of $\mbK$.

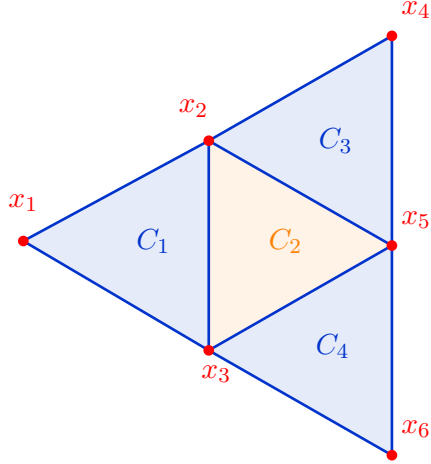
\begin{figure}[!h]
\begin{center}
\begin{tikzpicture}[line cap=round,line join=round,>=triangle 45,x=1.0cm,y=1.0cm]
\clip(-3,-3.5) rectangle (4,3.5);
\fill[line width=1pt,color=qqttcc,fill=qqttcc,fill opacity=0.1] (-2.18,0.06) -- (0.26,1.39) -- (0.26,-1.39) -- cycle;
\fill[line width=1pt,color=qqttcc,fill=qqttcc,fill opacity=0.1] (0.26,1.39) -- (2.67,2.78) -- (2.67,0) -- cycle;
\fill[line width=1pt,color=qqttcc,fill=qqttcc,fill opacity=0.1] (0.26,-1.39) -- (2.67,0) -- (2.67,-2.78) -- cycle;
\fill[line width=1pt,color=qqttcc,fill=orange,fill opacity=0.1] (0.26,-1.39) -- (2.67,0) -- (0.26,1.39) -- cycle;
\draw [line width=1pt,color=qqttcc] (-2.18,0.06)-- (0.26,1.39);
\draw [line width=1pt,color=qqttcc] (0.26,1.39)-- (0.26,-1.39);
\draw [line width=1pt,color=qqttcc] (0.26,-1.39)-- (-2.18,0.06);
\draw [line width=1pt,color=qqttcc] (0.26,1.39)-- (2.67,2.78);
\draw [line width=1pt,color=qqttcc] (2.67,2.78)-- (2.67,0);
\draw [line width=1pt,color=qqttcc] (2.67,0)-- (0.26,1.39);
\draw [line width=1pt,color=qqttcc] (0.26,-1.39)-- (2.67,0);
\draw [line width=1pt,color=qqttcc] (2.67,0)-- (2.67,-2.78);
\draw [line width=1pt,color=qqttcc] (2.67,-2.78)-- (0.26,-1.39);
\fill [color=ffqqqq] (0.26,1.39) circle (2.0pt);
\draw[color=ffqqqq] (0.06,1.82) node {$x_2$};
\fill [color=ffqqqq] (0.26,-1.39) circle (2.0pt);
\draw[color=ffqqqq] (0.36,-1.7) node {$x_3$};
\fill [color=ffqqqq] (-2.18,0.06) circle (2.0pt);
\draw[color=ffqqqq] (-2.18,0.57) node {$x_1$};
\draw[color=qqttcc] (-0.47,0.06) node {$C_1$};
\draw[color=orange] (1.27,0.06) node {$C_2$};
\fill [color=ffqqqq] (2.67,0) circle (2.0pt);
\draw[color=ffqqqq] (2.99,0.34) node {$x_5$};
\fill [color=ffqqqq] (2.67,2.78) circle (2.0pt);
\draw[color=ffqqqq] (2.99,3.13) node {$x_4$};
\fill [color=ffqqqq] (2.67,-2.78) circle (2.0pt);
\draw[color=ffqqqq] (2.99,-2.43) node {$x_6$};
\draw[color=qqttcc] (1.93,1.39) node {$C_3$};
\draw[color=qqttcc] (1.89,-1.33) node {$C_4$};
\end{tikzpicture}
\caption{A correlation graph that violates Assumption \ref{asm}.}
\label{counterex}
\end{center}
\end{figure}

First, we would like to emphasize that the core of this article is the linear computation theorem, for which the working assumption always holds. The branched generalizations are only consequences of this linear computation theorem.

Second, the fact that Assumption \ref{asm} does not hold is not a dead-end for using our scheme. In fact, even the simpler CIP might not hold, in which case one would need to perform a \textit{chordal extension}, which consists of adding virtual links between variables to construct a chordal graph. Basically, a chordal extension would make the graph chordal at the price of slightly weakening the correlative sparsity pattern. In general, the same can be performed to enforce Assumption \ref{asm}: one could add virtual links between variables to enforce the assumption to hold. For example, if one artificially links variables $x_3$ and $x_4$ in the correlation graph of $\mbK$, one obtains a new correlation graph, with an associated clique tree satisfying all our working assumptions (see Figure \ref{fixcounter}). This manipulation results in increasing the correlative sparsity from 3 to 4, which is a weakening of the correlative sparsity pattern. However, our framework still allows to reduce the dimensionality of the problem from 6 to 4.

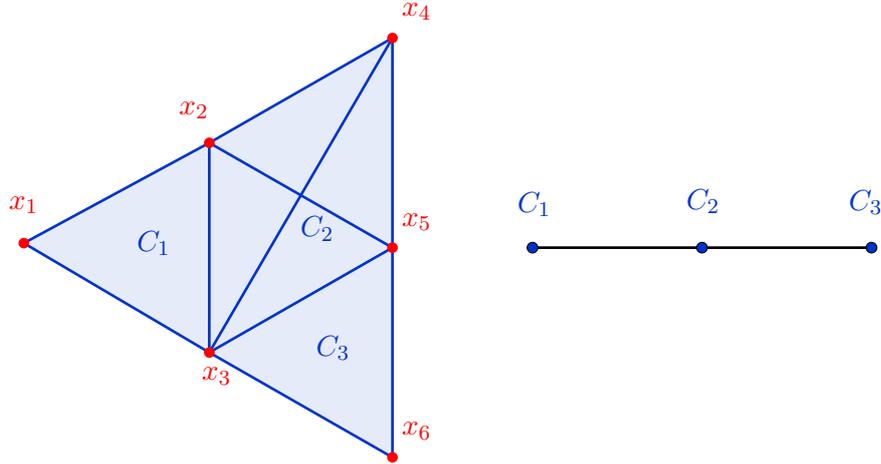
\begin{figure}[!h]
\begin{center}
\begin{tikzpicture}[line cap=round,line join=round,>=triangle 45,x=1.0cm,y=1.0cm]
\clip(-3,-3.5) rectangle (4,3.5);
\fill[line width=1pt,color=qqttcc,fill=qqttcc,fill opacity=0.1] (-2.18,0.06) -- (0.26,1.39) -- (0.26,-1.39) -- cycle;
\fill[line width=1pt,color=qqttcc,fill=qqttcc,fill opacity=0.1] (0.26,1.39) -- (2.67,2.78) -- (2.67,0) -- cycle;
\fill[line width=1pt,color=qqttcc,fill=qqttcc,fill opacity=0.1] (0.26,-1.39) -- (2.67,0) -- (2.67,-2.78) -- cycle;
\fill[color=qqttcc,fill=qqttcc,fill opacity=0.1] (0.26,1.39) -- (2.67,0) -- (0.26,-1.39) -- cycle;
\draw [line width=1pt,color=qqttcc] (-2.18,0.06)-- (0.26,1.39);
\draw [line width=1pt,color=qqttcc] (0.26,1.39)-- (0.26,-1.39);
\draw [line width=1pt,color=qqttcc] (0.26,-1.39)-- (-2.18,0.06);
\draw [line width=1pt,color=qqttcc] (0.26,1.39)-- (2.67,2.78);
\draw [line width=1pt,color=qqttcc] (2.67,2.78)-- (2.67,0);
\draw [line width=1pt,color=qqttcc] (2.67,0)-- (0.26,1.39);
\draw [line width=1pt,color=qqttcc] (0.26,-1.39)-- (2.67,0);
\draw [line width=1pt,color=qqttcc] (2.67,0)-- (2.67,-2.78);
\draw [line width=1pt,color=qqttcc] (2.67,-2.78)-- (0.26,-1.39);
\draw [line width=1pt,color=qqttcc] (0.26,-1.39)-- (2.67,2.78);
\fill [color=ffqqqq] (0.26,1.39) circle (2.0pt);
\draw[color=ffqqqq] (0.06,1.82) node {$x_2$};
\fill [color=ffqqqq] (0.26,-1.39) circle (2.0pt);
\draw[color=ffqqqq] (0.36,-1.7) node {$x_3$};
\fill [color=ffqqqq] (-2.18,0.06) circle (2.0pt);
\draw[color=ffqqqq] (-2.18,0.57) node {$x_1$};
\draw[color=qqttcc] (-0.47,0.06) node {$C_1$};
\fill [color=ffqqqq] (2.67,0) circle (2.0pt);
\draw[color=ffqqqq] (2.99,0.34) node {$x_5$};
\fill [color=ffqqqq] (2.67,2.78) circle (2.0pt);
\draw[color=ffqqqq] (2.99,3.13) node {$x_4$};
\fill [color=ffqqqq] (2.67,-2.78) circle (2.0pt);
\draw[color=ffqqqq] (2.99,-2.43) node {$x_6$};
\draw[color=qqttcc] (1.89,-1.33) node {$C_3$};
\draw[color=qqttcc] (1.68,0.25) node {$C_2$};
\end{tikzpicture}
\begin{tikzpicture}[line cap=round,line join=round,>=triangle 45,x=1.0cm,y=1.0cm]
\clip(-1.5,-3.5) rectangle (3.7,1.);
\draw [color=black,line width=1.pt] (-1.115168148756052,0.)-- (1.115168148756052,0.);
\draw [color=black,line width=1.pt] (1.115168148756052,0.)-- (3.3455044462681562,0.);
\draw [fill=qqttcc] (-1.115168148756052,0.) circle (2.0pt);
\draw[color=qqttcc] (-1.09,0.58) node {$C_1$};
\draw [fill=qqttcc] (1.115168148756052,0.) circle (2.0pt);
\draw[color=qqttcc] (1.13,0.62) node {$C_2$};
\draw [fill=qqttcc] (3.3455044462681562,0.) circle (2.0pt);
\draw[color=qqttcc] (3.27,0.6) node {$C_3$};
\end{tikzpicture}
\caption{A way to fix our counterexample.}
\label{fixcounter}
\end{center}
\end{figure}

\begin{rqe}
There might exist examples for which the DIP would only be obtained by completely destroying the sparsity pattern one wants to exploit. However, this could also happen with the more common CIP. In either case, an option might be to consider the dual problem of \eqref{general-lp-dual} as a way to find minimizing sequences $(v_k)_k$ that approximate $\ind_\mbK$, and to apply it to each one of the $\mbU_i := \{\mbx_i \in \mbX_i : \mbg_i(\mbx_i) \geq 0\}$ to obtain a sequence $\left(v_1^{(k)},\dots,v_m^{(k)}\right)_k$ such that $\left(v_i^{(k)}\right)_k$ approximates $\ind_{\mbU_i}$. Then, one would still have to prove that the convergence of the hierarchy is stable by product (which is nontrivial) to conclude that $\displaystyle\int \prod\limits_{i=1}^m v_i^{(k)}\circ\pi_{\mbX_i}(\mbx) \; d\mbx$ converges to $\vol \ \mbK$. The major drawback of this solution is that in order for $(v_k)_k$ to approximate $\ind_\mbK$, we cannot implement Stokes constraints, as they would modify problem \eqref{general-lp-dual} (see Remark \ref{rqstokes}) in a way that makes us lose the convergence to an indicator (this is precisely the point of Stokes constraints: they allow to obtain the volume without trying to approximate discontinuous functions with polynomials). However, we know that the volume approximation hierarchy has a bad convergence rate without Stokes constraints. In general, we should not expect any method that approximates indicators with polynomials to yield satisfactory results in terms of volume computation. Such a method should be considered only as a last resort if any trial to apply the above scheme fails. Finally, this option represents a framework that would be completely independent from the above and thus it remains out of the scope of this paper.
\end{rqe}

\color{black}

\section*{Conclusion}

\subsection*{Our results}

In this paper we addressed the problem of approximating the volume of sparse semi-algebraic sets with the Moment-SOS hierarchy of SDP relaxations.
As illustrated by our examples, our sparse formulation allows to dramatically decrease the computational time for each relaxation, and to tackle high dimensional volume computation problems that are not tractable with the usual SDP methods. By splitting the problems into low dimensional subproblems, one drastically reduces the dimension of each relaxation, without loss of precision. This reduction of complexity is due to the correspondance between the structure of our algorithm and the {correlative} sparsity pattern in the description of the semi-algebraic set.

We also showed that additional Stokes constraints have a huge effect on convergence and precision for volume computation, and that they can successfully be adapted to our sparse formulations. This yields a much better rate of convergence for the corresponding hierarchy. However, implementing these Stokes constraints leads to subtle constraints that have to be enforced if one wants to efficiently compute the volume:
\begin{itemize}
\item First, one should always prefer the linear formulation of  Theorem \ref{linco} whenever possible, since this ensures that Stokes constraints can always be efficiently implemented.
\item Then, in the more general case of Theorem \ref{main}, one should always avoid formulations in which the root at the computation tree has no Stokes constraint; fortunately, such configurations can always be avoided by chosing a leaf as the root of the clique tree.
\end{itemize}

Furthermore, in the branched case, one should be aware of the fact that each step of the algorithm introduces an approximation error, and the errors accumulate until the root is reached. Consequently, a formulation in which the clique tree has too many generations will lead to a larger global error than a formulation with less generations. For this reason, one should minimize the number of generations in the clique tree, which is equivalent to parallelizing as much as possible. In addition to that, when the problem has many dimensions and branches, parallelization can obviously drastically increase the speed of the computations.

\subsection*{Applications and future work}

To the best of our knowledge, this sparse method for solving volume problems is new and full of promises for future applications. For instance, the problem of computing the mass of any compactly supported measure absolutely continuous (with respect to the Lebesgue measure) can be adressed using this sparsity method. Also, measures that are not compactly supported but have some decay properties (e.g. Gaussian measures) can also be handled by our method, which may prove useful in computations for probability and statistics. Also, specific constraints could probably be used in addition to Stokes constraints when the semi-algebraic set presents a specific structure (e.g. a polytope, a convex body).

Furthermore, the framework of exploiting {correlative} sparsity can be applied to any method that relies on computations on measures, whether these measures are represented by their moments (as it is done in this paper), or by samples (as in the stochastic volume computation methods). In particular, we believe that this formalism could easily be extended to Monte-Carlo-based volume computations.

Finally, we also believe that this method can be adapted to the computation of regions of attraction, through the formalism developed in \cite{roa}, for high dimensional differential-algebraic systems that present a network structure, such as power grids, distribution networks in general and possibly other problems. The main difficulty resides in taking sparsity into account when formulating the Liouville equation, and keeping uniqueness of the solution as in the non-controlled non-sparse framework.

\section*{Acknowledgements}

We are grateful to an anonymous Reviewer for a deep insight that helped us significantly improve the initial version of this paper.
\color{black}

\appendix

%
%
%
%

\section{Disjoint Intersection Property}

\label{DIP}

\begin{defn}
Let $G = (V,E)$ be a graph with vertices set $V$ and edges set $E \subset V^2$. The following definitions can be found in e.g. \cite{graphs}:
\begin{itemize}
\item  The \textit{degree} $\deg v$ of $v \in V$ is the cardial of the set $\{w \in V : (v,w) \in E\}$ \textit{i.e.} the number of vertices connected to $v$. \color{black}
\item A \emph{clique} of $G$ is a subset of vertices $C \subset V$ such that $u,v \in  C$ implies $(u,v) \in E$.
\item A graph $G$ is \emph{chordal} if every cycle of length greater than 3 has a chord, i.e. an edge connecting two nonconsecutive vertices on the cycle.
\item A \emph{tree} $T=(\mcK,\mcE)$ is a graph without cycle.
\item {The \textit{treewidth} of a chordal graph is the size of its biggest clique minus 1. Thus the treewidth of a tree is 1 and the treewidth of a complete graph ($E = V^2$) of size $n$ is $n-1$.}
\item A \textit{rooted tree} is a tree in which one vertex has been designated the \textit{root}.
\item In a rooted tree, the \textit{parent} of a vertex $v$ is the vertex $w$ connected to it on the path to the root; $v$ is then called a \textit{child} of $w$; two vertices that have the same parent are called \textit{siblings}; a \textit{descendant} of a vertex $v$ is any vertex which is either the child of $v$ or is (recursively) the descendant of any of the children of $v$; $v$ is then called an \textit{ancestor} of itself and any of its descendants.
\item The vertices of a rooted tree can be partitioned between the root, the \textit{leaves} (the vertices that have parents but no children) and the \textit{branches} (that have children and parents).
\item Let $\mcK$ be the set of {maximal} cliques of $G$. A \emph{clique tree} $T = (\mcK,\mcE)$ of $G$ is a tree whose vertices are the {maximal} cliques of $G$.
\item A clique tree satisfies the clique intersection property (CIP) if for every pair of distinct cliques $C,C'\in\mcK$, the set $C \cap C'$ is contained in every clique on the path connecting $C$ and $C'$ in the tree. We denote by $\mcT^{ct}$ the set of clique trees of $G$ that satisfy the CIP.
\item {Let $A \subset V$. Then, the \textit{subgraph of $G$ generated by} $A$ is given by $ \langle A \rangle_G := (A,E \cap A^2)$.}
\end{itemize} 
\end{defn}

\begin{thm}
A connected graph $G$ is chordal if and only if $\mcT^{ct} \neq \varnothing$ if and only if $\mcK$ admits an ordering that satisfies the RIP.
\end{thm}

\begin{defn}
Let $G = (V,E)$ be chordal {and connected}.
Let $T = (\mcK,\mcE) \in \mcT^{ct}$ be a clique tree rooted in $C_1 \in \mcK$.
$T$ satisfies the \emph{Disjoint Intersection Property} (DIP) if $\forall C,C',C'' \in \mcK$, 
if $(C,C') \in \mcE$ and 
$(C,C'') \in \mcE$ then $C' = C''$ or $C' \cap C'' = \varnothing$.
In words, each clique has 
an empty intersection with all its siblings.
\end{defn}

We are now going to give a systematic way to enforce Assumption \ref{asm} and generate the associated clique trees. Let $(\mbg_1,\dots,\mbg_m)$ be a correlatively sparse family of polynomial vectors with a connected chordal correlation graph $G = (V,E)$. Let $\mcK$ be the set of maximal cliques of $G$. We construct the \textit{clique graph} $G_\mcK = (\mcK,\mcF)$ such that $(C,C') \in \mcF$ \textit{iff} $C \cap C' \neq \varnothing$. One can in turn define cliques (called \textit{meta-cliques}) for this new graph, and its correlative sparsity $CS_\mcK$ is the size of its biggest maximal meta-clique minus $1$. One can note that any clique tree is a subtree of $G_\mcK$ including all its vertices.

\begin{rqe}
If $G_\mcK$ itself is a tree (as in section \ref{ex1}), then it trivially satisfies the DIP and CIP, and Assumption \ref{asm} automatically holds.
\end{rqe}

%

\begin{lem} \label{ancestor} Let $T = (\mcK,\mcE)$ be a clique tree satisfying Assumption \ref{asm}.
\begin{itemize}
\item[1)] Let $C,C' \in \mcK$ such that $C \neq C'$ and $C \cap C' \neq \varnothing$. Then, up to permuting them, $C$ is a descendant of $C'$.
\item[2)] Let $K$ be a meta-clique. Then, $\langle K \rangle_{T}$ is an oriented path of $T$: the elements of $K$ are ancestor to one another and each $C \in K$ has its parent in $K$ except one of them.
\end{itemize}
\end{lem}
\begin{proof}
\begin{itemize}
\item[]
\item[1)] Let $C''$ be the last common ancestor of $C$ and $C'$, meaning that $C''$ is an ancestor of both $C$ and $C'$ but any child of $C''$ is the ancestor of at most one of them. Such ancestor exists since the root $C_1$ is a common ancestor to $C$ and $C'$. Then, $C''$ is on the path between $C$ and $C'$. Since $C \neq C'$, up to permuting them, we can suppose that $C \neq C''$.

By contradiction, we suppose that $C' \neq C''$. Then, let $\hat{C}$ be the child of $C''$ that is also the ancestor of $C$, and $\tilde{C}$ the child of $C''$ that is also the ancestor of $C'$. Both exist since $C''$ is an ancestor of $C$ and $C'$ and $C \neq C' \neq C''$. $C''$ being the latest common ancestor of $C$ and $C'$, we deduce that $\hat{C} \neq \tilde{C}$ so that $\hat{C}$ and $\tilde{C}$ are siblings. Then, the DIP ensures that $\hat{C} \cap \tilde{C} = \varnothing$. However, $\hat{C}$ and $\tilde{C}$ are on the path between $C$ and $C'$, so according to the CIP they both contain $C \cap C'$ which is nonempty. This is a contradiction.
\item[2)] According to point 1), all elements of $K$ are descendants of one another, so that they are all on the same path in $T$. We only have to show that any $C$ between two elements $C',C''$ of $K$ on this path is also an element of $K$. Indeed, let $C''' \in K$. Then, up to a permutation on $\{C',C'',C'''\}$ we can suppose that the unoriented path includes in this order: $(C',C,C'',C''')$. Then, $C$ is on the path between $C'$ and $C'''$, so the CIP implies that $C \supset C' \cap C'''$ is nonempty (since $C'$ and $C'''$ belong to the same clique $K$), and then $C$ has a nonempty intersection with $C'''$. This shows that $C$ has a nonempty intersection with any element of $K$, which by maximality of $K$ is the definition of $C \in K$.
\end{itemize}
\end{proof}

\begin{cor} If $G_\mcK$ is a complete graph (all pairs of maximal cliques have nonempty intersection as in section \ref{ex2}), then the only candidates for our clique tree are linear clique trees. In such case, Assumption \ref{asm} is equivalent to the existence of a reordering of $(\mbg_1,\dots,\mbg_m)$ such that Assumption \ref{ass-1b} holds.
\end{cor}

We now give an alogrithm to generate a clique tree $T = (\mcK,\mcE)$ that satisfies the DIP and the CIP. In case Assumption \ref{asm} does not hold, this algorithm automatically adds edges to $E$ until it finds an appropriate clique tree (see Algorithm \ref{algo}).

\begin{algorithm}
\SetAlgoLined
\caption{How to build an appropriate clique tree}
\label{algo}
\KwData{$G = (V,E)$ and its clique graph $G_\mcK = (\mcK,\mcF)$.}
\KwResult{$T = (\mcK,\mcE)$ satisfying CIP \& DIP if Assumption \ref{asm} holds, $G = (V,E)$ with additional edges else.}

Initialization: Choose $C_1 \in \mcK$ with minimal degree in $G_\mcK$\;
\qquad Initialize $i = j = k = 1$, $\mcP_1 := \{C_1\}$, $\mcE_1 := \varnothing$\;
\While {$k < |\mcK|$\label{whilk}}{
	\While {$\mathbb{M}_{ik} := \{K \text{\emph{ maximal meta-clique }} : K \cap \mcP_k^c \neq \varnothing, C_i \in K \} \neq \varnothing$\label{whili}}{
		Choose $K_j \in \underset{K \in \mathbb{M}_{ik}}{\argmax}|K|$\label{optclik}\;
		\While {$K_j \cap \mcP_k^c \neq \varnothing$\label{whilj}}{
		$l := \max \{r \leq k : C_r \in K_j\}$\;
			\eIf{$\exists (C, C') \in \argmax\{|\hat{C} \cap C_l| : (\tilde{C} , \hat{C}) \in \mcP_k \times K_j\cap\mcP_ k^c, \tilde{C} \cap \hat{C} \nsubseteq C_l\}$\label{ifcip}}{
				\eIf{$|C_l \cup C| > |C_l \cup C'|$\label{optcip}}{
					\lForEach{$v \in C',w \in C_l$}{$E \leftarrow E \cup \{(v,w),(w,v)\}$}
					\Return $G = (V,E)$ \;
					}{
					\lForEach{$v \in C,w \in C_l$}{$E \leftarrow E \cup \{(v,w),(w,v)\}$}
					\Return $G = (V,E)$ \;
				}
				}{Choose $C_{k+1} \in \underset{C \in K_j \cap \mcP_k^c}{\argmax}|C \cap C_l|$\label{optint}
			}
			\eIf{$\exists C \in \mcP_k$ s.t. $(C_l,C) \in \mcE_k$ \& $C \cap C_{k+1} \neq \varnothing$ \label{ifdip}}{
				\uIf{$|C \cup C_{k+1} < |C_l \cup C_{k+1}| \wedge |C_l \cup C|$\label{optdip}}{
					\lForEach{$v \in C,w \in C_{k+1}$}{$E \leftarrow E \cup \{(v,w),(w,v)\}$}
					\Return $G = (V,E)$ \;
					}
					\uElseIf{$|C_l \cup C| > |C_l \cup C_{k+1}|$}{
					\lForEach{$v \in C_{k+1},w \in C_l$}{$E \leftarrow E \cup \{(v,w),(w,v)\}$}
					\Return $G = (V,E)$ \;
					}
					\Else{
					\lForEach{$v \in C,w \in C_l$}{$E \leftarrow E \cup \{(v,w),(w,v)\}$}
					\Return $G = (V,E)$ \;
				}
				}{
				$\mcP_{k+1} := \mcP_k \cup \{C_{k+1}\}$\;
				$\mcE_{k+1} := \mcE_k \cup \{(C_l,C_{k+1})\}$\;
				$k \leftarrow k+1$ \;
			}
		}
		$j \leftarrow j+1$ \;
	}
	$i \leftarrow i+1$ \;
}
\Return $T := (\mcP_{|\mcK|},\mcE_{|\mcK|})$\;
\end{algorithm}

\begin{rqe}Algorithm \ref{algo} deserves some explanations:
\begin{itemize}
\item Minimizing $\deg C_1$ is a way to ensure Stokes constraints will be fully implementable, in contrast to the 2 step implementation in section \ref{ex1}.
\item Index $i$ denotes a clique that has already been added to the tree; the algorithm adds to the tree every clique that shares vertices with $C_i$, and then increments $i$.
\item Index $j$ denotes the meta-clique in which we are working; according to Lemma \ref{ancestor}, in such meta-clique the cliques should be added in line.
\item Index $k$ denotes the number of elements that have already been added to the tree; when $k$ is equal to the number of cliques, our clique tree is complete.
\item Index $l$ denotes the lates clique of the meta-clique $K_j$ that has been added to the tree; according to Lemma \ref{ancestor}, $C_l$ should then be the parent of $C_{k+1}$.
\item At line \ref{optclik} we maximize $|K_j|$ to favor linear configurations as they are the most compatible with Stokes constraints.
\item The if loop at line \ref{ifcip} checks whether it is possible to add the remaining cliques of $K_j$ to our tree without destroying the CIP.
\item The if loop at line \ref{ifdip} checks whether the clique we want to add destroys the DIP or not.
\item At line \ref{optint} we maximize $|C_{k+1} \cap C_l|$ so that it is less likely to pose problems with CIP and DIP in the future iterations.
\item The if loops at lines \ref{optcip} and \ref{optdip} are meant to minimize the correlative sparsity of the new graph $G = (V,E)$, since it is the limiting factor for the tractability of our algorithm.
\end{itemize}
\end{rqe}

\begin{thm}
Any clique tree returned by Algorithm \ref{algo} satisfies the DIP and the CIP.
\end{thm}
\begin{proof}
We are going to show by induction that at any step $k$, the graph $T_k := (\mcP_k,\mcE_k)$ is a tree that satisfies the CIP and the DIP. First, it is trivial that $T_1 = (\{C_1\},\varnothing)$ is a tree and satisfies the CIP and the DIP. Next, we suppose that we have constructed a tree $T_k$ satisfying CIP and DIP through iterations of our algorithm, and that the next iteration leads us to define a $T_{k+1} := (\mcP_k \cup \{C_{k+1}\}, \mcE_k \cup \{(C_l,C_{k+1}\})$. Since the only edge we added connected $C_l$ to a new vertex that was not in $T_k$, it did not introduce any cycle, thus $T_{k+1}$ is still a tree. We are now going to check the CIP and DIP.
\begin{itemize}
\item Let $C,C''\in\mcP_{k+1} ; C \cap C'' \neq \varnothing$. Let $C' \in \mcP_k$ be on the path between $C$ and $C''$ in $T_{k+1}$ ($C' \neq C_{k+1}$ because $C_{k+1}$ is on no path in $T_{k+1}$).
\begin{itemize}
\item If $C,C'' \in \mcP_k$ then by our induction hypothesis $C \cap C'' \subset C'$.
\item Else, without loss of generality we have $C'' = C_{k+1}$ and $C \in \mcP_k$.
\begin{itemize}
\item Since we successfully passed through the if loop of line \ref{ifcip}, we have $C \cap C_{k+1} \subset C_l$.
\item By our induction assumption ($T_k$ satisfies the CIP), we have $C \cap C_l \subset C'$ (because either $C' = C_l$ or $C'$ is on the path between $C$ and $C_l$, the parent of $C_{k+1}$).
\end{itemize}
This yields $C' \supset C \cap C_l \supset C \cap (C \cap C_{k+1}) = C \cap C_{k+1} = C \cap C''$.
\end{itemize}
Then, $T_{k+1}$ satisfies the CIP.
\item Let $C \in \mcP_k,C',C''\in\mcP_{k+1}$ such that $(C,C'),(C,C'')\in\mcE_{k+1}$.
\begin{itemize}
\item If $C',C'' \in \mcP_k$ then by our induction hypothesis $C' \cap C'' = \varnothing$.
\item Else, without loss of generality we have $C'' = C_{k+1}, C=C_l$, and since we successfully passed through the if loop of line \ref{ifdip}, we have $C' \cap C'' = C' \cap C_{k+1} = \varnothing$.
\end{itemize}
Then, $T_{k+1}$ satisfies the DIP.
\end{itemize}
\end{proof}

Finally, we conjecture that if Assumption \ref{asm} holds, then Algorithm \ref{algo} will directly return a clique tree satisfying the CIP and the DIP without adding any edge to $G$.

\section{On Monte Carlo simulations}
\label{app:MC}
We describe the very basic Monte Carlo approach used in section \ref{sec:highnonconvex}. Let $\xi_1,\ldots,\xi_N$ be i.i.d. samples from some law $\mu$. In our case $\mu$ is the uniform distribution on $[0,1]^n$. Further let $f$ be a function from the probability space into $\{0,1\}$. Again, in our case, $f$ would return $1$ if the sample $\xi_i$ is in the set $\mbK$, and $0$ else. By the strong law of large numbers
\[
X_N:= \frac{1}{N}\sum_{i=1}^N f(\xi_i) \to \int f d \mu \text{ for } N \to \infty .
\]
This makes $X_N$ a reasonable guess if $N$ is large. However, $X_N$ is still a guess, and it might happen  that $X_N$ is actually far away from the approximated quantity.

As a consequence of the Central Limit Theorem, the difference $X_N-\int f d \mu$ behaves (almost) like a normal distributed random variable with zero mean and variance $\sigma^2/N$ where $\sigma^2 = \int (f -\int f d \mu)^2 d \mu$. Note that the variance $\sigma^2$ can also be estimated based on the i.i.d. samples $\xi_1,\ldots,\xi_N$:
\[
S_N^2 := \frac{1}{N-1}\sum_{i=1}^N(\xi_i - X_N)^2.
\]
This allows to define a confidence interval for the approximated volume. Indeed, say we are interested in a 99\%-confidence interval. Then for $G$ a standard normal distributed random variable, we have $P(|G|< 2.58) \approx 0.99$ and consequently,
\[
P\left(X_N -\frac{2.58 S_N}{\sqrt{N}} \leq \int f d \mu \leq X_N +\frac{2.58 S_N}{\sqrt{N}}\right) \approx 0.99.
\]

\color{black}

\end{document}